\documentclass[a4paper,11pt]{amsart}
 
\usepackage[hidelinks]{hyperref}
\usepackage{hyperref}

\usepackage[table]{xcolor}
 
\usepackage{amsthm,amssymb,latexsym,amsfonts, stmaryrd}
\usepackage[latin1]{inputenc}
\usepackage[english,activeacute]{babel}
\usepackage{mathrsfs}
\usepackage{booktabs}
\usepackage{hyperref}
\usepackage{multirow}
\usepackage{graphicx}
\usepackage{tikz}
\usetikzlibrary{spy, shapes.geometric, 3d, shadows} 
\usepackage{tikz-3dplot}
\usepackage{pgfplots}
\pgfplotsset{compat=1.17}
 
\usepackage{caption}
\usepackage{array}
\usepackage{tabularx}
 
\usepackage{color,amssymb}
\usepackage{esint}
\usepackage{enumerate}
\usepackage{csquotes}
\usepackage{algorithm}
\usepackage{algorithmic}

\usepackage{parskip}
\usepackage{empheq}
\usepackage{hyperref}
\usepackage{cleveref}
 
\usepackage{cancel}
\usepackage[normalem]{ulem}
\definecolor{coralred}{rgb}{1.0, 0.25, 0.25}

\DeclareMathOperator{\vinf}{inf}
 
\newcolumntype{Y}{>{\centering\arraybackslash}X}
\newcolumntype{Z}{>{\centering\arraybackslash}p{0.15\hsize}}
\newcolumntype{A}{>{\centering\arraybackslash}p{0.05\hsize}}
\newcolumntype{B}{>{\centering\arraybackslash}p{0.175\hsize}}
 
\allowdisplaybreaks
 
\newcommand{\eps}{\varepsilon}

\newcommand{\dr}{{\rm{d}}_r}
\def\build#1_#2^#3{\mathrel{\mathop{\kern 0pt#1}\limits_{#2}^{#3}}}
 
\makeindex
 
\numberwithin{equation}{section}
 
\theoremstyle{plain}
\newtheorem{thm}{Theorem}[section]

\newtheorem{conj}[thm]{Conjecture}

\newtheorem{lemma}[thm]{Lemma}
\newtheorem{proposition}[thm]{Proposition}
\newtheorem{cor}[thm]{Corollary}
\newtheorem{remark}[thm]{Remark}
 
\theoremstyle{definition}
\newtheorem{defn}{Definition}
 
\newtheorem{assu}{Assumption}

\theoremstyle{plain}

\title[Freely floating body on a 3D fluid]{Freely floating cylinder on a 3D fluid governed by the Boussinesq equations in the axisymmetric without swirl case}
 
\author[G. Beck]{Geoffrey Beck$^\ast$}
\address{Univ Rennes, IRMAR UMR 6625 \& Centre Inria de l'Universit\'e de Rennes (MINGuS) \& ENS Rennes, France}
\email{$^{*}$Corresponding author: \texttt{geoffrey.a.beck@inria.fr}} 
\author[E. Contentin]{Ewan Contentin}
\address{Univ Rennes, IRMAR UMR 6625 \& ENS de Physique de Lyon, France}
\email{ewan.contentin@univ-rennes.fr}
\author[L. Martaud]{Ludovic Martaud}
\address{Univ Rennes, IRMAR UMR 6625 \& Centre Inria de l'Universit\'e de Rennes (MINGuS) \& ENS Rennes, France}
\email{ludovic.martaud@inria.fr} 
 
\begin{document}
 
\begin{abstract}

This paper deals with the interactions of waves governed by a non-linear dispersive Boussinesq type system with the vertical displacement of a cylindrical floating structure in an axisymmetric without swirl situation. The Boussinesq regime is a good approximation of free surface Euler's equations when the non-linear parameter and the shallowness parameter are small. The vertical motion of the floating body is governed by the Newton equation.
The full coupled wave-structure interaction problem under consideration is reduced to a boundary problem. The boundary condition satisfied by the discharge is given in terms of the vertical displacement of the floating cylinder. The latter is calculated using an ODE, which requires the knowledge of the trace of the surface elevation and its second-time derivative. We use the dispersion in order to exhibit a hidden second order ODE on the trace of the surface elevation. Finally it allows us to rewrite the waves-structure interaction problem as a system of non-local conservative PDEs plus bounded radial terms with a dispersive boundary layer, combined with an ODE at the boundary. This is what we call the Augmented formulation. Afterwards we show that this formulation is well-posed with two different methods.  In our proof, we have tracked down the explicit dependence on the shallowness parameter, which is the small parameter representing the dispersion. The first method gives a continuous solution with a small existence time, i.e. proportional to the product of the inverse of the non-linear parameter and the square of the shallowness parameter,  whereas the second one requires higher regularity of the solution but with a larger existence time, i.e. proportional to the product of the inverse of the non-linear parameter and the shallowness parameter. Finally, we study the return to equilibrium situation in the linear regime. In particular we improve previous results on the explicit time decay. We show that the center mass of the floating body cannot converge to its equilibrium faster than $\mathcal{O}(t^{-1/2})$ in 2D without viscosity and faster than $\mathcal{O}(t^{-3/2})$ with viscosity.

%
%

\end{abstract}
 
\maketitle
 
\smallskip\noindent\textbf{Keywords:}
Fluid-structure interaction, Dispersive perturbation of initial boundary value hyperbolic problems, Wave Energy Converters, Decay test.

\tableofcontents
 
\section*{Introduction and motivations}

We can mention four types of nearshore wave energy facilities: the bottom-hinged flaps which convert oscillations into energy, the oscillating water columns which convert the variations of the pressure in a chamber into electric current (see \cite{Bocchi3}) and floating buoys which convert the translation or rotation induced by waves into electricity (see \cite{BeckLannes,Bocchi,Lannes_float,Maity}). Here we are interested in the latter type: floating structures. Understanding floating structures models is also a keystone for iceberg and sea-iced models.
While the case of fully immersed bodies in a fluid has been studied for years \cite{Glass-Sueur-Takahashi}, the case of partially immersed bodies in an incompressible perfect fluid such as water with a free surface is rather new for mathematicians, even if this topic has been well-studied in engineering and oceanographic communities. This problem is particularly difficult to tackle, in part because several free surfaces are involved: the surface of the sea and the contact surface between the structure and the fluid. Usual CFD methods such as Reynolds Averaged Navier-Stokes are not able to account for the motion of one freely floating object for long times. Some effort have been done in this direction, but solving a boundary integral on the fluctuating contact surface requires lengthy computations. This is why simpler models are considered in engineering.

 In this paper we improve the 1D results with a greater precision in the description of the decay than in \cite{BeckLannes}, and we extend those results to the 2D configuration. In particular we show that it is not possible to obtain exponential decay at infinity, even by adding viscosity.

In the linear regime, the wave-structure can be described by John's model \cite{John1}.  This model can be derived from the Euler equations if we additionally assume that the motion of the body is of small amplitude and the time variation of the contact surface is neglected. This situation for 2D perfect fluid is studied from a mathematical point of view in \cite{lannes2025fjohnmodelcummins}. In particular, the authors introduce appropriate functional framework to deal with this model and show its well-posedness. Even if John's model misses the evolution of the contact surface, it still is the most widely used model for floating structures since numerical simulations can easily be performed (software WAMIT \cite{Wamit}). 
 
Unfortunately, the linear regime is not necessarily a good regime when we consider wave energy converters located near the shore. This is why we are interested in considering non-linear terms. In all the literature mentioned before, one difficulty came from the computation of the pressure exerted by the fluid on the body.  In \cite{Lannes_float}, the pressure is seen as a Lagrange multiplier associated to the constraint that the surface of the fluid coincides with the contact surface. In this work, a system of equations that describes the fluid-floating structure system was established avoiding the assumptions in John's model. Moreover this system exposes an added-mass effect. When a solid moves in a fluid, not only must it accelerate its own mass but also the mass of the fluid around it. This is what we call added-mass phenomenon. Indeed, some components of the hydrodynamic force applied on the body act as inertia in Newton's law. The added-mass can actually be larger than the proper mass of the solid, a fact that has been noticed in ocean engineering \cite{Zhou}. In the case of totally immersed bodies, the added-mass effect plays a crucial role to show well-posedness \cite{Glass-Sueur-Takahashi} and stability of numerical schemes \cite{Gerbeau}. Thus it is not surprising to see that it also plays a crucial role in the case of a partially immersed body. 
 
Many reduced models are used to describe the evolution of waves at the surface of a 2D incompressible perfect fluid in shallow water (see \cite{LannesModeling}). The most famous one is given by the so-called non-linear shallow-water equation (or Saint-Venant equation). It can be derived from the full free surface Euler equation when the shallowness parameter (the ratio between the depth of water at rest and the typical horizontal scale of the wave) tends to zero. However it omits the dispersive effects that play an important role in coastal areas. In order to account for these effects one can consider more precise models, the most simple ones being the so-called Boussinesq models. Indeed there are actually many asymptotically equivalent Boussinesq models, among which we choose the one whose structure is a dispersive perturbation of the non-linear shallow-water equation. Dispersive perturbations of hyperbolic systems with boundary conditions have been studied from a mathematical \cite{LannesBressanone} and numerical \cite{LannesWeynans} point of view. In order to obtain regular solutions in boundary hyperbolic systems one needs compatibility conditions of the initial conditions at the boundary. Such compatibility conditions are sometimes not needed when we consider dispersive perturbations of hyperbolic systems. Indeed, the dispersion creates a boundary layer that regularizes the solution. The Boussinesq system can be written as a conservative law with a non-local flux and a certain source term that exponentially decreases (in space) far from the boundary. This source term is the effect of the dispersive boundary layer. In \cite{LannesWeynans,BeckLannesWeynans}, this non-local conservative law formulation was used to perform numerical simulations when the surface elevation is known at the boundary. In the case of a floating body, the boundary conditions are related to the motion of the body forced by the fluid. In \cite{BeckLannes,BeckLannesWeynans}, the full coupled system can be reduced to a transmission problem for the Boussinesq equations with transmission conditions given in terms of the vertical displacement of the object and of the average horizontal discharge beneath it. These two quantities are in turn determined by two nonlinear ODEs with forcing terms coming from the exterior wave-field. We manage to use the same strategy for the 2D situation under the assumption that the flow is axisymmetric without swirl and we show some well-posedness results.
 
\begin{figure}[h]
\begin{picture}(200,100)
    \tdplotsetmaincoords{75}{110} 
        \put(100,50){\makebox(0,0){         

\begin{tikzpicture}[scale=2, tdplot_main_coords] 

	\draw[thick] (0, 0.2, -0.5) -- (4, 0.2, -0.5);
	\draw[thick] (4, 0.2, -0.5) -- (4, 2.3, -0.5);
	\draw[thick] (4, 2.3, -0.5) -- (0, 2.3, -0.5);
	\draw[thick] (0, 2.3, -0.5) -- (0, 0.2, -0.5);
        
    \def\xcenter{2.5}
    \def\ycenter{1.25}
    \def\amplitude{0.1}
    \def\sigma{0.05}

    \foreach \x in {0.5,0.6,...,4}
        \foreach \y in {0,0.1,...,2.5}
            \draw (\x,\y,{\amplitude * sin(deg(sqrt((\x-\xcenter)^2 + (\y-\ycenter)^2))/(2*\sigma))}) -- ++(0,0,0.01);

    \foreach \x in {0.5,0.6,...,4}
        \foreach \y in {0,0.1,...,2.5}
            \draw (\x,\y,{\amplitude * sin(deg(sqrt((\x-\xcenter)^2 + (\y-\ycenter)^2))/(2*\sigma))}) -- (\x+0.1,\y,{\amplitude * sin(deg(sqrt((\x+0.1-\xcenter)^2 + (\y-\ycenter)^2))/(2*\sigma))});

    \foreach \x in {0.5,0.6,...,4}
        \foreach \y in {0,0.1,...,2.5}
            \draw (\x,\y,{\amplitude * sin(deg(sqrt((\x-\xcenter)^2 + (\y-\ycenter)^2))/(2*\sigma))}) -- (\x,\y+0.1,{\amplitude * sin(deg(sqrt((\x-\xcenter)^2 + (\y + 0.1 - \ycenter)^2))/(2*\sigma))});

    \node[dashed,
    line width=1.5pt,
        cylinder, 
        draw = black, 
        aspect = 1, 
        minimum width = 0.75cm,
        minimum height = 0.5cm,
        shape border rotate = 90] (c) at (2.5,1.25,-0.055) { };

    \node[cylinder, 
    line width=1.5pt,
        draw = black, 
        text = purple,
        cylinder uses custom fill, 
        cylinder body fill = white, 
        cylinder end fill = white,
        aspect = 1, 
        minimum width = 0.75cm,
        minimum height = 1cm,
        shape border rotate = 90] (c) at (2.5,1.25,0.225) { };
    \end{tikzpicture}
}}
\end{picture}
\caption{\centering Waves propagation with a floating cylinder.}
\label{fig:1}
\end{figure}
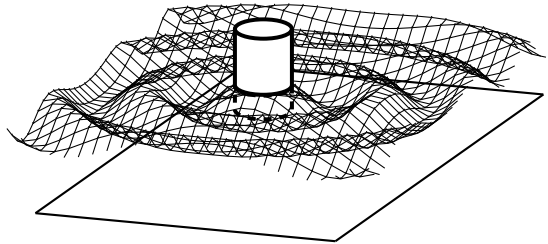
 
\section{Modelling  in coastal area and main results}
 
\subsection{Presentation of the fluid-structure problem}
The mathematical study of floating structures is a keystone to understand wave energy facilities. Let a {\it{partially immersed rigid object in a 3D fluid}} (see above Figure \ref{fig:1}) allowed to move only vertically. We will consider that the basin is infinite and have "small" finite depth such that we can assume being in some {\it{shallow-water regimes}}. All the equations of the paper are dimensionless. 
 
The domain of wave propagation $\mathbb{R}^2$ is divided into
\begin{itemize}
\item an {\it{interior region}} $\mathcal{I}$ where is located the floating body,
\item and an {\it{exterior region}} $\mathcal{E} = \mathbb{R}^2 / \mathcal{I}$. 
\end{itemize}
Since one considers a floating object with vertical wall and since the solid is constrained to move only vertically, the interior domain $\mathcal{I}$ is fixed. The case of non-vertical wall in shallow-water regime is studied in \cite{IguchiLannes} and \cite{Lannes20232d}.

\subsubsection{Rigid solid dynamics}
The dynamic of a floating body is given by the {\it{Newton's equation}}:
\begin{equation}\label{Newton}
\tau_{\rm buoy}^2   \ddot{\delta} + \frac{1}{\eps}m= \displaystyle \frac{1}{| \mathcal{I} | }\int_{\mathcal{I}}  \frac{P_{\rm i}}{\eps}dx + F_{\rm ext},
\end{equation}
where the unknown $\delta$ refers to the difference between the vertical displacement of the body's center of mass and its equilibrium's ordinate. The notation $\ddot{\cdot}$ refers to the second derivative respect to time. $F_{\rm ext}$ stands for a given external force continuous with respect to time. In the hydrodynamic force, the integrand $P_{\rm i}$ is the {\it{pressure exerted by the fluid under the body}}. It is clearly the difficult part since it requires to compute the fluid variables under the object. Moreover the wet surface (i-e the contact surface between the structure and the fluid), on which the fluid exerts the pressure, moves with respect to the time.
 
Note that \eqref{Newton} is the dimensionless version of the fondamental principle of the dynamics (see appendix A in \cite{BeckLannes}). The mass factor has been therefore replaced by a dimensionless buoyancy factor $\tau_{\rm buoy}^2$ and $\eps^{-1}m$ stands for the adimensional gravity constant. The notation $| \mathcal{I} |$ stands for the area of the interior domain $\mathcal{I}$. In dimensionless variable, the diameter of $\mathcal{I}$ is proportional to the Helmholtz number.
 
\subsubsection{Waves' dynamics}
In shallow water regime the fluid's domain is rescaled such that the height of the waves at rest is $1$. Also the surface elevation is of amplitude $\eps$.
Thus the surface elevation (i-e the difference between the water column at time $t$ and the water column at the equilibrium) is parametrized at time $t$ by the function $x\in \mathbb{R}^2 \mapsto \zeta(t, x)$, and the {\it{water column}} (i-e the height of waves from the bottom of the basin) is parametrized at time $t$ by the function $x\in \mathbb{R}^2 \mapsto h(t, x) = 1 + \eps \zeta(t,x)$. The {\it{discharge}} (i-e the vertical integral of the horizontal component of the velocity field) at time $t$ and position $x$ is denoted $Q(t,x) \in \mathbb{R}^2$. Finally we denote by $P(t,x)$ the {\it{pressure at the surface}} of the fluid. One considers that the motion of waves in coastal area is governed by a viscous perturbation of the {\it{Boussinesq-Abbott equations}}:
\begin{empheq}[left=\empheqlbrace]{align}
& \partial_t \zeta +\mbox{div}\, Q=0, \label{AbbottDL:eq1} \\
& (1- \kappa^2 \nabla \mbox{div} )\partial_t Q+\varepsilon \mbox{div}\, \big( \frac{Q \otimes Q}{h} \big)+h \nabla  \zeta= - h \nabla \frac{P}{\varepsilon}  + h \nabla \left(\frac{\nu}{h} \mbox{div} Q \right),  \label{AbbottDL:eq2} \quad r > 0.\\
& h = 1 + \varepsilon \zeta,
\end{empheq}\label{AbbottDL}
In these equations, the dimensionless parameters $\kappa, \varepsilon $ and $\nu$ are small. $\kappa$ refers to the shallowness parameter, $\varepsilon$ is the non-linear parameter and $\nu$ is related to viscous effect. Strictly speaking, $\nu$ is not related to the inverse of the Reynolds number but to the 2D equivalent of the nonstandard viscosity term introduced in \cite{Maity}. The dispersion is due to the presence of $(1- \kappa^2 \nabla \mbox{div} )$ operator and the fact that the propagation will be radial. \\

It is essential to separate the exterior and the interior regions because they are constrained in different ways. To underline the differences we will use the index $\cdot_{\rm e}$ and respectively $\cdot_{\rm i}$ to refer to the restriction on the exterior and respectively interiors regions, i-e for any functions $f$, one has
 $$
 f_{\rm e} := f_{| \mathcal{E}}
 \quad \text{and} \quad
 f_{\rm i} := f_{| \mathcal{I}}.
 $$
\begin{itemize}
\item In the {\it{interior region}}: the surface elevation must fit the wet surface (i-e the bottom of the object),
\begin{equation}\label{constrain-interior}
h_{\rm i}(t,x)= \delta(t) + h_{\rm i,eq}(x) \quad \mbox{ for } t>0, \quad x\in \mathcal{I},
\end{equation}
where $h_{\rm i,eq}$ is the {\it{parametrization of the wet surface at equilibrium from the bottom of the basin}}.
In consequence the surface's elevation $\zeta_{\rm i}$ does not depend on the space variable and by equation \eqref{AbbottDL:eq2} of the Boussinesq-Abbott system one has:
\begin{equation}\label{QintBefore}
\mbox{div } Q_{\rm i}(t, \cdot) = - \dot{\delta}(t) \quad \mbox{ for } t>0.
\end{equation}
In particular in $\mathcal{I}$, one has
\begin{equation}\label{lapQ}
\nabla \mbox{div} Q_{\rm i} = 0.
\end{equation}
However there is no constraint on the surface pressure $P_{\rm i}$ which, under the general approach of \cite{Lannes_float}, can be understood as the {\it{Lagrange multiplier associated with the constraint on the surface elevation}}.
\item In the {\it{exterior region}}, the surface elevation $\zeta_{\rm e}$ is free, but the surface pressure $P_{\rm e}$ is constrained, assumed to be constant 
\begin{equation}\label{constrain-exterior}
{P}_{\rm e}(t,x) = 0 \quad \mbox{ for } t>0, \quad x \in \mathcal{E},
\end{equation}
the gradient of pressure in equation \eqref{AbbottDL:eq2} therefore vanishes. 
Note that in reality, ${P}_{\rm e}(t, x) = P_{\rm atm}$ but since we only consider the gradient of the pressure, without loss of generalities one can consider $P_{\rm atm} = 0$.
In this region both the surface's elevation and the discharge are not constraint and are determined by the Boussinesq-Abbott system. 
\end{itemize}
However the system is not complete and needs transmission conditions at the contact line $\Gamma = \overline{\mathcal{E}} \cap \overline{\mathcal{I}}$, which are given by:
\begin{itemize}
\item (the conservation of the fluid volume)
\begin{equation}\label{ConsFluidVol}
[ Q \cdot e_r ]_{\Gamma} = 0,
\end{equation}
where $e_r$ is the unitary radial vector.
\item (the balance of the energy)
\begin{equation}\label{ConservationOfEnergy}
\left[ \frac{P}{\varepsilon}  + \zeta + \nu \dot{\zeta} + \kappa^2 \ddot{\zeta}+ \frac{\varepsilon}{2}\frac{| Q |^2}{h^2} \right]_{\Gamma} = 0.
\end{equation}
\item (the hypothesis that there is no vortex line at the contact line)
\begin{equation}\label{vortex-line}
\left[ \frac{Q}{h} \times n_\Gamma \right]_{\Gamma} = 0,
\end{equation}
where $n_\Gamma$ is the normal vector at the contact line. 
\end{itemize}
where the jump over the contact line is defined by
$$
[ f ]_{\Gamma} := \lim_{\eta^+ \to 0} f(R+ \eta^+) - \lim_{\eta^- \to 0} f(R- \eta^-).
$$
The link between the condition \eqref{ConservationOfEnergy} on the pressure at the concat line and the conservation of energy will be explained in Subsection \ref{sec:EqOnPressure}. Note that the energy of the Boussinesq-Abbott system is not conserved. This is why, in fact, \eqref{ConservationOfEnergy} is not truly linked to the conservation of energy but with an approximation in $\mathcal{O}(\varepsilon (\kappa^2+\nu))$ of the energy. In \cite{Lannes20232d} the authors show that if we do not make the non vorticity hypothesis on $\frac{Q}{h}$ at the surface, translated by \eqref{vortex-line}, one can construct multiple solutions of the system. To avoid these multiple solutions we impose the hypothesis \eqref{vortex-line}. 
 
In ''shallow'' water without viscosity $\nu = 0$, one can formally take $\kappa=0$ to recover the Saint-Venant's equations. This case is well-known and has been treated in \cite{Bocchi} in 2D but the wave's system neglects some important effects due to dispersion.
In \cite{GPSW2} the boundary problem is treated as a transition between a congested zone, under the floating object, and a free zone in the exterior domain. 
However, the issue to use these methods here is the vertical walls at the contact line, which prevent the continuity of the surface elevation. 

\subsection{Main result in the axisymmetric without swirl case}
 
We consider the case where the floating object is a cylinder of radius $R$, and there is axisymmetry about the axis going down the middle of the cylinder. Furthermore we consider that there is no swirl in the fluid and at its surface, and that the fluid fits the bottom of the solid in the inner domain $\mathcal{I}$ and at the contact line $\Gamma$. In that case, the flow is axisymmetric without swirl and therefore $(\zeta$, $Q)$ reads
$$
\zeta = \zeta(t,r)
\quad \text{and} \quad
Q = q (t,r) e_r,
$$
where $r$ stands for the radial coordinate (see Figure \ref{fig:2}). Moreover we will assume that $h$ does not go to 0. In such a situation, the condition \eqref{vortex-line} is satisfied. 
 
\begin{assu}\label{hyp-system}
We will assume that
\begin{itemize}
\item (axisymmetric without swirl)
$$
\zeta = \zeta(t,r)
\quad \text{and} \quad
Q = q (t,r) e_r, 
\quad t \geq 0, \quad r \in \left [ 0, + \infty \right ),
$$
\item (floating cylinder)
$\mathcal{I} $ is the ball of center $0$ and radius $R >0$, in particular the contact line $\Gamma$ is the circle of center $0$ with radius $R$,
\item (flat bottom of the cylinder) $h_{\rm i,eq}$ is a positive constant,
\item (compatibility condition) at $t= 0$, the discharge $Q$ is such that $\mbox{div} Q _{|_{t= 0}} = - \dot{\delta}_{|_{t= 0}}$,
\item (floating cylinder initially does not touch the bottom) $\delta_{|_{t=0}} \in \mathbb{R}$ such that $\delta_{|_{t=0}} + h_{\rm i,eq}(x) > 0$,
\item (Boussinesq-Abbott regime) $\varepsilon,\nu \in [0,1)$ and $\kappa \in (0,1)$ are such that $\varepsilon(\kappa^2 + \nu) \ll 1$,
\item (regularity of external forces) $F_{\rm ext}$ is at least bounded and continuous.
\end{itemize}
\end{assu}

\begin{figure}
\begin{tikzpicture}		
	\put(0, 0){\makebox(0,0){	
    	\draw[thick, red] (0,0) -- (3,0);
    	\draw[thick, ->, blue](3,0) -- (10,0) node[right] {$r$};
    	\draw[thick, ->] (0,0) -- (0,4.2) node[above] {$h(t,r)$};

		\draw[violet] (3,-0.1) -- (3,0.1);
		\node[below] at (1.5, -0.1) {{\color {red}$\mathcal{I}$}};

		\draw (0,1) -- (3,1);
		\draw (3,1) -- (3,4);
		
		\draw (3, 2) -- (10, 2);
		\node at (-0.2, 2) {\small 1};
		\node at (0, 2) {-};
		\draw[<->] (4, 2) -- (4, 2.8);
		\node at (4.7, 2.4) {$\eps \zeta(r, t)$};
		
		\draw[->] (3.5, 1) -- (3.8, 1);
		\draw[->] (3.5, 0.8 ) -- (3.8, 0.8);	
		\draw[->] (3.5, 1.2) -- (3.8, 1.2);
		\node at (4.6, 1) {$q(r, t)$};

		\draw[domain=0:2*pi+0.7,smooth,variable=\x,blue,thick] plot ({3 + \x},{2 - exp(-\x/5)*sin(deg(\x- 2.5))});
		\node[below] at (6.5, -0.1) {{\color {blue} $\mathcal{E}$}};
		\node[above] at (1.5, 0.1) {{\color {red}$f_i$}};
		\node[above] at (6.5,0.1){{\color {blue}$f_e$}};
		\node[above] at (2.8,0.1) {{\color {red}$\underline{f_i}$}};
		\node[below] at (3.2,-0.1) {{\color {blue}$\underline{f_e}$}};	
		
} }	

\end{tikzpicture}
\caption{\centering Illustration of a slice of the problem. The function $f$ refers to any spatial function.}
\label{fig:2}
\end{figure}
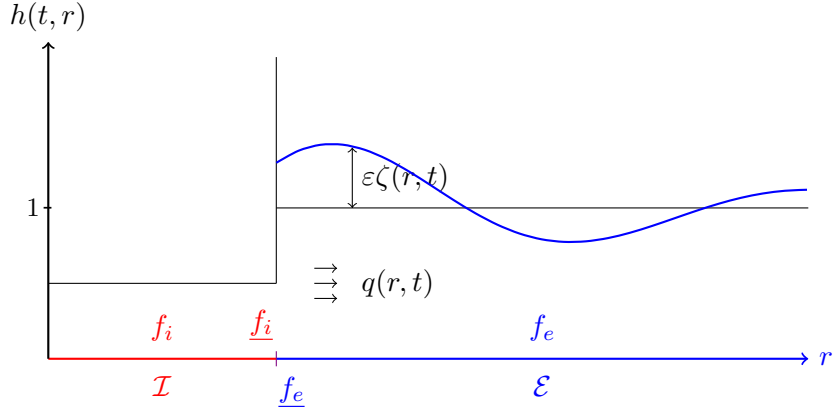

The assumption that the bottom of the object is flat is not necessary. However it simplifies the calculations. 
With these considerations, \eqref{QintBefore} gives
\begin{equation}\label{Qint}
q_{\rm i}(t, r) = - \frac{r}{2}\dot{\delta}(t) \quad \mbox{ for } t>0, \quad 0 \leq r \leq R.
\end{equation}
In \cite{Bocchi} such a configuration was considered for the hyperbolic nonlinear shallow water equations (equations \eqref{AbbottDL:eq1} and \eqref{AbbottDL:eq2} with $\kappa=0$ and $\nu = 0$). The considered Boussinesq \eqref{AbbottDL:eq1} and \eqref{AbbottDL:eq2} equations are a dispersive and viscous perturbation of the hyperbolic nonlinear shallow water equations. This perturbation induces drastic changes on the behavior and analysis of the associated initial boundary value problem. 
Thanks to the symmetries of the problem and to the polar coordinates, for any radial vector field without azimuthal component $Q(r,\theta) = q(r) e_r$, the divergence operator is simply given by
$$
{\rm{div}}Q = {\rm{d}}_{r} q := \partial_r q + \frac{q}{r}.
$$
Thus the Boussinesq system can be written as the radial Boussinesq-Abbott system
\begin{empheq}[left=\empheqlbrace]{align}
&\partial_t \zeta +\mbox{d}_r\, q=0, \label{BAr:eq1} \\
&(1- \kappa^2 \partial_r \mbox{d}_r )\partial_t q+\varepsilon  \mbox{d}_r\, \left(  \frac{|q|^2}{h} \right)+ h \partial_r  \zeta = - h \partial_r \frac{P}{\varepsilon}  + \nu h \partial_r \frac{\dr q}{h}, \label{BAr:eq2} \quad r > 0,\\
& h = 1 + \varepsilon \zeta.
\end{empheq}
We refer to this system as \hypertarget{BAr}{(BAr)}.
The main results of the paper are the following three:
\begin{enumerate}
\item a convenient way to rewrite the waves-structure system as a non-local conservation law plus bounded perturbation,
\item a well-posedness theorem,
\item a more detailed study of the decrease in the position of the center of mass of the floating solid in the linear regime.
\end{enumerate}

\begin{thm}[Augmented formula]\label{thmf:Aug} 
Let Assumption \ref{hyp-system} be satisfied.
Let $(\zeta, q)$ and $(\delta, \dot{\delta})$ be a regular enough solution of the radial Boussinesq-Abbott system {\rm \hyperlink{BAr}{(BAr)}} combined to the Newton's equation \eqref{Newton} with the transmission's conditions \eqref{ConsFluidVol}, \eqref{ConservationOfEnergy} and \eqref{vortex-line}. Then $u := (\zeta_{\rm e}, q_{\rm e})$ and $Z := (\delta, \dot{\delta}, \underline{\zeta_{\rm e}}, \underline{\dot{\zeta}_{\rm e}})$ are solutions of the Augmented formulation Boussinesq system
\begin{empheq}[left=\empheqlbrace]{align}
&\partial_t u + \partial_r \widetilde{\mathfrak{F}}_\kappa [u, Z] = \mathcal{S}_r(u) + \nu \eps S_\nu(u) , & \text{in } (R, \infty), \label{Boussi-augmented1}\\
&\partial_t Z + \widetilde{\mathfrak{D}}(Z, \mathcal{I}_{\rm hyd} u, F_{\rm ext}) = 0, & \text{for } t > 0, \label{Boussi-augmented2}\\
&\underline{u_1} = Z_3, \underline{u_2} = -\frac{R}{2} Z_2 , & \text{at } r = R,  \label{Boussi-augmented3}\\
&u_{|_{t=0}} = (\zeta_{{\rm e}|_{t=0}}, q_{{\rm e}|_{t=0}}), Z_{|_{t=0}} = (\delta_{|_{t=0}}, {\dot{\delta}}_{|_{t=0}}, \underline{\zeta_{\rm e }}_{|_{t=0}}, \underline{\dot{\zeta}_{\rm e}}_{|_{t=0}}), & \text{at } t = 0,\label{Boussi-augmented4}
\end{empheq}
where $\widetilde{\mathfrak{F}}_\kappa, \mathcal{I}_{\rm hyd}$ and $S_\nu$ are non-local and non-linear operators, and $S_r$ is a non-linear operator. All of them are detailed in Remark \ref{rk:DefAugF}. 
Reciprocally, let $(u, Z)$ be a regular enough solution of the Augmented formula. One introduces $(\zeta_{\rm i}, q_{\rm i}, P_{\rm i})$ given by
$$
\zeta_{\rm i} = \varepsilon^{-1} ( Z_1 + h_{\rm i, eq} -1)
\, , \, 
q_{\rm i}(r) = - \frac{r}{2} Z_2,
$$
and
$$
\frac{P_{\rm i}(r)}{\varepsilon} = \frac{r^2-R^2}{1 + \varepsilon \zeta_{\rm i} } \left( 
 \frac{\dot{Z_2}}{2}  - \frac{3 \varepsilon | Z_2 |^2}{4 (1 + \varepsilon \zeta_{\rm i})} \right)
 + Z_3 - \underline{\zeta_{\rm i}} + 
\nu(Z_4 - \underline{\dot{\zeta}_{\rm i}}) + 
\kappa^2 (\dot{Z}_4 - \underline{\ddot{\zeta}_{\rm i}}) + 
\frac{\varepsilon}{2} \left (\frac{| u_2|^2}{(1 + \varepsilon u_1)^2} - \frac{| \underline{q_{\rm i}} |^2}{(1 + \varepsilon \zeta_{\rm i})^2}\right),
$$
for $0 \leq r \leq R$.
Then $(\delta, \dot{\delta}, P_{\rm i}) := (Z_1, Z_2, P_{\rm i})$ solves the Newton's equation \eqref{Newton}, and 
$$
(\zeta, q, P) := \begin{cases}
(u_1,u_2,0) & \text{ in } \mathcal{E}, \\
(\zeta_{\rm i}, q_{\rm i}, P_{\rm i}) & \text{ in } \mathcal{I},
 \end{cases}
$$
solves the radial Boussinesq-Abbott system of equations \eqref{BAr:eq1} and \eqref{BAr:eq2} with the transmission's conditions \eqref{ConsFluidVol}, \eqref{ConservationOfEnergy} and \eqref{vortex-line}. 
\end{thm}
 
In this formula, $\widetilde{\mathfrak{F}}_\kappa$ and $\mathcal{I}_{\rm hyd}u$ present a non-local effect due to the dispersion. This non-locality does not appear in the hyperbolic case \cite{Bocchi}. 
This theorem is equivalent to the one we can find in \cite{BeckLannes} but with radial terms due to the 2D situation, represented in equation \eqref{Boussi-augmented1} by $S_r(u)$, and loss terms due to the presence of viscosity.
This Augmented formula is the keystone to show wellposedness. 
Section \ref{sec:AugFor} is devoted to the proof of this theorem. 
 
\begin{defn}\label{def:WP}
A problem given by a PDE is said to be well-posed in the Banach space $X$ for time of order $\mathcal{O}(T_{\varepsilon, \kappa, \nu, R}^+)$ if there exists $c >0$ uniform with respect to $\varepsilon, \kappa, \nu$ and $R$ such that for any initial condition in $X$ and any continuous source term, the problem admits a unique solution which lies in  $\mathscr{C}^0 \left( [0, cT_{\varepsilon, \kappa, \nu,R}^+ ) \, , \,  X \right)$.
\end{defn}
 
\begin{thm}[Wellposedness of the Augmented formula]\label{thmf:WP}
Let Assumption \ref{hyp-system} be satisfied. 
 
\begin{itemize}
\item[(i)]\label{thmf:pt1} Let $\nu \geq  0$. The Augmented formulation \eqref{Boussi-augmented1}, \eqref{Boussi-augmented2}, \eqref{Boussi-augmented3}, \eqref{Boussi-augmented4} of the waves-structure system is well-posed in $L^\infty_r \times W^{1, \infty}_{r, \kappa} \times \mathbb{R}^4 $ for time at order $\mathcal{O}(T_{\rm ODE})$ where the existence time is given by 
$$T_{\rm ODE} = \frac{\kappa^2 \eps^{-1}}{(1+\nu \kappa^{-1})^2},$$ 
and the functional frameworks are
$$
L^\infty_r := L^\infty( [R, \infty) \, , \, rdr)
\quad \text{and} \quad
W^{1,\infty}_{r, \kappa} := \{ f \in L^\infty_r \, | \,  \kappa\partial_r f \in L^\infty_r \}.
$$
\item[(ii)]\label{thmf:pt2} Let $\nu = 0$. The Augmented formulation \eqref{Boussi-augmented1}, \eqref{Boussi-augmented2}, \eqref{Boussi-augmented3}, \eqref{Boussi-augmented4} of the waves-structure system is well-posed in $ H^1_r \times H^2_\kappa \times \mathbb{R}^4$ for time at order $\mathcal{O}(T_{\varepsilon, \kappa, R})$
where the existence time is given by
$$
T_{\varepsilon, \kappa,R}:= \frac{-(\sqrt{\alpha} +\max(1, \frac{1}{R^2})\kappa^{-1}) + \sqrt{(\sqrt{\alpha} +\max(1, \frac{1}{R^2})\kappa^{-1})^2 + 4\eps^{-1}R}}{2},
$$
for $\alpha := \max(2R^+, 9)$, $R^+ = R + \rho$ for $0<\rho \ll 1$, and the functional frameworks are defined by
$$
H^1_r := \{ f \in L^2( [R, \infty) \, , \, rdr) \, | \, || f ||_{H^1_r}^2 := \int_R^\infty |f(r)|^2 r dr +  \int_R^\infty |\partial_r f(r)|^2 r dr  < \infty \},
$$
and
$$
H^2_\kappa := \{ f \in H^1_r  \, | \, || f ||_{H^2_\kappa}^2 := || f ||_{H^1_r}^2  + \kappa^2 \int_R^\infty |\partial_r {\rm d}_r f(r)|^2  r dr < \infty \}.
$$
\end{itemize}
Moreover in both situations, if $T^\ast$ denotes the maximal existence time and $T^\ast < \infty$, one has
\begin{equation}\label{Blow-up-crit}
\lim \sup_{t \to T^\ast} \|\frac{1}{h_{\rm e}(t)}, \zeta_{\rm e}(t), q_{\rm e}(t)\|_{L^\infty_r} + |\dot{\delta}(t)| + \frac{1}{|h_{\rm i}(t)|} + |F_{\rm ext}(t)| = + \infty.
\end{equation}
\end{thm}
The blow-up criterium \eqref{Blow-up-crit} is rather similar to the one given in \cite{BeckLannes} Theorem 3.3. In preparation for future works on asymptotic analysis, we were careful about the dependence on the small parameters $\varepsilon, \kappa, R$ or $R^{-1}$, and $\nu$ in the first point. Different regimes with interest are summed up in the following table (see Table \ref{tab:T}), with $\lambda = \sqrt{\alpha} + \max(1,\frac{1}{R^2})\kappa^{-1}$ and $\nu = 0$.

\begin{center}
\begin{tabularx}{\textwidth}{|Y|Y|Y|Y|Y|Y|}
\hline
\textbf{Regime} & $\eps \ll \kappa^2$ and $R = \mathcal{O}(1)$ & $\eps \gg \kappa^2$ & $\eps \sim \kappa^2$ &  $R \to 0$ & $R^{-1} \to 0$ \\ \hline
\begin{center}
$T_{\rm ODE} \sim$ 
\end{center}
 & \begin{center} $\kappa^2\eps^{-1}$ \end{center}
 & \begin{center} $\kappa^2\eps^{-1}$ \end{center}
 & \begin{center} $(1+ R^{-1})^{-1}$ \end{center}
 & \begin{center} $\kappa^2\eps^{-1}$ \end{center}
 & \begin{center} $\kappa^2\eps^{-1}$ \end{center}\\ \hline

\begin{center}
$T_{\eps, \kappa, R} \sim$
\end{center}
 & \begin{center} $\eps^{-1/2}$ \end{center}
 & \begin{center} $\eps^{-1}\lambda^{-1}$ \end{center}
 & \begin{center} $\eps^{-1}\lambda^{-1}$ \end{center}
 & \begin{center} $R$ \end{center}
 & \begin{center} $R^{-1}$ \end{center}\\ \hline
\end{tabularx}
\captionof{table}{Form of the existence time of the solution in the different regimes.}
\label{tab:T}
\end{center}

We can obtain the limit for high Helmholtz numbers ($R \to \infty$) with Theorem \ref{thmf:WP}-(i), but the limits for weak Helmholtz numbers $R \to 0$ (with the second method) or high Helmholtz numbers $R \to \infty$ do not commute with the limit $\eps \to 0$ with the first point of the theorem, in the meaning that $\underset{\eps \to 0}{\lim}\underset{\tilde{R} \to 0}{\lim} T_{\eps, \kappa, R} \neq \underset{\tilde{R} \to 0}{\lim}\, \underset{\eps \to 0}{\lim}T_{\eps, \kappa, R}$ with $\tilde{R} = \min(R, R^{-1})$.\newline
\indent When the dispersion vanishes $\kappa \to 0$, the lifetime follows it. This is because we have used the regularization by dispersion in our proof. 
In the hyperbolic case $(\kappa = 0)$ the wellposedness is shown with Kreiss symmetrizers techniques (see \cite{Bocchi}). In the linear case those techniques consist in the multiplication of the system by a symmetric matrix, which satisfies some properties that take care of the boundary conditions. From this symmetrized system one can obtain energy inequalities and show the wellposedness of the system. In non-linear regime those techniques are adapted with an iterative quasi-linear scheme.
However when $\kappa > 0$ those symmetrizers will not commute with the dispersion $\kappa^2 \nabla{\rm div} $ and thus similar techniques cannot be applied. To circumvent the absence of a symmetrization method, we will use the Augmented formulation \eqref{Boussi-augmented1}, \eqref{Boussi-augmented2}, \eqref{Boussi-augmented3}, \eqref{Boussi-augmented4}. This is why we do not have the same results for the wellposedness. In particular, Kreiss symmetrizer method requires compatibility conditions which are written in the form of an equality on the initial conditions. This is not the case for the dispersive case. We believe that in order to reach the dispersionless limit, we should introduce compatibility conditions. These compatibility conditions in dispersion regime would be written under the form of inequalities, which converge to hyperbolic compatibility conditions inequalities when the dispersion vanishes, see \cite{BLM}.\newline
\indent In the linear case ($\eps = 0$) Theorem \ref{thmf:WP} gives that $\delta$ belongs to $\mathscr{C}^0(\mathbb{R}^+)$. Therefore, similarly to \cite{BeckLannes}, we can study the behavior of $\delta$ in the return to the equilibrium situation for large time scale. It is the case where the solid is dropped with no speed and the waves are at rest at $t = 0$. The behavior of $\delta$ is given by the following theorem. 
\begin{thm}[Return to the equilibrium]\label{thmf:RTTE}
Let Assumption \ref{hyp-system} be satisfied except for the last point. One assumes that $\varepsilon = 0$ and $\kappa, \nu \geq 0$ instead of this last point. One also assumes the technical Assumption \eqref{conjecture-denom}.
If the initial conditions are given by $(\zeta, q, \delta, \dot{\delta}) = (0, 0, \delta_0 \neq 0, 0)$ then $\delta \in H^2(\mathbb{R}^+)$ solves an integro-differential equation which does not require the computation of $\zeta_e$ or $q_e$ (see equation \eqref{eq:CumminsDelta}).
On one hand, if $\nu = 0$ and $\kappa >0$, then the decay of $\delta$ cannot be stronger than $\mathcal{O}(t^{-1/2^+})$. On the other hand, if $\nu > 0$ or $\kappa = 0$, then the decay of $\delta$ cannot be stronger than $\mathcal{O}(t^{-1^+})$, and the slowest decay is in $\mathcal{O}(t^{-3/2^+})$.
\end{thm}
In the previous theorem, for any $\rho \in \mathbb{R}$, $\rho^\pm$ denotes any number
$ \rho^\pm = \eta \pm \beta$ for any arbitrary small real number $0 < \beta \ll 1$. Let us denote
$$
\eta_0 := \begin{cases}
\frac{\nu}{2\kappa^2} & \text{ if } \nu \geq 0 , \kappa>0, \\
\nu^{-1} & \text{ if } \nu >0, \kappa=0, \\
\frac{R}{2\tau_0^2(0)} & \text{ if } \nu =\kappa=0.
\end{cases}
$$
For $k \in \left \{ 0, 1\right \}$, the behavior of $\delta^{(k)}$ is summed in Table \ref{tab:delta}.
 
\begin{center}
\begin{tabularx}{\textwidth}{|Y|Y|Y|Y|Y|}
\hline
\textbf{Parameters} & Dimension & $\delta \in$ & $\delta^{(k)} \notin$ & Decay of $\delta, \dot{\delta}$ \\ \hline
 
\begin{center}
$\nu = 0$ and $\kappa > 0$
\end{center}
 & \begin{center} 2D \end{center}
 & \begin{center} $H^2$ \end{center}
 & \begin{center} $L^2(t^\beta dt)$, $\beta \in (0, 2]$ \end{center}
 & \begin{center} Slower than $\mathcal{O}(t^{-1/2^+})$ \end{center}\\ \hline
 
\begin{center} $\nu > 0$ or $\kappa = 0$ \end{center}
 & \begin{center} 2D \end{center}
 & \begin{center} $H^2(t^\beta dt)$, $\beta \in (0, 2)$ \end{center}
 & \begin{center} $L^2(t^{2k + 2}dt)$ \end{center}
 & \begin{center} Slower than $\mathcal{O}(t^{-(3/2+k)^+})$ and faster than $\mathcal{O}(t^{-(k+1)})$ \end{center}  \\ \hline
 
\begin{center}
$\nu \geq 0$ and $\kappa \geq 0$
\end{center}
 & \begin{center} 1D \end{center}
 & \begin{center} $H^2(e^{\eta_0 t}dt)$ \end{center}
 & \begin{center} $L^2(e^{\eta_0 t}t^\beta dt)$, $\beta \in (0, 2]$ \end{center}
 & \begin{center} Slower than $\mathcal{O}(e^{-\frac{\eta_0 t}{2}}t^{-1/2^-})$ \end{center}\\ \hline
 
\end{tabularx}
\captionof{table}{Properties of $\delta$ depending on the dimension $\&$ on the small parameters $\kappa$ and $\nu$.}
\label{tab:delta}
\end{center}
 
We believe that the technical Assumption \eqref{conjecture-denom} is verified based on numerical evidence. It is also motivated in Section \ref{sec:RTTE}.
In the dispersive ($\kappa >0$) non-viscous ($\nu=0$) case in 1D or 2D, the strongest decay is in $\mathcal{O}(t^{-1/2^+})$. 
This is an improvement of one of the results from \cite{BeckLannes}. Indeed in \cite{BeckLannes} the authors show that the strongest decay was in $\mathcal{O}(t^{-3/2^+})$ for the 1D case. 
Furthermore we show that in 2D, even with viscosity one cannot have exponential decreasing of $\delta$ as in the 1D  case.
 
The following three sections are each dedicated to the demonstration of one of Theorems \ref{thmf:Aug}, \ref{thmf:WP} and \ref{thmf:RTTE}. They can be read independently of each other.
 
\subsection{Notations}
\begin{itemize}
\item $\varepsilon$ is the non-linearity parameter, $\kappa$ is the dispersion or shallowness parameter, $\nu$ is the viscosity parameter and $R$ is the  Helmholtz number. 
\item For any $a, b \in \mathbb{R}$ 
$$a \lesssim b \Leftrightarrow \exists C > 0 \, ; \,  a \leq C b, $$ with $C$ uniform with respect to $\kappa,\eps,\nu$ and $R$.
\item For any function $u : (R, \infty) \to \mathbb{R}$, one denotes its trace when it is well-defined by $$\underline{u} = \lim_{\eta \to 0} u(R+ \eta).$$
\item $(e_r,  e_\theta)$ is the orthonormal basis of $\mathbb{R}^2$ in polar coordinates. The divergence operator in polar coordinates reads, for $u = (u_r, u_\theta)$
$${\rm div}u := {\rm{d}}_{r} u_r + \frac{1}{r} \partial_\theta u_\theta
\quad \text{where} \quad
{\rm{d}}_{r} \cdot= \partial_r \cdot+ \frac{\cdot}{r}.$$
\item The radial Sobolev spaces are defined by
$$
W^{s,n}_r = W^{s,n}([R, +\infty), rdr) , \text{ for } s,n \in \mathbb{N},
$$
in particular
$$ 
L^2_r = W^{2,0}_r = L^2([R, +\infty), r dr), \,
H^n_r = W^{2,n}_r= H^n([R, +\infty), r\rm{dr}), \quad n \in \mathbb{N}.
$$
\item We note the weighted Sobolev space $H^{1}_\kappa$ defined as the functionnal space equipped with the norm $$||\cdot||_{H^{1}_\kappa}^2 := ||\cdot||_{L_{r}^2}^2 + \kappa^2 || {\rm{d}}_{r} \cdot||_{L^2_r}^2.$$
\item We note the weighted Sobolev space $H^{2}_\kappa$ the functionnal space equipped with the norm $$||\cdot||_{H^{2}_\kappa}^2 := ||\cdot||_{H^1_r}^2 + \kappa^2 || \partial_r {\rm{d}}_{r} \cdot||_{L^2_r}^2.$$ 
\item We note $\mathbb{H} = L^2_r \times H^1_\kappa$ and $\mathbb{H}^1 = H^1_r \times H^2_\kappa$.
\item $I_0$, $I_1$, $K_0$, $K_1$ are modified Bessel functions of order $0$ or $1$ of first and second kind (see Definition \ref{def:K}). 
\item For any $f : \mathbb{R} \to \mathbb{R}$ one denotes 
 $$
 f_{\rm e} := f_{| \mathcal{E}},
 \quad 
 f_{\rm i} := f_{| \mathcal{I}},
$$
and the traces when they are well-defined
$$
  \underline{f_{\rm e}} = \lim_{\eta \to 0} f(R+\eta),
   \quad 
   \underline{f_{\rm i}} = \lim_{\eta \to 0} f(R-\eta).
 $$
\item For any time function $f$ regular enough 
$$ 
\dot{f} = \frac{\partial f}{\partial t}
\quad \text{and} \quad
\ddot{f} = \frac{\partial^2 f}{\partial t^2}.
$$
 
\item For any $\rho \in \mathbb{R}$, $\rho^\pm$ denotes any real number strictly superior/inferior to $\rho$. More precisely, $\rho^\pm$ is the number
$ \rho^\pm = \eta \pm \beta$ for an arbitrary small real number $0 < \beta \ll 1$.
 
\item For any finite real number $\eta_0$ we define the complex half-plane $\mathbb{C}_{\eta_0}$ as
$$ \mathbb{C}_{\eta_0} = \left\{ \eta + i \omega | \, \eta > \eta_0, \omega \in \mathbb{R}\right\}. $$

\end{itemize}
 
\section{Augmented formulation}\label{sec:2}
Firstly we want to rewrite the waves-structure system as in the first Theorem \ref{thmf:Aug}. To do so, the method will be the following:
\begin{enumerate}
\item {\it{(Equation on the pressure)}}.
Firstly we aim at getting an elliptic equation for the surface pressure in the inner domain $\mathcal{I}$. 
Since we constraint the fluid to fit the bottom of the solid one can obtain by taking the space divergence of equation \eqref{BAr:eq2}:
\begin{equation}\label{elliptic-pressure}
\begin{cases}
-\ddot{\delta} + \frac{3 \eps}{2h_i}\dot{\delta}^2 = - h {\rm{d}}_{r}(\partial_r \frac{P_{\rm i}}{\varepsilon})& \text{ in $\mathcal{I}$},\\
\frac{\underline{P_{\rm i}}}{\varepsilon}
= \left[ \zeta + \nu \dot{\zeta} + \kappa^2 \ddot{\zeta} + \frac{\varepsilon}{2}\frac{| Q |^2}{h^2} \right]_{\Gamma} & \text{ on $\Gamma$},
\end{cases}
\end{equation}  
with the boundary conditions given by the energy balance of the fluid-structure system. 
 
\item {\it{(Eliminating the Lagrange multiplier)}}.\label{ElimLagM} 
If we inject this elliptic problem in the Newton's equation \eqref{Newton}, one gets nonlinear system of ODEs
\begin{equation}\label{ODEIBVP}
{\tau^2_\kappa}(\varepsilon \delta)  \ddot{\delta}
+ \beta(\delta, \dot{\delta}, \zeta_{\rm e})
= \mathfrak{f}( \underline{\zeta_{\rm e}}, \kappa\underline{\dot{\zeta_{\rm e}}}, \kappa^2\underline{\ddot{\zeta_{\rm e}}}),
\end{equation}
with forcing term $\mathfrak{f}(\underline{\zeta_{\rm e}}, \kappa\underline{\dot{\zeta_{\rm e}}}, \kappa^2\underline{\ddot{\zeta_{\rm e}}})$ coming from the wave-field in the free area
that do not need the computation of the surface pressure. 
The coefficients ${\tau^2_\kappa}(\varepsilon \delta)$, $ \beta(\delta, \dot{\delta}, \zeta_{\rm e})$ and the forcing term $\mathfrak{f}( \underline{\zeta_{\rm e}}, \kappa\underline{\dot{\zeta_{\rm e}}}, \kappa^2\underline{\ddot{\zeta_{\rm e}}})$ will be explicited as in \cite{BeckLannes}. 
We want to underline that the coefficient in front of $\ddot{\delta}$ is higher than the intrinsic buoyancy $\tau^2_{\rm buoy}$ of the solid. This is due to the fact that the pressure force is exerted on the solid by the fluid and it involves acceleration term (see \eqref{elliptic-pressure}). This is what we call the added-mass effect. The equations of waves-structure system are now completed but one needs to avoid higher order terms $\kappa^2\underline{\ddot{\zeta_{\rm e}}}$.
 
\item {\it{(Hidden equation)}}\label{EqOnTrace}
In order to replace $\kappa^2\underline{\ddot{\zeta_{\rm e}}}$ we will work with non-local dispersive operators which reverse the dispersive effect in equation \eqref{BAr:eq2}. It leads to another nonlinear equation of type:
\begin{equation}
\kappa^2 \underline{\ddot{\zeta_{\rm e}}} + \tilde{\beta}(\delta, \dot{\delta}, q_{\rm e}, \zeta_{\rm e}, \underline{\zeta_{\rm e}}, \kappa\underline{\dot{\zeta_{\rm e}}}) = \kappa\ddot{\delta}C_\kappa.
\end{equation}
The term $C_\kappa$ represents a dispersive boundary layer effect as in \cite{BeckLannes} which decays at least exponentialy.
The dependance on $\zeta_{\rm e}$ and $q_{\rm e}$ in $\tilde{\beta}$ comes from the non-local effect of inverting the operator $(1-\kappa^2 \nabla \mbox{div})$ which requires to compute those quantities over the exterior domain $\mathcal{E}$.
If we sum up the three previous steps we complete the waves-structure system and we can obtain the Augmented formulation.
 
\item {\it{(Augmented formulation)}}
In order to do so we consider the following unknown
$$
Z := (\delta, \dot{\delta}, \underline{\zeta_{\rm e}}, \kappa\dot{\underline{\zeta_{\rm e}}}),
$$
and by combining (\ref{ElimLagM}) and (\ref{EqOnTrace}) our new unknown solves the following system of ODE's
\begin{equation}\label{ODEaugm}
\mathcal{M} (Z) \,  \frac{d}{dt}Z={\mathcal G}\left (Z ,   \mathfrak{Q} \right ), 
\end{equation}
where the source $ \mathfrak{Q}$ coming from the computation of the waves' unknowns on the exterior domain. Combining that with \eqref{AbbottDL:eq1} and \eqref{AbbottDL:eq2} leads to the Augmented formula and Theorem \ref{thmf:Aug}.
\end{enumerate}
 
As it was explained in the introduction, the difficulty raises in the computation of the pressure in the inner domain $P_{\rm i}$. In order to avoid this problem in the two first subsections we manage to eliminate this factor in the description of the system by rewritting it as a transmission problem.
\subsection{Equation on the pressure}\label{sec:EqOnPressure}
In the inner domain $\mathcal{I}$, the unknown are the pressure $P_{\rm i}$ and the discharge $q_{\rm i}$ (and $(\delta, \dot{\delta})$). The expression of the pressure $P_{\rm i}$ is given by the following proposition.
\begin{proposition}
If $(\zeta, q)$ is a regular enough solution of {\rm \hyperlink{BAr}{(BAr)}} then the pressure in the interior domain $P_{\rm i}$ is given by
\begin{equation}\label{eq:pi(r)}
\forall r  \leq R, \, \frac{P_{\rm i}(r)}{\varepsilon} - \frac{P_{\rm i}(R)}{\varepsilon}  = \frac{r^2-R^2}{2h_{\rm i}} \left( 
 \frac{\ddot{\delta}}{2}  - \frac{3 \varepsilon | \dot{\delta}|^2}{4 h_{\rm i}} \right).
\end{equation}
\end{proposition}
In the expression of the inner pressure, the constant $P_{\rm i}(R)$ is later computed thanks to the energy balance of the waves-structure system.
\begin{proof}
Let $(\zeta, q)$ be a regular enough solution of the radial Boussinesq-Abbott system {\rm \hyperlink{BAr}{(BAr)}}.
In the interior domain $\mathcal{I}$ as a consequence of Assumption \ref{hyp-system}, \eqref{Qint} and \eqref{constrain-interior}, both viscous term $\nu \partial_r \frac{\dr q_{\rm i}}{h_{\rm i}}$ and dispersive term $\kappa^2 \partial_r {\rm d}_r q_{\rm i}$ are null. The space derivative of $\zeta_{\rm i}$ is also null. Thus $(\zeta, q)$ satisfies
$$
\partial_r  \zeta_{\rm i} = \kappa^2 \partial_r {\rm d}_r q_{\rm i} = \nu \partial_r \frac{\dr q_{\rm i}}{h_{\rm i}}  = 0.
$$
Thus, equation \eqref{BAr:eq2} of the radial Boussinesq-Abbott system in the interior domain reads
$$
\partial_t q_{\rm i}+\varepsilon  \mbox{d}_r\, \left(  \frac{|q_{\rm i}|^2}{h_{\rm i}} \right) = - h_{\rm i} \partial_r \frac{P_{\rm i}}{\varepsilon}.
$$
By replacing $q_{\rm i}$ by its value, given by \eqref{Qint}, one obtains
$$
- \frac{r}{2} \ddot{\delta} + \frac{3 \varepsilon r | \dot{\delta}|^2}{ 4 h_{\rm i}}  = -  h_{\rm i} \partial_r \frac{P_{\rm i}(r)}{\varepsilon},
$$
and by integration on $[r,R]$ it becomes
\begin{equation}\label{eq:pi(r)-pi(R)}
 \frac{P_{\rm i}(r)}{\varepsilon} -  \frac{P_{\rm i}(R)}{\varepsilon} = \frac{r^2-R^2}{h_{\rm i}} \left( 
 \frac{\ddot{\delta}}{2}  - \frac{3 \varepsilon | \dot{\delta}|^2}{4 h_{\rm i}} \right).
\end{equation}
\end{proof}
Moreover one can express $\frac{P_{\rm i}(R)}{\varepsilon}$ as a function of the different quantities through the study of the energy balance of the waves-structure system. 
 
\subsubsection{Energy balance}
Inner and exterior domains are coupled by two transmission conditions. Firstly the conservation of the volume implies the continuity of the discharge in $r = R$
\begin{equation}\label{eq:ConsMass}
\underset{r\to R^-}{\lim}q_{\rm i}:= \underline{q_{\rm i}} = \underline{q_{\rm e}} := \underset{r\to R^+}{\lim}q_{\rm e}.
\end{equation}
Secondly there is a condition on the pressure at $r = R$, linked to the balance of the energy that is satisfied at order $\mathcal{O}(\varepsilon (\kappa^2+\nu))$. More precisely, if this second condition is satisfied, then the balance of energy of the waves-structure system reads:
\begin{equation}\label{eq:ConsEnGlob}
\frac{d}{dt} {\mathfrak{E}}_{\rm tot} =  F_{\rm ext}  \dot\delta  - \int_{\mathbb{R}} \frac{\nu}{h} | {\rm d}_r q |^2 + \mathcal O(\varepsilon (\kappa^2+\nu)),
\end{equation}
where ${\mathfrak{E}}_{\rm tot}$ stands for the energy of the shallow water-waves / floating structure system in the Boussinesq regime. In this energy balance, $F_{\rm ext}  \dot\delta$ stands for the energy given to the system by the external force $F_{\rm ext}$ and $ - \int_{\mathbb{R}^2} \frac{\nu}{h} | {\rm d}_r q |^2$ is the loss due to the viscosity $\nu$.
It is not strictly speaking an energy balance, as there are $\mathcal O(\varepsilon (\kappa^2+\nu))$ remainder terms. However the Boussinesq-Abbott model is a good approximation of the free-surface Euler equations at order  $\mathcal O(\varepsilon(\kappa^2+\nu))$. Thus, the error in these remainder terms is of the order of the model and the total energy considered $\mathfrak{E} _{\rm tot}$ is a good approximation of the mechanical energy of the concrete waves-structure system.
 
\subsubsection{Local energy balance of the fluid}\label{EnLocFluid}
If we define the local energy density $\mathfrak e$ and the local energy flux $\mathfrak f$ as 
\begin{equation}
\begin{cases}
\mathfrak e:= \frac{1}{2} \left( \zeta^2 + \frac{q^2}{h} + \kappa^2 \frac{({\rm{d}}_{r} q)^2}{h} \right),\\
\mathfrak f:= q\left( \kappa^2 \ddot{\zeta} + \nu \dot{\zeta} + \zeta + \frac{\eps}{2}\frac{q^2}{h^2} + \frac{P}{\varepsilon} \right),
\end{cases}
\end{equation}
then the local energy balance of the fluid is
\begin{equation}\label{eq:ConsOfTheEnergy}
\partial_t \mathfrak e + {\rm{d}}_{r} \mathfrak f = \frac{P}{\varepsilon} {\rm{d}}_{r} q - \nu \frac{|\dr q|^2}{h}
+ \varepsilon (\kappa^2 \mathfrak{r} - \nu \dr  \mathfrak{r}_\nu),
\end{equation}
where
$$
\mathfrak{r}:= \frac{\zeta \ddot{\zeta}}{h} - \frac{({\rm{d}}_{r} q)^3}{2h^2} - q \ddot{\zeta} \frac{\partial_r \zeta}{h}
\quad \text{and} \quad 
\mathfrak{r}_\nu = \left(\frac{q \dot{\zeta}\zeta}{h}\right).
$$
\begin{remark}
The case $\kappa = 0$, $\nu = 0$ and $P = 0$ finds out the exact local conservation of the energy of shallow-water equations.
\end{remark}
We obtain this local balance by multiplying equations of the radial Boussinesq-Abbott system {\rm \hyperlink{BAr}{(BAr)}} by $\zeta_{\rm i}$ and $\frac{q_{\rm i}}{h_{\rm i}}$.
The proof of this is detailled in appendix \ref{app:B}. 
 
\subsubsection{Energy of the solid}
The energy associated with the solid, denoted by $\mathfrak E _{\rm solid}$, is simply given by the sum of the potential and the kinetic energies:
\begin{equation}\label{def:EnSolid}
\mathfrak E_{\rm solid} = R^2 \left(\frac{1}{\varepsilon} m \delta + \frac{1}{2}\tau^2_{\rm buoy}\dot{\delta}^2\right),
\end{equation}
and by derivation with respect to the time
\begin{equation}
\frac{d}{dt}\mathfrak E_{\rm solid} = R^2 \left(\frac{1}{\varepsilon} m + \tau^2_{\rm buoy}\ddot{\delta}\right)\dot\delta.
\end{equation}
Thanks to Newton's motion law \eqref{Newton}, it is exactly
\begin{equation}\label{eq:EnSolid}
\frac{d}{dt}\mathfrak E_{\rm solid} = \int_0^R \frac{P_{\rm i}}{\varepsilon} \dot\delta dx + F_{\rm ext}  \dot\delta.
\end{equation}
As a consequence of the balance energy of the fluid \eqref{eq:ConsOfTheEnergy} and the balance energy of the solid \eqref{eq:EnSolid} the energy of the waves-structure system satisfies the following balance.
\begin{proposition}[Energy balance of the waves-structure system]\label{prop:ConsEn}
Any regular enough solution of the radial Boussinesq-Abbott system {\rm \hyperlink{BAr}{(BAr)}} coupled to the Newton's equation \eqref{Newton} satisfying the transmission's condition \eqref{eq:ConsMass} satisfies the following energy balance
\begin{equation}
\frac{d}{dt} {\mathfrak{E}}_{\rm tot} = R \left(\underline{\mathfrak{f}_{\rm e}} - \underline{\mathfrak{f}_{\rm i}} \right) 
+ F_{\rm ext}  \dot\delta
- \int_\mathcal{\mathbb{R}} \frac{ \nu  |\dr q|^2}{h}
+ \varepsilon \mathfrak{R}_{\kappa^2, \nu} ,
\end{equation}
where the total energy $\mathfrak{E}_{\rm tot}$ of the waves-structure system is given by
\begin{align*}
{\mathfrak{E}}_{\rm tot} &= {\mathfrak{E}}_{\rm fluid} + {\mathfrak{E}}_{\rm solid}\\
&= \frac{1}{2} \int_{\mathbb{R}} \left( \zeta^2 + \frac{q^2}{h} + \kappa^2 \frac{({\rm{d}}_{r} q)^2}{h} \right) + 2\pi R^2 \left(\frac{1}{\varepsilon} m \delta + \frac{1}{2}\tau^2_{buoy}\dot{\delta}^2\right),
\end{align*}
the remainder is defined by
$$
\varepsilon \mathfrak{R}_{\kappa^2, \nu} := \varepsilon \kappa^2 \left( \int_0^R \mathfrak r_{\rm i} + \int_R^{+\infty} \mathfrak r_{\rm e} \right) - \eps \nu \int_\mathbb{R} \mathfrak r_\nu.
$$
and the local energy $\mathfrak e$, the local energy flux $\mathfrak f$ and the local remainders $\mathfrak r_i, \mathfrak r_e, \mathfrak r_\nu $ are defined in Subsection \ref{EnLocFluid}.
\end{proposition}
\begin{proof}
By the local energy balance \eqref{eq:ConsOfTheEnergy} we have the relation
\begin{align*}
\frac{d}{dt}\mathfrak E _{\rm fluid} &=\int_0^R\frac{d}{dt}\mathfrak e_{\rm i} + \int_R^{+\infty} \frac{d}{dt}\mathfrak e_{\rm e} \\
& = -\int_0^R \left( {\rm{d}}_{r}\mathfrak f_{\rm i} - \frac{P_{\rm i}}{\varepsilon}{\rm{d}}_{r} q_{\rm i} \right) - \int_R^{+\infty} {\rm{d}}_{r} \mathfrak f_{\rm e} + \int_{\mathbb{R}} \frac{\nu  |\dr q|^2}{h} + \varepsilon \mathfrak{R}_{\kappa^2, \nu}\\
&= - \underline{\mathfrak f_{\rm i}} R - \int_0^R \frac{P_{\rm i}}{\varepsilon} \dot{\delta} + \underline{\mathfrak f_{\rm e}} R + \int_{\mathbb{R}} \frac{\nu  |\dr q|^2}{h} + \varepsilon \mathfrak{R}_{\kappa^2, \nu}.
\end{align*}
Adding this to \eqref{eq:EnSolid} leads to the result.
\end{proof}
In particular the condition $\underline{\mathfrak f_{\rm e}} - \underline{\mathfrak f_{\rm i}} = 0$ is sufficient to satisfy the energy balance \eqref{eq:ConsEnGlob}.

\subsubsection{Elliptic problem over the inner pressure}
Finally, thanks to Proposition \ref{prop:ConsEn}, we chose $\frac{P_{\rm i}(R)}{\eps}$ in \eqref{eq:pi(r)} such that the flux defined in Section \ref{EnLocFluid} satisfies $\underline{\mathfrak f_{\rm e}} - \underline{\mathfrak f_{\rm i} } = 0$. Indeed it satisfies the energy balance with an approximation at order $\eps \kappa^2$ if we assume that $\nu$ is of the same order as $\kappa^2$. Since the Boussinesq-Abbott system is already such an approximation of Euler's equations we consider that this approximation of the energy is good enough. We interpret this with the following proposition. 
\begin{proposition}\label{prop:pi(R)}
Let $(\zeta, q)$ be a regular enough solution of the radial Boussinesq-Abbott system {\rm \hyperlink{BAr}{(BAr)}} coupled to the Newton's equation \eqref{Newton} satisfying the transmission's condition \eqref{eq:ConsMass}. 
If the inner pressure $P_{\rm i}$ satisfies
\begin{equation}\label{eq:P_bord}
\frac{P_{\rm i}(R)}{\varepsilon}
= \underline{\zeta_{\rm e}} - \underline{\zeta_{\rm i}} + 
\nu(\underline{\dot{\zeta}_{\rm e}} - \underline{\dot{\zeta}_{\rm i}}) + 
\kappa^2 (\underline{\ddot{\zeta}_{\rm e}} - \underline{\ddot{\zeta}_{\rm i}}) + 
\frac{\varepsilon}{2} \left (\frac{| \underline{q_{\rm e}} |^2}{\underline{h_{\rm e}}^2} - \frac{| \underline{q_{\rm i}} |^2}{\underline{h_{\rm i}}^2}\right),
\end{equation}
then one has
\begin{equation}
\frac{d}{dt} {\mathfrak{E}}_{\rm tot} = F_{\rm ext}  \dot\delta
- \int_{\mathbb{R}} \frac{\nu  |\dr q|^2}{h}
+ \eps\left (\kappa^2\left ( \frac{\zeta \ddot{\zeta}}{h} - \frac{({\rm{d}}_{r} q)^3}{2h^2} - q \ddot{\zeta} \frac{\partial_r \zeta}{h}\right )
-\nu \frac{q \dot{\zeta}\zeta}{h}\right).
\end{equation}
\end{proposition}
If we take the equation \eqref{eq:pi(r)-pi(R)} that the inner pressure $P_{\rm i}$ verifies and Proposition \ref{prop:pi(R)} we obtain the expression of $P_{\rm i}$. If we replace the known quantities in $P_{\rm i}(R)$ we have:
\begin{proposition}[Expression of the inner pressure $P_{\rm i}$]\label{prop:Pi}
Let $(\zeta, q)$ be a regular enough solution of {\rm \hyperlink{BAr}{(BAr)}}, and $P_{\rm i}$ such that condition \eqref{eq:P_bord} is satisfied. Then the inner pressure is given for all $r \leq R$ by
\begin{equation}\label{eq:P_i}
 \frac{P_{\rm i}(r)}{\varepsilon}  = \frac{r^2-R^2}{2h_{\rm i}} \left( 
 \frac{\ddot{\delta}}{2}  - \frac{3 \varepsilon | \dot{\delta}|^2}{4 h_{\rm i}} \right)
+ \underline{\zeta_{\rm e}} - \underline{\zeta_{\rm i}} +
\nu(\underline{\dot{\zeta}_{\rm e}} - \dot{\delta}) +
\kappa^2 (\underline{\ddot{\zeta}_{\rm e}} - \ddot{\delta}) +
\frac{\varepsilon R^2 \dot{\delta}^2}{8}\left (\frac{1}{\underline{h_{\rm e}}^2} - \frac{1}{\underline{h_{\rm i}}^2}\right).
\end{equation} 
\end{proposition}
 
\begin{remark}\label{rk:EnBudg}
One could also consider the energy flux defined by
$$ \tilde{\mathfrak{f}} =
q\left( \kappa^2 \frac{\ddot{\zeta}}{h} + \nu \frac{\dot{\zeta}}{h} + \zeta + \frac{\eps}{2}\frac{q^2}{h^2} + \frac{P}{\varepsilon} \right),$$
and the balance of local energy is
\begin{equation}\label{eq:ConsOfTheEnergy-BeckLannes}
\partial_t \mathfrak e + {\rm{d}}_{r}  \tilde{\mathfrak{f}} = \frac{P}{\varepsilon} {\rm{d}}_{r} q - \nu \frac{|d_r q|^2}{h}+ \varepsilon \kappa^2 \tilde{\mathfrak{r}},
\end{equation}
where
$$ \tilde{\mathfrak{r}}:= - \frac{({\rm{d}}_{r} q)^3}{2h^2} - q \ddot{\zeta} \frac{\partial_r \zeta}{h}.$$
Therefore the conservation of the energy leads to choose 
$$\frac{P_{\rm i}(R)}{\varepsilon} = \underline{\zeta_{\rm e}} - \underline{\zeta_{\rm i}} + \nu \left( \underline{\dot{\zeta_{\rm e}}} - \dot{\delta} \right) + \kappa^2 \left(\frac{\underline{\ddot{\zeta_{\rm e}}}}{\underline{h_{\rm e}}} - \frac{\ddot{\delta}}{\underline{h_{\rm i}}} \right) + \frac{\varepsilon R^2}{8}\dot{\delta}^2\left( \frac{1}{\underline{h_{\rm e}}^2} - \frac{1}{\underline{h_{\rm i}}^2} \right).$$
This is the choice in 1D in \cite{BeckLannes} for $\nu = 0$. 
In our work, we have chosen \eqref{eq:ConsOfTheEnergy} instead of \eqref{eq:ConsOfTheEnergy-BeckLannes} in order to simplify the proof of well-posedness (see Remark \ref{rk:WPSL_ZetaOvH}). 
\end{remark}
Now we can modify the Newton's equation \eqref{Newton} to avoid to compute the inner pressure $P_{\rm i}$.
\subsection{Eliminating the Lagrange multiplier}\label{sec:NML}
In \eqref{Newton} the Lagrange multiplier $\frac{P_{\rm i}}{\varepsilon}$ can be eliminated by injecting the elliptic equation over the pressure \eqref{eq:P_i}. In the new equation obtained from this, we will obtain the intertia of the floating body. This inertia will be a part of the pressure exerced by the water on the bottom of the solid.
\begin{proposition}\label{prop:AddedMass}
Let $(\zeta, q)$ be a smooth enough solution of the radial Boussinesq-Abbott system {\rm \hyperlink{BAr}{(BAr)}} and $(\delta, \dot \delta, P_{\rm i})$ solution of the Newton's equation \eqref{Newton}, such that $P_{\rm i}$ satisfies the transmission condition \eqref{eq:P_bord}. The vertical motion of the cylinder is solution of the second order ODE
\begin{equation}\label{AddedMass}
\tau_\kappa^2( \varepsilon \delta) \ddot{\delta}  + \nu \dot{\delta}  + \delta -  \varepsilon a( \varepsilon \delta, \underline{\zeta_{\rm e}}) \, \dot{\delta}^2
= \kappa^2 \underline{\ddot{\zeta_{\rm e}}} + \nu \underline{\dot{\zeta_{\rm e}}} + \underline{\zeta_{\rm e}} + F_{\rm ext},
\end{equation}
where
$$
\tau_\kappa^2(\varepsilon \delta) := \tau_{\rm buoy}^2 + \kappa^2 +  \frac{ R^2}{8 h_{\rm i}}  \quad \text{ with } \quad h_{\rm i} = 1 + \varepsilon \delta,
$$
and
$$
a(\varepsilon \delta, \underline{\zeta_{\rm e}}) :=\frac{R^2}{16 h_{\rm i}^2} - \frac{R^2}{8\underline{h_{\rm e}}^2}. 
$$ 
\end{proposition}

It is quite similar to the buoyancy factor found in 1D, up to the choice of the trace of the interior pressure that we did (see Remark \ref{rk:EnBudg}, and Remark 2.5 in \cite{BeckLannes}). When $\kappa=0$, one recovers the added-mass coefficient given in \cite{Bocchi}.
Furthermore, $\tau_\kappa^2(\eps\delta) > \tau_{\rm buoy}^2$, which means in particular that, even for a very light object, the second order term $\ddot{\delta}$ cannot be neglected. 
We name this equation \textit{the modified Newton's equation}.
This equation is non-linear and contains second order terms in $\kappa^2 \underline{\ddot{\zeta_{\rm e}}}$. 
\begin{proof}
The method to obtain the modified Newton's equation consists in two steps. Firstly we inject the Archimede's principle into equation \eqref{Newton} and then we replace the inner pressure $P_{\rm i}$ by its expression given in \eqref{eq:P_i}.
The Newton's equation at the equilibrium, known as the Archimede's principle, gives
\begin{equation}
\frac{1}{\varepsilon} m = \frac{1}{| \mathcal{I} |} \int_{\mathcal{I}} (\delta - \zeta_{\rm i} ) = \delta - \frac{1}{| \mathcal{I} |} \int_{\mathcal{I}} \zeta_{\rm i}. 
\end{equation}
Injecting this into the Newton's equation \eqref{Newton} brings us to
\begin{equation}\label{eq:NewPArchi}
\tau_{\rm buoy}^2 \ddot{\delta} + \delta = \displaystyle \frac{1}{| \mathcal{I} | }\int_{\mathcal{I}} \left( \frac{P_{\rm i}}{\varepsilon} + \zeta_{\rm i}\right) + F_{\rm ext}.
\end{equation}
Nonetheless thanks to Proposition \ref{prop:Pi} the inner pressure $P_{\rm i}$ reads
\begin{equation}
\frac{P_{\rm i}(r)}{\varepsilon}  = \frac{r^2-R^2}{2h_{\rm i}} \left( 
\frac{\ddot{\delta}}{2}  - \frac{3 \varepsilon | \dot{\delta}|^2}{4 h_{\rm i}} \right) + 
\underline{\zeta_{\rm e}} - \underline{\zeta_{\rm i}} + 
\nu(\underline{\dot{\zeta}_{\rm e}} - \dot{\delta}) + 
\kappa^2 (\underline{\ddot{\zeta}_{\rm e}} - \ddot{\delta}) + 
\frac{\varepsilon R^2 \dot{\delta}^2}{8}\left (\frac{1}{\underline{h_{\rm e}}^2} - \frac{1}{\underline{h_{\rm i}}^2}\right).
 \end{equation}
Therefore we can compute the integration of $\frac{P_{\rm i}}{\eps} + \zeta_{\rm i}$. Indeed, considering that the solid is a cylinder, $\zeta_{\rm i}$ does not depend on $r$, the computation of the integral leads to
\begin{align*}
\frac{1}{|\mathcal{I}|}\int_\mathcal{I} \left(\frac{P_{\rm i}(r)}{\varepsilon} + \zeta_{\rm i}\right) &= \int_0^R \left(\frac{P_{\rm i}(r)}{\varepsilon} + \zeta_{\rm i}\right)r{\rm d}r \\
&= -\frac{1}{h_{\rm i}}\left( 
 \frac{\ddot{\delta}}{2}  - \frac{3 \varepsilon | \dot{\delta}|^2}{4 h_{\rm i}} \right)\frac{R^2}{4} +\underline{\zeta_{\rm e}} + \nu(\underline{\dot{\zeta}_{\rm e}} - \dot{\delta}) + 
\kappa^2 (\underline{\ddot{\zeta}_{\rm e}} - \ddot{\delta}) + 
\frac{\varepsilon R^2 \dot{\delta}^2}{8}\left (\frac{1}{\underline{h_{\rm e}}^2} - \frac{1}{\underline{h_{\rm i}}^2}\right).
\end{align*} 
Injecting this into equation \eqref{eq:NewPArchi} leads to the modified Newton's equation \eqref{AddedMass}.
\end{proof}
\begin{remark}
If we had choosen $\underline{P_{\rm i}}$ as in \cite{BeckLannes} (see Remark \ref{rk:EnBudg}), the modified Newton's equation would be
\begin{equation}
\tilde{\tau}_\kappa^2( \varepsilon \delta) \ddot{\delta}  + \nu \dot{\delta}  + \delta -  \varepsilon \tilde{a}( \varepsilon \delta, \underline{\zeta_{\rm e}}) \, \dot{\delta}^2
= \kappa^2 \frac{\underline{\ddot{\zeta_{\rm e}}}}{\underline{h_{\rm e}}} + \nu \underline{\dot{\zeta_{\rm e}}} + \underline{\zeta_{\rm e}} + F_{\rm ext},
\end{equation}
with $$
\tilde{\tau}_\kappa^2(\varepsilon \delta) := \tau_{\rm buoy}^2 + \kappa^2\frac{1}{\underline{h_{\rm e}}} +  \frac{ R^2}{8 h_{\rm i}},
$$ 
but it respects $\tilde{\tau}_\kappa^2(\varepsilon \delta) - \tau_{\kappa}^2(\eps\delta) = \mathcal{O}(\eps\kappa^2)$. However it would induce another nonlinearity with the factor $\frac{\underline{\ddot{\zeta_{\rm e}}}}{\underline{h_{\rm e}}}$, which is something that we want to avoid (see Remark \ref{rk:WPSL_ZetaOvH}).
\end{remark}
One has to mention that $\tau_\kappa^2(\varepsilon\delta)$ does not contain the whole ''added mass effect''. Indeed to match the Newton's equation we consider that every coefficient in front of $\ddot{\delta}$ is part of the intertia brought by the pressure. 
Hence another part of it will come from the computation of $\kappa^2 \underline{\ddot{\zeta_{\rm e}}}$. It will be detailed in the next Subsection \ref{sec:HidEq}.
 
\subsection{Hidden equation}\label{sec:HidEq}
In this section we want to get two points. On one hand we aim at getting a conservative system of PDE's with a bounded source term from the radial Boussinesq-Abbott system {\rm\hyperlink{BAr}{(BAr)}}. It will enable us to show Theorem \ref{thmf:Aug}. Another motivation to do this is to try to fit the 1D case. 
On the other hand we desire to have another equation over $\kappa^2 \underline{\ddot{\zeta}_{\rm e}}$, to link it with the modified Newton's equation \eqref{AddedMass}. 
If we catch up an equation of this kind then $\left(\delta, \underline{\zeta_{\rm e}}\right)$ will satisfy an ODE system of order 2. It is essential to consider that $\kappa >0$ since this approach is useless in the hyperbolic case. 
Then, by \eqref{Qint}, we have access to $q_{\rm i}$. By equation \eqref{eq:pi(r)} it also allows to get $P_{\rm i}$. Therefore such a system of ODE's describes the behavior of all quantities of interest in the inner domain $\mathcal{I}$.
This subsection is the third point (\ref{EqOnTrace}) of the method presented at the beginning of Section \ref{sec:2}
 
\subsubsection{Inversion of dispersive operator}\label{sec:InvDispOpe}
 
To obtain a conservative system of PDE's plus bounded perturbations from the radial Boussinesq-Abbott {\rm\hyperlink{BAr}{(BAr)}} it is necessary to reverse the dispersive operator $(1-\kappa ^2 \partial_r {\rm{d}}_{r})$. 
To do so we introduce regularizing operators $\mathfrak{R}_0$ and $\mathfrak{R}_1$. They are defined for $\kappa >0$ as the inverse operators of $(1-\kappa ^2 \partial_r {\rm{d}}_{r})$ and $(1 - \kappa ^2 {\rm{d}}_{r} \partial_r)$ with homogeneous Dirichlet's and Neumann's boundary conditions at $r = R$, that is for any $f \in L^2_r$
\begin{equation}\label{def:R0}
u = \mathfrak{R}_0 f \Leftrightarrow
\begin{cases}
 (1 - \kappa^2 \partial_r {\rm{d}}_{r})u = f   & \text{in } (R, \infty),\\ 
  \underline{u} = 0 & \text{at } r=R, \\
  \underset{r \rightarrow + \infty }{\lim} u = 0, &
\end{cases}
\end{equation}
and
\begin{equation}\label{def:R1}
v = \mathfrak{R}_1 f \Leftrightarrow 
\begin{cases}
 (1 - \kappa^2 {\rm{d}}_{r} \partial_r )v = f   & \text{in } (R, \infty),\\ 
  \underline{\partial_r v} = 0 & \text{at } r=R, \\
  \underset{r \rightarrow + \infty }{\lim} v = 0. & 
\end{cases}
\end{equation}
We call these operators regularizing operators because they are defined from $H^n_r$ into $H^{n+2}_r$, as it will be explained by Lemmata \ref{lem:OpDisp} and \ref{lem:R1}.
In a first place we need to study these operators. To do so we introduce the modified Bessel functions, defined in \cite{A&S} as following. 
\begin{defn}[Modified Bessel functions]\label{def:K}
Let $\alpha \in \mathbb{R}$. We define the modified Bessel functions of first kind and second kind, $I_\alpha$ and $K_\alpha$, as the solutions of the ODE of unknown $y$
$$
\forall x > 0, \, x^2\frac{d^2 y}{dx^2} +x \frac{dy}{dx} - (x^2 + \alpha^2)y = 0.
$$
$I_\alpha$ is the solution such that $I_\alpha(x) \underset{x\to \infty}{\longrightarrow} +\infty$ and $K_\alpha$ is the solution such that $K_\alpha(x) \underset{x\to \infty}{\longrightarrow} 0$.
\end{defn}
Thanks to these functions one considers the following lemma: 
\begin{lemma}\label{lem:R0Homog}
The elliptic problem 
\begin{equation}\label{eq:R0Homog}
\begin{cases}
 (1 - \kappa^2 \partial_r {\rm{d}}_{r}) K = 0 & \text{for } r > R, \\ 
  \underline{K} = 1 & \text{at } r = R, \\
  \underset{r \rightarrow + \infty }{\lim} K = 0. &
\end{cases}
\end{equation}
admits a unique solution given by $K : r \mapsto  \frac{K_1(r/\kappa)}{K_1(R/\kappa)}$.
\end{lemma} 
\begin{proof}
A change of parameters in $x = r/\kappa$ in the elliptic problem \eqref{eq:R0Homog} gives that the solution satisfies
\begin{equation}
x^2 \partial_x ^2 K + x \partial_x K -(x^2 + 1) K = 0.
\end{equation}
The solutions of this equation are exactly given by the modified Bessel functions $I_1$ and $K_1$. In particular we keep only $K_1$ which is in $L^2_r$. Therefore the solution is given by $K: r \mapsto  \frac{K_1(r/\kappa)}{K_1(R/\kappa)}$.
\end{proof}
\begin{remark}
In 1D the solutions are given by the exponential function instead of modified Bessel functions, see \cite{BeckLannes}.
\end{remark}
Therefore by a variation of parameters over the solution of the homogeneous problem given by Lemma \ref{lem:R0Homog} one obtain the following expression for $\mathfrak R_0$, and the following properties.
\begin{lemma}\label{lem:OpDisp}
Let $u\in L^2_r$. Then there exists a unique solution to the problem \eqref{def:R0} given by
\begin{align*}\label{prop:exprR0}
\mathfrak R_0 u : r \mapsto &\frac{1}{\kappa^2}  \left( \int_r^{+\infty} r' K_1(r'/\kappa) u(r') \rm{dr'} \right) I_1(r/\kappa) - \frac{1}{\kappa^2} \left( \int_R^r r' I_1(r'/\kappa) u(r') \rm{dr'} \right) K_1(r/\kappa) \\&- \frac{1}{\kappa^2}  \left( \int_R^{+\infty} r' K_1(r'/\kappa) u(r') \rm{dr'} \right) I_1(R/\kappa).
\end{align*}
Furthermore the operator $\mathfrak R_0$ satisfies:
\begin{enumerate}
\item For any $u \in L^2_r$, one has
\begin{equation}
\| \mathfrak R_0 u \|_{L^2_r}^2 + \kappa^2 \| \dr \mathfrak R_0 u \|_{L^2_r}^2 \leq \| u \|_{L^2_r}^2.
\end{equation}
\item For any $u \in L^\infty_r$, one has
$$ \| \mathfrak R_0 u \|_{L^\infty_r} \lesssim \| u \|_{L^\infty_r}
\quad \text{and} \quad
 \| \partial_r \mathfrak R_0 u\|_{L^\infty_r} \lesssim \frac{1}{\kappa} \|u\|_{L^\infty_r}.$$ 
\item For any $u, v \in L^\infty_r$, one has
\begin{equation}
|\mathfrak{R}_0(uv) | \leq \|u\|_{L^\infty_r} |\mathfrak{R}_0 v|.
\end{equation}
\end{enumerate}
\end{lemma}
In particular we recover the properties of $\mathfrak{R}_0$ from \cite{BeckLannes}, that is to say $\mathfrak{R}_0, \, \partial_r \mathfrak{R}_0$ are bounded operators of both $L^2_r$ and $L^\infty_r$.
\begin{proof}
By definition for any $u \in L^2_r$, 
$$
\mathfrak R_0 u - \kappa^2 \partial_r \dr \mathfrak R_0 u = u.
$$
We multiply this equation by $r\mathfrak R_0 u$, then we integrate on $(R, +\infty)$ to obtain
\begin{equation}
\| \mathfrak R_0 u \|_{L^2_r}^2 + \| \kappa \dr \mathfrak R_0 u \|_{L^2_r}^2 \leq \int_R^{+\infty} (u\mathfrak R_0 u) r dr.
\end{equation}
With the Cauchy-Schwartz inequality the first point is proved. \\
 
By using the expression of $\mathfrak R_0$, one has for any $u \in L^\infty_r$
\begin{align*}
\| \mathfrak R_0 u \|_{L^\infty_r} \leq & \|u\|_{L^\infty_r} \left[   \| \left(  \frac{1}{\kappa^2} \int_r^{+\infty} K_1(r'/\kappa) \rm{dr'} \right) I_1(r/\kappa) \|_{L^\infty_r} + \right. \| \left(\frac{1}{\kappa^2} \int_R^r I_1(r'/\kappa) \rm{dr'} \right) K_1(r/\kappa)\|_{L^\infty_r} \\&\left.  +   \| \left( \frac{1}{\kappa^2} \int_R^{+\infty} K_1(r'/\kappa) \rm{dr'} \right) I_1(R/\kappa)\|_{L^\infty_r}\right].
\end{align*}
With the following lemma, one obtains $\| \mathfrak R_0 u \|_{L^\infty_r} \lesssim \| u \|_{L^\infty_r}$.
\begin{lemma}\label{lem:ineg}
The quantities
$$
\|  \left( \frac{1}{\kappa^2} \int_r^{+\infty} K_1(r'/\kappa) \rm{dr'} \right) I_1(r/\kappa) \|_{L^\infty_r} 
\quad
\text{and} \quad
\| \left(\frac{1}{\kappa^2} \int_R^r I_1(r'/\kappa) \rm{dr'} \right) K_1(r/\kappa)\|_{L^\infty_r},
$$
are bounded uniformaly with respect to $\kappa$ and $R$.
\end{lemma}
\begin{proof}
To prove this lemma one needs the  following Debye's asymptotic behavior of modified Bessel function for large $z \gg 1$, see \cite{A&S}:
\begin{equation}\label{eq:AsympBesselReels}
\begin{cases}
I_0(z) = \frac{e^z}{\sqrt{2 \pi z}} \left( 1 + \frac{1}{8z}+ \mathcal{O}(z^{-2}) \right), \\
I_1(z) = \frac{e^z}{\sqrt{2 \pi z}} \left( 1 - \frac{3}{8z} + \mathcal{O}(z^{-2}) \right), \\
K_0(z) = \frac{e^{-z}}{\sqrt{2 \pi^{-1} z}} \left( 1 - \frac{1}{8z}+ \mathcal{O}(z^{-2}) \right), \\
K_1(z) = \frac{e^{-z}}{\sqrt{2 \pi^{-1} z}} \left( 1 + \frac{3}{8z} + \mathcal{O}(z^{-2}) \right).
\end{cases}
\end{equation}

By a change of variable  in $z = r/\kappa$, the first term that we want to bound is the supremum of the function $f$:
$$
 \|  \left( \frac{1}{\kappa^2} \int_r^{+\infty} r' K_1(r'/\kappa) {\rm dr'} \right) I_1(r/\kappa) \|_{L^\infty_r} \leq
 \sup_{z \geq 0} f(z),
$$ 
where $f(z)= z K_0(z) I_1(z)$.
Firstly, according to \cite{A&S}, the modified Bessel function $K_0$ behaves as $\ln(z)$ in $0$ and therefore $f(z) \to 0$ when $z \to 0$.
Since modified Bessel functions are regular, $f$ is differentiable and for any $z > 0$ and one deduces for large $z \gg 1$ from the asymptotic behaviors \eqref{eq:AsympBesselReels} that 
$$
\lim_{z \to 0} f(z) \to 0, \quad
\lim_{z \to \infty} f(z) \to \frac{1}{2}
\quad \text{and} \quad
f'(z) = \frac{1}{16z} + \mathcal{O}(z^{-2}).
$$

For large enough $z$, the continuous function $f$ increases. It shows that $f$ admits an upper finite bound, and that independently of $\kappa$ and $R$, $\|  \left( \frac{1}{\kappa^2} \int_r^{+\infty} r' K_1(r'/\kappa) {\rm dr'} \right) I_1(r/\kappa) \|_{L^\infty_r}$ is bounded.
This is consistent with numerical evidence that claims that $f$ is increasing in fact for all $z>0$ and converges towards $\frac{1}{2}$ (see Figure \ref{fig:drR0_infty}).

Secondly, as previously by a change of variable	
$$
 \|  \left( \frac{1}{\kappa^2} \int_R^{r} I_1(r'/\kappa) {\rm dr'} \right) K_1(r/\kappa) \|_{L^\infty_r} \leq
 \sup_{z \geq 0} g(z),
 $$
where $g(z)= z I_0(z) K_1(z)$.
By the well-known identity $K_0 I_1 + K_1 I_0 = z^{-1}$ for all $z > 0$, one has
$$
\forall z > 0, \, g(z) = 1 - f(z).
$$
In particular $g$ admits finite limit in $0$ and $+\infty$, which are respectivly $0$ and $\frac{1}{2}$. Furthermore one deduces that $g$ is decreasing for large $z$. Thus its supremum is reached at a finite point $z_\ast$ and $ \|\left( \frac{1}{\kappa^2} \int_R^{r} I_1(r'/\kappa) {\rm dr'} \right) K_1(r/\kappa) \|_{L^\infty_r}  \leq g(z^*)$ independently of $\kappa$ or $R$. It is also consistent with the numerical evidence that claims that $g$ is decreasing for all $z >0$ (see Figure \ref{fig:drR0_infty})	.
\end{proof}
Also the expression of $\mathfrak R_0 u$ is differentiable and 
\begin{align*}
\partial_r \mathfrak R_0 u &= \frac{1}{\kappa^3}  \left( \int_r^{+\infty} r' K_1(r'/\kappa) u(r') \rm{dr'} \right) I_1'(r/\kappa) - \frac{1}{\kappa^3} \left( \int_R^r r' I_1(r'/\kappa) u(r') \rm{dr'} \right) K_1'(r/\kappa) \\& - \frac{2r}{\kappa^2}K_1(r/\kappa)I_1(r/\kappa)u(r).
\end{align*}
As previously one can show that
$$
\| \frac{2r}{\kappa}K_1(r/\kappa)I_1(r/\kappa) \|_{L^\infty_r} \leq \sup_{z \geq 0} k(z) \text{ where } k(z)= 2z  I_1(z) K_1 (z),
$$
and
$$
\lim_{z \to 0} k(z) \to 0, \quad
\lim_{z \to \infty} k(z) \to 1
\quad \text{and} \quad
k'(z) = \frac{3}{4z} + \frac{2}{32 z} + O(z^{-3}).
$$
It proves that $k$ is uniformly bounded with respect to $\kappa$ and $R$ (see Figure \ref{fig:drR0_infty}).
\begin{figure}[h!]
\includegraphics[scale=0.8]{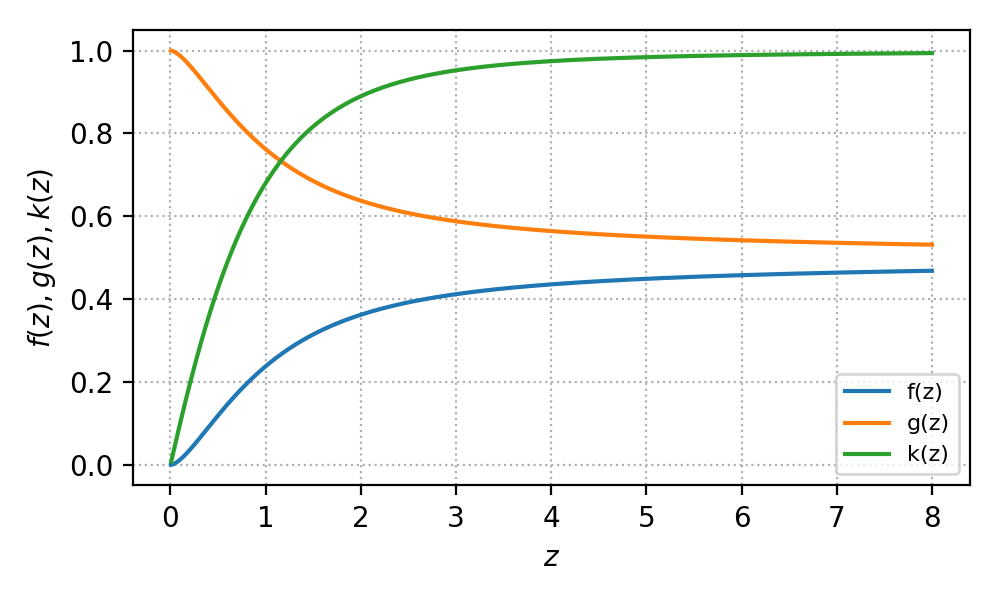}
\caption{Illustration of $f$, $g$, $k$ (defined in the proof of Lemma \ref{lem:ineg} and above	), with $R = 1$ and $\kappa = 0.1$.}
\label{fig:drR0_infty}
\end{figure}
 
The last point of the lemma is trivial from the expression of $\mathfrak{R}_0(uv)$ for any $u, v \in L^\infty_r$.
\end{proof}
By the same reasoning over $\mathfrak{R}_1$, one has that $\mathfrak{R}_1$ and $\partial_r \mathfrak{R}_1$ are bounded operators of both $L^2_r$ and $L^\infty_r$, and the following lemma stands.
\begin{lemma}\label{lem:R1} 
One has the following properties:
\begin{enumerate}
\item Let $u \in L^2_r$, then the problem \eqref{def:R1} admits a unique solution given for $r \geq R$ by
\begin{align*}
\mathfrak{R}_1 u(r) =& \left[ I_1\left(\frac{R}{\kappa}\right) \int_R^{+\infty} \frac{\rho}{\kappa} K_0(\rho/\kappa) u(\rho) d\rho + 2 R \, I_0(R/\kappa)K_0(R/\kappa) u(R)\right]\frac{1}{K_1(R/\kappa)} \\& + K_0(r/\kappa) \int_R^{r} \frac{s}{\kappa} I_0(s/\kappa) u(s) ds + I_0 (r/\kappa)\int_r^{+\infty} \frac{s}{\kappa} K_0(s/\kappa) u(s) ds.
\end{align*}
\item For any $u \in L^2_r$,
\begin{equation}
\| \mathfrak R_1 u \|_{L^2_r}^2 + \kappa^2 \| \partial_r \mathfrak R_1 u \|_{L^2_r}^2 \leq \| u \|_{L^2_r}^2.
\end{equation}
\item For any $u \in L^\infty_r$,
$$ \| \mathfrak R_1 u\|_{L^\infty_r} \lesssim \|u\|_{L^\infty_r}
\quad \text{and} \quad
\| \partial_r \mathfrak R_1 u\|_{L^\infty_r} \lesssim {\kappa}^{-1} \|u\|_{L^\infty_r}. $$ 
\end{enumerate}
\end{lemma}
 
Finally operators $\mathfrak R_0$ and $\mathfrak R_1$ satisfy the following commutation formula
\begin{lemma}\label{lem:commut1}
For any $u \in L^2_r$, one has
\begin{equation}
\mathfrak R_0 \partial_r u = \partial_r \mathfrak R_1 u.
\end{equation}
\end{lemma}
\begin{proof}
Let $u \in H^1_r$. Then $\mathfrak R_0 \partial_r u$ is defined by
\begin{equation}
\begin{cases}
 (1 - \kappa^2 \partial_r {\rm{d}}_{r}) (\mathfrak R_0 \partial_r u) = \partial_r u   & \text{in } (R, \infty),\\ 
  \underline{(\mathfrak R_0 \partial_r u)} = 0 & \text{at } r=R.
  \end{cases}
\end{equation}
Furthermore applying $\partial_r$ to the definition of $\mathfrak R_1 u$ leads to 
\begin{equation}
\begin{cases}
 (1 - \kappa^2 \partial_r {\rm{d}}_{r} ) \partial_r (\mathfrak R_1 u) = \partial_r u   & \text{in } (R, \infty),\\ 
  \underline{\partial_r (\mathfrak R_1 u)} = 0 & \text{at } r=R. 
\end{cases}
\end{equation}
By unicity of the solution of $\mathfrak{R}_0$ and $\mathfrak{R}_1$ we obtain the lemma.
\end{proof}
In the same way we show that
\begin{lemma}\label{lem:commut2}
For any $u \in L^2_r$, one has
\begin{equation}
\dr \mathfrak R_0 u = \mathfrak R_1 \dr u + \kappa \underline{u} \frac{K_0(r/\kappa)}{K_1(R/\kappa)}.
\end{equation}
\end{lemma}
The proof of this lemma only consists in verifying that ${\rm d}_r \mathfrak{R}_0 u - \kappa \underline{u} \frac{K_0(r/\kappa)}{K_1(R/\kappa)}$ is a solution of $(1-\kappa^2{\rm d}_r\partial_r)v = {\rm d}_r u$ with boundary condition $\underline{\partial_r v} = 0$.
Inter alia it allows to show the following proposition by induction on $s$.
\begin{proposition}\label{prop:commut}
Let $s \in \mathbb{N}$ and $p \geq 2$. Let $u \in W^{s,p}_r$. Then one has
\begin{equation}
\| \mathfrak R_i u \|_{W^{s,p}_r} \lesssim \| u \|_{W^{s,p}_r}
\quad
\text{ and }
\quad
\| \mathfrak R_i u \|_{W^{s+1,p}_r} \lesssim {\kappa}^{-1} \| u \|_{W^{s,p}_r},
\end{equation}
for $i = 0, 1$.
\end{proposition}
By induction for $s \in \mathbb{N}$, thanks to the previous lemmata \ref{lem:commut1} and \ref{lem:commut2} we show that the proposition is true for $p = 2$. The lemmata also allows to show the proposition in $W^{s, \infty}_r$ and by Sobolev's interpolations we show that it is true in $W^{s, p}_r$ for any $p \geq 2$. 
 
\subsubsection{Hidden equation at the boundary}
The second step is to apply the dispersive operator $\mathfrak R_0$ to obtain a conservative system plus bounded perturbations from the radial Boussinesq-Abbott system {\rm \hyperlink{BAr}{(BAr)}}. In fact thanks to the commutation Lemma \ref{lem:commut1} between $\mathfrak R_0$ and $\mathfrak R_1$ one has the following lemma over a more general system.
\begin{lemma}\label{lem:HidEq}
For any $f, g$ regular enough the system of equations of unknown $(\zeta, q)$ given for $r > R$ by
\begin{empheq}[left=\empheqlbrace]{align}
&\partial_t \zeta + {\rm d}_r\, q=0,\label{BAinit1}  \\
&(1-\kappa ^2 \partial_r {\rm{d}}_{r}) \partial_t q + \partial_r f +g =0 , \label{BAinit} 
\end{empheq}
is equivalent to 
\begin{equation}\label{eq:3.34}
\begin{cases}
\kappa^2 {\ddot{\zeta}} +  f + \kappa^2 {{\rm{d}}_{r} K} \underline{\dot{q}} =   \mathfrak{R}_1 f - \kappa^2 {{\rm{d}}_{r} \mathfrak{R}_0 g} ,\\
\partial_t q + \partial_r \mathfrak{R}_1 f + \mathfrak{R}_0 g = \underline{\dot{q}} K.
\end{cases}
\quad \forall r \in (R, + \infty),
\end{equation}
where the dispersive boundary layer $K$ is defined in Lemma \ref{lem:R0Homog}.
\end{lemma}
In 1D in \cite{BeckLannes}, $K$ is replaced by an exponential and the same kind of lemma stands. 
The radial Boussinesq-Abbott system {\rm\hyperlink{BAr}{(BAr)}} is just the system \eqref{BAinit1}--\eqref{BAinit} with specific choice 
$$
f = \zeta + \eps \left( \frac{\zeta^2}{2} + \frac{q^2}{h} \right) - \nu \dr q,
$$
and
$$
g = \eps \frac{q^2}{rh} - (\nu \partial_r(\ln(h)) + \partial_r  \nu  )\dr q.
$$ 
In the 1D case with $\nu = 0$ (see \cite{BeckLannes}), one has $g = 0$ and $f = \zeta + \eps \left( \frac{\zeta^2}{2} + \frac{q^2}{h} \right)$. 
Indeed the nonviscous term corresponding to $g$ in the radial Boussinesq-Abbott system {\rm \hyperlink{BAr}{(BAr)}} comes from the radial effect in the definition of ${\rm{d}}_r$. 
However because of the dispersive operators the system of equation \eqref{eq:3.34} is nonlocal and requires the computation of the quantities in the domain $r> R$, contrary to the modified Newton's equation \eqref{AddedMass}. 
\begin{proof}
Applying $\mathfrak{R}_0$ to equation \eqref{BAinit} for $r > R$, it leads to
\begin{equation}\label{traceproof1}
\partial_t q - \underline{\partial_t q} K + \mathfrak{R}_0 \partial_r f + \mathfrak R _0 g = 0,
\end{equation}
where $\underline{\partial_t q} K$ belongs to the kernel of the operator $(1-\kappa^2 \partial_r \dr)$ with homogeneous Dirichlet boundary condition (see Lemma \ref{lem:R0Homog}). Then applying Lemma \ref{lem:commut1} to equation \eqref{traceproof1} gives the wanted equation
\begin{equation}\label{traceproof2}
\partial_t q + \partial_r \mathfrak{R}_1 f  + \mathfrak R _0 g = \underline{\dot{q}} K.
\end{equation}
The next step is to apply $-\kappa^2 {\rm{d}}_{r}$ on \eqref{traceproof2} and to use the fact that $\partial_t \zeta + \dr q = 0$ to get
\begin{equation}\label{traceproof3}
\kappa^2 {\ddot{\zeta}} +  f + \kappa^2 {{\rm{d}}_{r} K} \underline{\dot{q}} =   \mathfrak{R}_1 f -\kappa^2 {{\rm{d}}_{r} \mathfrak{R}_0 g}.
\end{equation}
To come back to the system \eqref{BAinit1}, \eqref{BAinit} one can force $\partial_t \zeta + {\rm d}_{r} q = 0$ and integrate equation \eqref{traceproof3} in space. Then applying $(1 - \kappa^2 \partial_r {\rm{d}}_{r})$ brings to \eqref{BAinit}. 
\end{proof} 
The final step is to use this lemma over the radial Boussinesq Abbott system {\rm \hyperlink{BAr}{(BAr)}} to obtain an ODE on the trace satisfied by the surface elevation. Afterwards we will combine it with the modified Newton's equation \eqref{AddedMass} to show that the displacement of the solid $\delta$ and the surface elevation $\zeta$ satisfy an ODE of order 2 with a nonlocal term due to the dispersion. \\
\indent Let $(\zeta, q)$ be a solution of the radial Boussinesq-Abbott system {\rm \hyperlink{BAr}{(BAr)}} and $(\delta, \dot{\delta})$ be a solution of the added mass equation \eqref{AddedMass} such that the transmission condition \eqref{eq:ConsMass} is satisfied.
Applying Lemma \ref{lem:HidEq} leads to the \textit{hidden equation}, defined for $r > R$ as
\begin{align*}\label{eq:HidEq}
\kappa^2 {\ddot{\zeta_{\rm e}}} + \zeta_{\rm e} + \eps \left( \frac{\zeta_{\rm e}^2}{2} + \frac{q_{\rm e}^2}{h_{\rm e}} \right) - \nu \dr q_{\rm e} + \kappa^2 {{\rm{d}}_{r} K} \underline{\dot{q}_{\rm e}} = & \mathfrak{R}_1 \left (\zeta_{\rm e} + \eps \left( \frac{\zeta_{\rm e}^2}{2} + \frac{q_{\rm e}^2}{h_{\rm e}} \right) - \nu \dr q_{\rm e} \right )\\
& - \kappa^2 {{\rm{d}}_{r} \mathfrak{R}_0 \left (\eps \frac{q_{\rm e}^2}{rh_{\rm e}} - (\nu \partial_r(\ln(h_{\rm e})) + \partial_r \nu   )\dr q_{\rm e}\right )}.
\end{align*}
Furthermore, through the transmission condition \eqref{eq:ConsMass}, the boundary condition reads $\underline{\partial_t q_{\rm e}} = -\frac{R}{2}\ddot{\delta}$. From this we obtain a second equation linking $\ddot{\delta}$ to other quantities of interest and in particular to $\kappa^2 \underline{\ddot{\zeta}_{\rm e}}$.
In particular, taking the previous equation at the trace $r \to R$ leads to the following lemma.

\begin{lemma}\label{lem:EqTrace}
Let $(\zeta, q)$ be a regular enough solution of the radial Boussinesq-Abbott system {\rm \hyperlink{BAr}{(BAr)}} and $(\delta, \dot{\delta})$ be a solution of the added mass equation \eqref{AddedMass}, such that the transmission condition \eqref{eq:ConsMass} is satisfied. Then the following equation is satisfied
\begin{equation}\label{EqTrace}
\kappa^2 \underline{\ddot{\zeta}_{\rm e}} + \nu \underline{ \dot{\zeta}_{\rm e} } + \underline{\zeta_{\rm e}} + \eps \underline{\mathfrak{f}} = \underline{f_{\rm{hyd}}} - \kappa \ddot{\delta} G(R),
\end{equation}
with the following quantities:
\begin{itemize}
\item $\mathfrak{f} := \frac{\zeta_{\rm e}^2}{2} + \frac{|q_{\rm e}|^2}{h_{\rm e}} $,
\item $f_{\rm{hyd}} := \mathfrak{R}_1 \left( \zeta_{\rm e} - \nu \dr q_{\rm e} + \varepsilon \mathfrak{f} \right) + \varepsilon \kappa^2 {\rm{d}}_{r} \mathfrak{R}_0 \left( \frac{q_{\rm e}^2}{r h_{\rm e}} \right) - \nu \kappa^2{\rm{d}}_{r} \mathfrak{R}_0 \partial_r \left( \rm{ln}(h)\right) {\rm{d}}_{r} q$,
\item $G : r \mapsto \frac{R}{2} \frac{K_0(r/\kappa)}{K_1(R/\kappa)}.$
\end{itemize}
\end{lemma}
It is important to mention that on one hand $f_{\rm hyd}$ is defined as a non-local term which requires the computation of $(\zeta_{\rm e},q_{\rm e})$ in the full exterior domain $r > R$. On the other hand $\underline{\mathfrak{f}}$ only depends on $\delta$, $\dot{\delta}$, $\underline{\zeta_{\rm e}}$ and $\underline{\dot{\zeta_{\rm e}}}$. Indeed, because of the conservation of the mass and \eqref{Qint}, one has $\underline{q_{\rm e}} = \underline{q_{\rm i}} = - \frac{R}{2}\dot{\delta}$. The quantity $G$ represents the dispersive boundary layer effect and is defined such that $G(R)$ is a positive quantity. Indeed $G(r) = \mathcal{O}(r^{-1/2}e^{-r}) $ when $r \to \infty$. This makes the decreasing on $G$ even faster than the boundary layer in \cite{BeckLannes} which decreases exponentialy.
 
Coupling this to the modified Newton's equation \eqref{AddedMass} gives the system of ODEs, called SHODE (system of hidden ODE, see \cite{Beck-Martaud-disp}):
\begin{equation}\label{eq:ODETrace}
\begin{cases}
\tau_\kappa^2(\varepsilon \delta)\ddot{\delta} + \delta  + \nu \dot{\delta}-  \varepsilon a( \delta, \underline{\zeta}) \, \dot{\delta}^2= \kappa^2 \underline{\ddot{\zeta}} + \nu\underline{\dot{\zeta}} + \underline{\zeta} + F_{\rm ext}, \\
\kappa^2 \underline{\ddot{\zeta}_e} + \nu \underline{\dot\zeta} +\underline{\zeta} + \eps \underline{\mathfrak{f}} = \underline{f_{\rm{hyd}}} - \kappa \ddot{\delta} G(R),
\end{cases}
\end{equation}
where the added-mass $\tau_\kappa^2(\varepsilon \delta)$ is defined in Proposition \ref{prop:AddedMass}.
Also to simplify the notations we note $G := G(R)$. 
It is easy to see that this system of ODEs can be written with the SHODE variable $Z = (\delta, \dot{\delta}, \underline{\zeta}, \underline{\dot{\zeta}})^T$ (see Appendix \ref{app:LinAddM}). If we write $f_{\rm hyd} = \mathbf{f}_{\rm hyd} + \eps \mathtt{f}_{\rm hyd}$, with $\mathbf{f}_{\rm hyd}$ linear in $\zeta, q$ and $\mathtt{f}_{\rm hyd}$ nonlinear in $\zeta, q$, one has
\begin{equation}\label{eq:ODEOper}
\mathcal{M} \partial_t Z + \mathcal{T} Z = P\left(\underline{\textbf{f}_{\rm hyd}}, F_{\rm ext} \right) + \varepsilon \tilde{P}\left(Z, \underline{\mathtt{f}_{\rm hyd}}, F_{\rm ext}\right),
\end{equation}
with 
$$
\mathcal{M} = \begin{pmatrix}
 1& 0 & 0 & 0 \\
 \nu & \tau_\kappa^2(0) & -\nu & -\kappa^2 \\
 0& 0 & 1 & 0 \\
 0& \kappa G & \nu & \kappa^2 \\
\end{pmatrix},
\quad
\mathcal{T} = 
\begin{pmatrix}
 0& -1 & 0 & 0 \\
 1& 0 & -1 & 0 \\
 0& 0 & 0 & -1 \\
 0& 0 & 1 & 0 \\
\end{pmatrix},$$
$$
P\left(\underline{f_{\rm hyd}}, F_{\rm ext} \right) = \begin{pmatrix}
0 \\
F_{\rm ext}\\
0\\
\underline{\mathbf{f}_{\rm hyd}}
\end{pmatrix},
\quad \text{and} \quad
\tilde{P}(Z, \underline{f_{\rm hyd}}, F_{\rm ext}) = \begin{pmatrix}
0 \\
\gamma(Z, \underline{f_{\rm hyd}}, F_{\rm ext}) \\
0 \\
\underline{\mathtt{f}_{\rm hyd}} \\
\end{pmatrix},$$
where $\gamma$ is defined in appendix \ref{app:LinAddM}.
 
From now on we introduce the augmented variable $(u, Z)$. With respect to our notations $u$ stands for the wave variables in the exterior domain $u = (\zeta_{\rm e}, q_{\rm e})$ and $Z =(\delta, \dot{\delta}, \underline{\zeta_{\rm e}}, \underline{\dot{\zeta}_{\rm e}})$ is the variable of the SHODE. Now we dispose  of a system of equations coupled with a system of ODEs at the boundary. It totally describes the behavior of the waves-structure system. It is summed up in the next Subsection \ref{sec:AugFor}.
 
\subsection{Augmented formula}\label{sec:AugFor}
From now on we study the problem in the exterior domain $r > R$ and for any quantity $u$ we use the notation $u$ rather than $u_{\rm e}$.\newline
\indent Now we define the \textit{Augmented formula}. Firstly it combines the conservative system of PDEs plus bounded perturbations given by the application of the dispersive operator $\mathfrak{R}_0$ on the radial Boussinesq-Abbott {\rm \hyperlink{BAr}{(BAr)}}, and secondly the SHODE \eqref{eq:ODEOper}. Those systems of equations link the wave variables and the SHODE variables. Therefore it describes all the quantities of the problem in the exterior domain, and to reconstruct the different quantities in the inner domain, through $\delta$ and equation on the pressure \eqref{eq:pi(r)}.  
\begin{defn}[Augmented formula]
The augmented variables $(u,Z)$ are solutions of the following system 
\begin{equation}\label{eq:AUG}
\begin{cases}
\partial_t u + \partial_r \mathcal{F}_\kappa \left[ u \right] = \mathcal{D}(\underline{f_{\rm hyd}}, Z) \mathcal{S}_{\kappa}(r) + \mathcal{S}_{r}(u) + \eps \nu \mathcal{S}_{\nu}(u) \quad & \text{ for } r  > R,\\ 
 \underline{u} =  (\underline{\zeta}, -\frac{R}{2}\dot{\delta}) & \text{ for } r = R,\\
\mathcal{M} \partial_t Z + \mathcal{T} Z = P\left(\underline{f_{\rm hyd}}, F_{\rm ext}\right) + \varepsilon \tilde{P}\left(Z, \underline{f_{\rm hyd}}, F_{\rm ext}\right) \quad  & \text{ for } t > 0,
\end{cases} \tag{AUG}
\end{equation}
with initial conditions
\begin{equation}
\begin{cases}
u_{\vert_{t=0}}= (\zeta_0, q_0) & \text{ for } r > R, \\
Z_{\vert_{t=0}} = (\delta_0, \dot\delta_0, \underline{\zeta_0}, \underline{\dot\zeta_0}), \\
\underline{u}_{\vert_{t=0}} =  (\underline{\zeta_0}, -\frac{R}{2}\dot{\delta}_0) & \text{ for } r = R.
\end{cases}
\end{equation}
where:
\begin{itemize}
\item $f_{\rm hyd}$ is defined previously in Lemma \ref{lem:EqTrace} as 
\begin{equation}\label{eq:fhyd}
f_{\rm{hyd}} := \mathfrak{R}_1 \left( \zeta - \nu \dr q + \varepsilon ( \frac{\zeta^2}{2} + \frac{q^2}{h} ) \right) + \varepsilon \kappa^2 {\rm{d}}_{r} \mathfrak{R}_0 \left( \frac{q^2}{r h} \right) + \varepsilon \nu {\rm{d}}_{r} \mathfrak{R}_0 \partial_r \left( \rm{ln}(h)\right) {\rm{d}}_{r} q,
\end{equation}
\item $\mathcal{D}(\underline{f_{\rm hyd}}, Z) := \frac{1}{\tau_{\kappa}^2(0)+G} \left(\eps \gamma + \underline{f_{\rm{hyd}}} - \eps \underline{\mathfrak f} - \delta -  \nu \dot{\delta} + F_{\rm ext} \right)$ with $\gamma$ defined in appendix \ref{app:LinAddM},
\item the flux is
 $$\mathcal{F}_\kappa\left[ u \right] = \begin{pmatrix} q\\
\mathfrak{R}_1 (\zeta - \nu\dr q + \varepsilon \mathfrak f)
\end{pmatrix},$$
\item the boundary layer effect (due to the dispersion) is
 $$\mathcal{S}_{\kappa}(r) = \begin{pmatrix}
0 \\
-\frac{R}{2}K(r)\end{pmatrix},$$ 
\item the viscous term is
$$\mathcal{S}_{\nu}(u) =
\begin{pmatrix}
0\\
-\mathfrak R _0 \left( \partial_r \frac{\left( \rm{ln}(h)\right)}{\eps} {\rm{d}}_{r} q \right)
\end{pmatrix},$$
\item the radial term (due to the 2D situation) is
$$\mathcal{S}_{r}(u) =- \begin{pmatrix}
\frac{q}{r}\\
\eps \mathfrak R _0 \left[\frac{q^2}{rh} \right]
\end{pmatrix},$$
\item the terms of the ODE $\mathcal{M}$, $\mathcal{T}$, $P$ and $\tilde{P}$ are defined previously in equation \eqref{eq:ODEOper}.
\end{itemize}

\end{defn}
Now that the Augmented formula is defined one can add compatibility conditions on the initial conditions to fit with the initial problem: if the initial conditions satisfy compatibility conditions
\begin{equation}\label{CondCompat}
\begin{cases}
\zeta_{0}( r = R ) = \underline{\zeta_{0}}, \\
\underline{\dr q_{0}} = -\underline{\dot{\zeta}_{0}},
\end{cases}
\end{equation}
then any solution of \eqref{eq:AUG} is a solution of {\rm \hyperlink{BAr}{(BAr)}}. Thanks to the previous results, we proved Theorem \ref{thmf:Aug}. 
\begin{remark}\label{rk:DefAugF}
This Augmented formula is exactly the one described in Theorem \ref{thmf:Aug}, with the notations:
\begin{itemize}
\item $\tilde{\mathfrak{F}}_\kappa[u, Z] = \mathcal{F}_\kappa[u] - \mathcal{D}\left (\underline{f_{\rm hyd}}, Z\right )\begin{pmatrix}
0 \\
\frac{\kappa R}{2}\frac{K_0(r/\kappa)}{K_1(R/\kappa)}\end{pmatrix}$, where $\kappa \partial_r K_0(r/\kappa) = - K_1(r/\kappa)$,
\item $\mathcal{I}_{\rm hyd} u = \underline{f_{\rm hyd}}$,
\item $\tilde{\mathcal{D}}\left ( Z, \mathcal{I}_{\rm hyd} u, F_{\rm ext}\right ) = \mathcal{M}^{-1} \left( \mathcal{T} Z - P(\mathcal{I}_{\rm hyd} u, F_{\rm ext}) - \eps \tilde{P}\left ( Z, \mathcal{I}_{\rm hyd} u, F_{\rm ext}\right )\right )$.
\end{itemize}
\end{remark}
The operators $S_r$ and $\nu \eps S_\nu$ can be seen as \textit{bounded perturbations} of the main operator $\partial_r \mathcal{F}_\kappa$. In the 1D situation \cite{BeckLannes} the radial term $S_{r}$ is null. Therefore in 1D without the viscosity, $\nu S_{\nu} = S_r = 0$ and the Augmented formula describes a conservative equation.
The term $ S_{\nu}$ only contains the smaller part of what is brought by the viscosity. Indeed $\ln(h)$ is of order $\eps$ and this is why we separated this term from the one in the flux $\mathcal{F}_\kappa$. 
\begin{remark}\label{rk:AUGdr}
We decided to write the main operator as $\partial_r \mathcal{F}_\kappa$, but one could also write the Augmented formula with a main operator of the form ${\rm d}_r \breve{\mathcal{F}}_\kappa$. If we do so we obtain different perturbations, but it will still be bounded on $[R, +\infty)$. 
\end{remark}
This formula gives several results. Firstly, one has the wellposedness of the waves-structure system, given by Theorem \ref{thmf:WP}. Furthemore in the linear case, one finds out that $\delta \in H^2(\mathbb{R}^+)$, and we can study its asymptotic behavior in a certain case (see Theorem \ref{thmf:RTTE}). Moreover, the Augmented formula allows to compute numerical simulations, as it was done in 1D in \cite{BeckLannesWeynans}. Future work will examine this issue. 
 
\section{Wellposedness of the Augmented formula}\label{sec:WP}
 
We assume that Assumption \ref{hyp-system} is satisfied. We define the weighted space $H^2_\kappa$ as the Sobolev space $H^2_r$ equipped with the norm 
$$
||\cdot||_{H^{2}_\kappa}^2 := ||\cdot||_{H^1_r}^2 + \kappa^2 || \partial_r {\rm{d}}_{r} \cdot||_{L^2_r}^2.
$$
 
We work with the augmented variables $u = (\zeta, q)$ and $Z = (\delta, \dot{\delta}, \underline{\zeta_{\rm e}}, \underline{\dot{\zeta}_{\rm e}})$.
To show Theorem \ref{thmf:WP} we use a different method for each point of the theorem.
\begin{itemize}
\item {\it{(ODE type iterative scheme)}}
Firstly, the Augmented formula \eqref{eq:AUG} is treated as an ODE system of type 
$$ \partial_t V = B[V],$$
where $V = \left( \zeta, q, \kappa{\rm{d}}_{r} q, \delta, \dot{\delta}, \underline{\zeta}, \kappa\underline{\dot{\zeta}}\right)$ and $ B : \mathscr{C}_t^0\left((L^\infty_r)^3 \times \mathbb{R}^4\right) \to \mathscr{C}_t^0\left((L^\infty_r)^3 \times \mathbb{R}^4\right)$ is a nonlinear locally Lipschitz field. Here, thanks to the local Cauchy-Lipschitz theorem, one obtains a time existence of order $\mathcal{O}(T_{\rm ODE})$, where $T_{\rm ODE}$ is defined in Theorem \ref{thmf:WP} and is of the form $\eps^{-1}\kappa^2$. Now we want to obtain a time existence with weaker power of $\kappa$. To do so one needs higher regularity on initial conditions.
\item {\it{(Semi-linear type iterative scheme)}} 	
In a second step the wellposedness is shown with the following iterative scheme
\begin{equation}
\begin{cases}
\begin{pmatrix}
1 & 0 \\ 0 & (1-\kappa^2 \partial_r \rm{{\rm{d}}_{r}})
\end{pmatrix} \partial_t u^{k+1} + \begin{pmatrix}
0 & \rm{{\rm{d}}_{r}} \\
\partial_r & 0
\end{pmatrix} u^{k+1} = \varepsilon \partial_r \mathcal F[u^{k}] + \varepsilon \mathcal G[u^{k}] + S(Z^k), \\
\underline q^{k+1} = - \frac{R}{2}\dot\delta^{k+1}, \\
u_{|_{t = 0}} = u_0.
\end{cases}
\end{equation}
$\mathcal{F}$ and $\mathcal{G}$ are nonlinear operators of rational fraction type. This system is coupled with a SHODE like system, giving an ODE of type 
\begin{equation}
\partial_t Z^{k+1} + \mathfrak{D}(u^{k+1}, Z^{k+1}) + \eps \mathfrak{N}(u^k, Z^k) = 0, 
\end{equation}
where $\mathfrak{D}$ is a nonlocal but linear operator and $\mathfrak{N}$ is a nonlocal and nonlinear operator.
The key step is to obtain energy estimates for the linear system, i.e. the system above where nonlinearities have been replaced by bounded perturbations. 
It is modeled in the iterative scheme by the fact that nonlinearities are considered at step $k$ whereas linearities are determined at step $k+1$.
The energy in question is equivalent to $|| \cdot ||^2_{H^2_\kappa}$.
Unfortunately with this method, it is impossible to obtain energy estimates for the linear system which are uniform with respect to $\kappa$. 
Afterwards we will use those energy estimates with a Picard fixed point type argument to obtain an estimate of the existence time of the wellposedness which is better than with the ODE scheme. 
Finally, because of the boundary conditions, this time existence depends on the relation between $\kappa$ and $\varepsilon$, and because of the radiation the Helmholtz number will also have an impact in the time existence through $R$. 
\item {\it{(Quasi-linear type iterative scheme.)}}
Another method would be to treat the problem with its quasilinear formulation
\begin{equation*}
\begin{pmatrix}
1 & 0 \\ 0 & (1-\kappa^2 \partial_r {\rm{d}}_{r})
\end{pmatrix}  \partial_t u^{k+1} + A\left[u^{k}\right] \begin{pmatrix}
0 & \rm{{\rm{d}}_{r}} \\
\partial_r & 0
\end{pmatrix}u^{k+1} = 0.
\end{equation*}
This one is described in \cite{BeckLannes} in 1D. 
We will not do this case but we can suppose that in our case we also reach an existence time uniform in $\kappa$ independently of the relation between $\varepsilon$ and $\kappa$.
Therefore this hypothetical time existence is larger than the time existence of the semi-linear type scheme.
\end{itemize}
 
To show the different estimates, one needs the Moser-type inequalities: for any $n \in \mathbb{N}$ and $u, v \in H^n_r$, one has
\begin{equation}\label{eq:MTineq}
\| u v \|_{H^n_r} \lesssim \|u\|_{L^\infty_r} \| v\|_{H^n_r} + \|v\|_{L^\infty_r} \| u\|_{H^n_r}.
\end{equation}
 
\subsection{The ODE's perspective}\label{sec:estimODE}
We can formulate the Augmented formula \eqref{eq:AUG} as an ODE system of 7 variables. Indeed the dispersive operator $\mathfrak{R}_0$ allows to write that any solution $(\zeta, q)$ of the radial Boussinesq-Abbott {\rm \hyperlink{BAr}{(BAr)}} satisfies
\begin{equation}\label{eq:estim_3EDO}
\begin{cases}
\partial_t \zeta + {\rm{d}}_{r} q = 0,\\
\partial_t q + \partial_r \mathfrak R _1 \left( \zeta - \nu {\rm{d}}_r q\right) + \eps \partial_r \mathfrak R _1 \mathfrak{f} + \mathfrak R _0 \left( \varepsilon \frac{q^2}{rh} - \nu \partial_r \ln(h) {\rm{d}}_r q \right) = -\frac{R}{2}\ddot{\delta} K, \\
\kappa^2 \partial_t {\rm{d}}_r q + \nu {\rm{d}}_r q - \zeta - \eps \mathfrak{f} + f_{\rm{hyd}} = \kappa \ddot{\delta} G,
\end{cases}
\end{equation}
where $\mathfrak{f}$ stands for $\frac{\zeta^2}{2} + \frac{q^2}{h}$, $K$ is defined by Lemma \ref{lem:R0Homog}, $G$ in Lemma \ref{lem:EqTrace} and the non-local term $f_{\rm{hyd}}$ is given by \eqref{eq:fhyd}.
Also thanks to what was done in Section \ref{sec:AugFor}, this system is coupled with the SHODE
\begin{equation}\label{eq:ode}
\begin{cases}
\tau^2_\kappa (0)  \ddot{\delta} + \nu\dot\delta + \delta = \kappa^2 \underline{\ddot{\zeta}} + \nu\underline{\dot{\zeta}} +  \underline{\zeta} + \eps \gamma + F_{\rm ext}, \\
\kappa^2 \underline{\ddot{\zeta}} + \nu \underline{\dot\zeta} +\underline{\zeta} + \eps \underline{\mathfrak{f}} = \underline{f_{\rm{hyd}}} - \kappa \ddot{\delta} G(R).
\end{cases}
\end{equation}
Therefore systems \eqref{eq:estim_3EDO} and \eqref{eq:ode} are coupled and we can treat the problem as if $\kappa\dr q$ was a seventh variable.
It leads to consider the problem as an ODE system 
\begin{equation}\label{eq:ODEgene}
\begin{cases}
\partial_t V = M^{-1}\mathcal{B}[V], \\
V_{|_{t = 0}} = V_0.
\end{cases}
\end{equation}
of unknown $V = \left(
\zeta,
q,
\kappa{\rm{d}}_{r} q,
\delta,
\dot{\delta},
\underline{\zeta},
\kappa\underline{\dot{\zeta}}
\right)$, with $ M$ the diagonal matrix of diagonal $\left( 1, 1, \kappa, 1,  \tau^2_\kappa(0), 1, \kappa\right)$ and the operator
$$
\mathcal{B}[V] = 
\begin{pmatrix}
- {\rm{d}}_{r} q\\
- \left[ \partial_r \mathfrak R _1 \left( \zeta - \nu {\rm{d}}_r q + \eps \mathfrak{f} \right) + \mathfrak R _0 \left( \varepsilon \frac{q^2}{rh} - \nu \partial_r \ln(h) {\rm{d}}_r q \right) + \frac{R}{2}\mathcal{D} K \right]\\
\zeta - \nu {\rm{d}}_r q + \eps \mathfrak{f} - f_{\rm{hyd}} + \kappa \mathcal{D} G\\
\dot{\delta}\\
\mathcal{D}\\
\dot{\underline{\zeta}}\\
\underline{f_{\rm{hyd}}} - \kappa \mathcal{D}G(R) - \underline{\zeta} - \nu \underline{\dot{\zeta}} - \eps \underline{\mathfrak f}
\end{pmatrix},
$$
where $\mathcal{D}$ is given by $\mathcal{D} = \mathcal{D}(u, Z)  = (\tau_{\kappa}^2(0)+G)^{-1}\left(\eps \gamma + \underline{f_{\rm{hyd}}} - \eps \underline{\mathfrak f} - \delta -  \nu \dot{\delta} + F_{\rm ext} \right)$. Note that $\mathcal{D}$ is exactly $\ddot{\delta}$ but we want to underline that there is no term of second order derivative in $\mathcal{B}$.
 
As explained earlier we want to show the first point of Theorem \ref{thmf:WP}. In fact it is equivalent to the following proposition, according to Definition \ref{def:WP}.
\begin{proposition}\label{prop:WPODE}
Let $\eps \in [0, 1)$, $\kappa > 0$ and $\nu \geq 0$.
If the initial condition belongs to $(L^\infty_r)^3 \times \mathbb{R}^4$, then there exists $c >0$ independent of $\eps, \kappa, R$ such that the system \eqref{eq:ODEgene} is well-posed in $\mathscr{C}^0( [0, cT_{\rm ODE}), \left(L^\infty_r \right)^3 \times \mathbb{R}^{4} )$ with $T_{\rm ODE} = \frac{\kappa^2 \eps^{-1}}{(1+\nu \kappa^{-1})^2}$. 
\end{proposition}
If one takes $V = (V_1, V_2, V_3, V_4, V_5, V_6, V_7) \in \mathscr{C}^0( [0, cT_{\rm ODE}), \left(L^\infty_r \right)^3 \times \mathbb{R}^{4} )$ satisfying \eqref{eq:ODEgene} then we can reconstruct the variables of interest by defining $\zeta = V_1$, $q = V_2$ and $\delta = V_4$. 
Furthermore we use that $\partial_t V_1 = \kappa^{-1}V_3 = {\rm d}_r V_2$ to say that $q \in W^{1, \infty}$.
We reconstructed the first and the third equations of the Augmented formula \eqref{eq:AUG} but we still have to show the equation on the trace of $(\zeta, q)$.
Since $q \in W^{1, \infty}$ its trace is well defined. From Lemma \ref{lem:OpDisp} all the terms in the following equation can be taken at the trace 
$$
\partial_t V_2 = - \left[ \partial_r \mathfrak R _1 \left( V_1 - \nu \kappa^{-1}V_3 + \eps \mathfrak{f} \right) + \mathfrak R _0 \left( \varepsilon \frac{V_2^2}{r(1+\eps V_1)} - \nu \kappa^{-1} \partial_r \ln(1+\eps V_1) V_3 \right) + \frac{R}{2}\partial_t V_5 K \right].
$$ 
Since $\mathfrak{R}_0$ and $\mathfrak{R}_1$ are defined such that $\underline{\mathfrak{R}_0 u} = 0$ and $\underline{\partial_r \mathfrak{R}_1 u} = 0$ for any $u \in L^2_r$, the equation on $\partial_t V_2$ allows to construct $V_2$ at the boundary by $\partial_t \underline{V_2} = - \frac{R}{2}\partial_t V_5$, where $\partial_t V_5$ is welldefined by the problem, and $K(R) = 1$ by definition (see Lemma \ref{lem:R0Homog}). 
Therefore we find that $\underline{q} = \underline{V_2} = - \frac{R}{2}V_5$.
Finally we set $\underline{\zeta} = V_6$ which is well defined by the hidden equation.
We have reconstruct all the variables as they are described by the augmented formula \eqref{eq:AUG}.
\begin{remark}
We can work with the trace of $\zeta$ even if $\zeta \in L^\infty_r$ thanks to the hidden equation, without using any sharp trace theorem.
\end{remark}
The strategy of the proof of this proposition consists in using the Cauchy-Lipschitz theorem. To do so we will compare two solutions of the system \eqref{eq:ODEgene}, and give an estimate of the linear part. Finally an estimate of the non-linear part will give the estimate of the time existence.
\begin{proof}
Let $V_k = (\zeta_k, q_k, \kappa{\rm{d}}_{r} q_k, \delta_k, \dot{\delta_k}, \underline{\zeta_k}, \kappa\underline{\dot{\zeta_k}})$ in $(L^\infty_r)^3 \times \mathbb{R}^4$ and $h_k = 1 + \eps \zeta_k$, for $k \in \{ 1, 2\}$. Then one has the following space estimates. 
\begin{itemize}
\item According to Lemma \ref{lem:OpDisp}, $\partial_r \mathfrak R _1 : L^\infty_r \mapsto L^\infty_r $ is bounded with a bound of order $\kappa^{-1}$
\begin{align*}
| \partial_r \mathfrak R _1 \mathfrak{f}_1 - \partial_r \mathfrak R _1 \mathfrak{f}_2 |  
&\leq \kappa^{-1} C_1(\|\zeta_1, q_1, \zeta_2, q_2\|_{L^\infty_r})\|(\zeta_1, q_1) - (\zeta_2, q_2)\|_{\left ({L^\infty_r}\right )^2},
\end{align*}
and 
\begin{equation*}
| \nu (\partial_r \mathfrak R _1 {\rm{d}}_r q_1 - \partial_r \mathfrak R _1 {\rm{d}}_r q_2)|  \leq \kappa^{-2} \nu C_2\| \kappa{\rm{d}}_r q_1 - \kappa{\rm{d}}_r q_2 \|_{L^\infty_r},
\end{equation*}
where $C$, $C_1$ and $C_2$ are three constants independent on $\eps$, $\kappa$ and $\nu$.
\item On one hand, by Lemma \ref{lem:OpDisp}, one has  
\begin{align*}
\left |\kappa^2 \dr \mathfrak R_0 \left[  \partial_r \ln(h_1) \dr q_1 - \partial_r \ln(h_2) \dr q_2 \right] \right | &\leq \kappa^{-1}\|\kappa\dr q_1 - \kappa\dr q_2 \|_{L^\infty_r} |\kappa^2 \dr \mathfrak R_0\left[ \partial_r \ln(h_1) - \partial_r \ln(h_2)\right] |\\
& \leq \kappa^{-1}\|\kappa\dr q_1 - \kappa\dr q_2 \|_{L^\infty_r} |\kappa^2 \dr \partial_r \mathfrak R_1 \ln(h_1/h_2)|.
\end{align*}
The definition of $\mathfrak{R}_1$ gives
\begin{align*}
|\kappa^2 \dr \partial_r \mathfrak R_1 \ln(h_1/h_2)| & = |\mathfrak R_1 \ln(h_1/h_2) - \ln(h_1/h_2)|\\
& \leq \|\ln(h_1/h_2)\|_{L^\infty_r}.
\end{align*}
Moreover thanks to the mean value inequality, since $\zeta_1, \zeta_2 \in L^\infty_r$ then
$$
 \|\ln(h_1/h_2)\|_{L^\infty_r} \leq \eps h_{\min}^{-1} \|\zeta_1 - \zeta_2 \|_{L^\infty_r},
$$
and
$$\|\kappa^2 \dr  \mathfrak R_0 \left[ \partial_r \ln(h_1) \dr q_1 - \partial_r \ln(h_2) \dr q_2 \right] \|_{L^\infty_r} \lesssim \kappa^{-2}\eps\|\kappa\dr q_1 - \kappa\dr q_2  \|_{L^\infty_r} \|\zeta_1 - \zeta_2\|_{L^\infty_r}.$$
On the other hand 
\begin{align*}
|\dr \mathfrak R_0 \left( \frac{q_1^2}{rh_1} - \frac{q_2^2}{rh_2}\right )|& \leq \left(\frac{1}{R} + \kappa^{-1}\right)\|\frac{q_1^2}{rh_1} - \frac{q_2^2}{rh_2}\|_{L^\infty_r}\\
& \leq \left(\frac{1}{R} + \kappa^{-1}\right)\|q_1, q_2, \frac{1}{h_1}, \frac{1}{h_2}\|_{L^\infty_r} \|q_1 - q_2\|_{L^\infty_r}.
\end{align*}
\end{itemize}
We deduce that the group generated by the linear part of $M^{-1}\mathcal{B}$ provides a upper bound of type $\kappa^{-1}(1+\nu\kappa^{-1})c$, where $c$ is a constant independent of $\eps$, $\kappa$, $\nu$ and $R$. 
Applied to the non-linearities it leads to obtain for $V_1$, $V_2$ $\in W := \left(L^\infty_r \right)^3 \times \mathbb{R}^{4}$, if we denote $\eps \mathcal{N}[V_k]$ the non-linear part of $M^{-1}\mathcal{B}[V_k]$ ($k\in \{1, 2\}$):
$$
\|\eps \mathcal{N}[V_1] - \eps \mathcal{N}[V_2] \|_{\mathscr{C}^0([0, T), W)} \lesssim \kappa^{-1} (1+\nu \kappa^{-1}) \left(\kappa^{-1}\eps (1+\nu \kappa^{-1})\right) T \|V_1 - V_2 \|_{W}.
$$
With the previous estimates, the constant in the inequality does not depend on $\kappa\dr q_k$ or $\underline{\zeta_k}, \underline{\dot{\zeta}_k}$. Therefore those terms do not enter in the blow-up criterium.
By local Cauchy-Lipschitz theorem, there exists $c > 0$ of order $1$ such that \eqref{eq:ODEgene} admits a unique solution in $\mathscr{C}^0( [0, cT_{\rm ODE}), (L^\infty_r)^3 \times \mathbb{R}^{4} ))$ with a time existence $T_{\rm ODE} = \frac{\kappa^2 \eps^{-1}}{(1+\nu \kappa^{-1})^2}$.
\end{proof}	
In particular the Augmented formula \eqref{eq:AUG} is well-posed in $\mathscr{C}^0( [0, cT_{\rm ODE}), L^\infty_r \times W^{1,\infty}_{r, \kappa} \times \mathbb{R}^{4} )$. Comments over this time existence will be done at a later stage, but one can see that we recover the global well-posedness of the linear system ($\eps = 0$), but not the wellposedness of the Saint-Venant system ($\kappa = 0$).

\begin{remark}\label{rk:LocWPinfty}
The local wellposedness is shown with $(\zeta, q) \in \mathscr{C}^0_t(L^\infty_r \times W^{1, \infty}_{r, \kappa})$. By the same method, thanks to Proposition \ref{prop:commut}, one can show that for initial conditions in $(W^{s,p}_r)^3 \times \mathbb{R}^4$ such that $L^\infty_r \subset W^{s, p}_r$, the system \eqref{eq:ODEgene} is wellposed in $\mathscr{C}^0_t\left ([0, cT_{\rm ODE}), (W^{s, p}_r)^2 \times W^{s,p}_{r, \kappa} \times \mathbb{R}^4\right )$, where $W^{s,p}_{r, \kappa} := \left \{ f \in W^{s-1, p}_{r}, \kappa\partial_r^s f \in L^p_r\right \}$. In particular it proves the existence and the uniqness of solutions of the waves-structure system with $(\zeta, q) \in H^1_r \times H^2_\kappa$.
\end{remark}
 
\begin{remark}
The viscosity $\nu$ could depend on $q$. Under regularity hypothesis on $\nu(q)$ it would change the cofficients, but the same technique would give the well-posedness.
\end{remark}

\begin{remark} 
To generalize to the case of $\mathbb{R}^2$ without symmetries we would like to set dispersive operators which reverse $(1-\kappa^2\nabla {\rm{div}})$. 
For any $f \in L^2\left(\mathbb{R}^2\right)$, we define $\mathfrak{\tilde{R}}_0: L^2(\mathcal{E}) \to H^1({\rm div})$ by
\begin{equation}
u = \mathfrak{\tilde{R}}_0 f \Leftrightarrow
\begin{cases}
 (1 - \kappa^2 \nabla {\rm{div}})u = f   & \text{in } \mathcal{E},\\ 
  \underline{u \cdot \vec{n}} = 0 & \text{on } \Gamma, \\
  \underset{r \rightarrow + \infty }{\lim} u = 0, & 
\end{cases}
\end{equation}
where $\vec{n}$ is the normal exterior vector at the contact line $\Gamma$.
However if we apply $\mathfrak{\tilde{R}}_0$ to equation \eqref{AbbottDL:eq2} we obtain
\begin{equation}
\begin{cases}
\partial_t \zeta +{\rm div}\, Q=0,\\
\partial_t Q+\varepsilon \mathfrak{\tilde{R}}_0{\rm{div}}\, \left( \frac{Q \otimes Q}{h} \right)+ \mathfrak{\tilde{R}}_0\left (h \nabla  \zeta\right )= - \mathfrak{\tilde{R}}_0\left (h \nabla \frac{P}{\varepsilon}  + h \nabla \frac{\nu}{h} {\rm{div}} Q\right ), \\
 h= 1 + \varepsilon \zeta.
\end{cases}
\end{equation}
Contrary to the axisymmetric without swirl case case, the operator $\tilde{\mathfrak{R}}_0$ does not regularize enough, in the sense that $\tilde{\mathfrak{R}}_0(h\nabla \zeta)$ does not belong to $H^1$.
\end{remark}
\subsection{Semi-linear perspective}\label{sec:4.2}
We set in this section $\nu = 0$. Even if it does not add any difficulty, it adds unnecessary technicality (see \ref{rk:WPSL_visc}).
We set $\mathbb{H} := L^2_r \times H^{1}_\kappa $ with $H^{1}_\kappa$ the Sobolev space $H^1_r$ equipped with the norm 
$$
||\cdot||_{H^{1}_\kappa}^2 := ||\cdot||_{L^2_r}^2 + \kappa^2 || {\rm{d}}_{r} \cdot||_{L^2_r}^2.
$$
In the same approach we define \label{def:H} $\mathbb{H}^1 := H^1_r \times H^{2}_\kappa$ where $H^{2}_\kappa$ is the Sobolev space $H^2_r$ equipped with the norm 
$$
||\cdot||_{H^{2}_\kappa}^2 := ||\cdot||_{H^1_r}^2 + \kappa^2 || \partial_r {\rm{d}}_{r} \cdot||_{L^2_r}^2.
$$
We remind that $L^2_r$ (respectively $H^1_r$) stands for the $L^2$ space (respectively the $H^1$ space) on $[R, +\infty)$ with the measure $rdr$.
In this section, we use the notation of the Augmented variable $V = (u, Z)$ and we want to prove the well-posedness with $u$ in $\mathbb{H}^1$ with a time existence larger than with the ODE's perspective (for small values of $\kappa$), that is to say the second point of Theorem \ref{thmf:WP}. 
In fact it is equivalent to the following proposition.
\begin{proposition}\label{prop:WPSL}
Let $\eps \in [0, 1)$, $\kappa > 0$ and $\nu = 0$.
For initial condition in $ \mathbb{H}^1 \times \mathbb{R}^{4} $, there exists $c > 0$ independent of $\eps, \kappa, R$ such that the Augmented system \eqref{eq:AUG} is well-posed in $\mathscr{C}^0( [0, T^*), \mathbb{H}^1 \times \mathbb{R}^{4} )$ for time order $T^* = cT_{\eps, \kappa, R}$, where $T_{\eps, \kappa, R}$ is described in Theorem \ref{thmf:WP}.
\end{proposition}
The strategy consists in using energy estimates on \eqref{eq:AUG} to use a Picard fixed point. 
To do so we consider the system described as following.
Let $T > 0$ and $V^\ast  := (u^\ast, Z^\ast) \in \mathscr{C}^0_t ([0,T), \mathbb{H}^1 \times \mathbb{R}^4)$, 
we have a look at the solutions to the linear type problem
\begin{equation}\label{eq:SLFormGeneral}
\begin{cases}
\mathcal{D}_\kappa \partial_t u + \mathcal L [u] = \varepsilon \partial_r \begin{pmatrix} 0 \\
\mathcal F[V^\ast] \end{pmatrix} + \varepsilon \begin{pmatrix} 0 \\ \mathcal G[V^\ast] \end{pmatrix} \quad \text{for } r > R,\\
\underline q = - \frac{R}{2}\dot\delta, \\
u_{|_{t = 0}} = u_0,
\end{cases}
\end{equation}
with the dispersive matrix 
$$\mathcal{D}_\kappa =\begin{pmatrix}
1 & 0\\
0 & (1-\kappa^2 \partial_r {\rm{d}}_{r})
\end{pmatrix},
$$
$\mathcal{L}$ the operator
$$\mathcal{L} = \begin{pmatrix}
0 & \rm {\rm{d}}_{r} \\
\partial_r & 0
\end{pmatrix},$$
and $\mathcal{F}$, $\mathcal{G}$ are nonlinear operators. 
Also $\delta$ is described by the following equation.
\begin{equation}\label{eq:ADDMASS}
\tau_\kappa^2(0)\ddot{\delta} + \delta =  \underline \zeta + \kappa^2 \underline{\ddot{\zeta}} + \varepsilon \mathcal{F}_m [V^\ast] + F_{\rm ext},
\end{equation}
where $\mathcal{F}_m$  is a non-linear operator.
One can notice that $\mathcal{F}$, $\mathcal{G}$ and $\mathcal{F}_m$ will be used as
\begin{itemize}
\item $\mathcal{F}[V^*] := \frac{(V^*_1)^2}{2} + \frac{(V^*_2)^2}{1+\eps V_1^*}$,
\item $\mathcal{G}[V^*] := \frac{(V^*_2)^2}{r(1+\eps (V^*_1))}$, for $r > R$,
\item $ \mathcal{F}_m [V^*] = \gamma(V^*)$ with $\gamma$ defined in appendix \ref{app:LinAddM}.
\end{itemize}
We want to do an iterative scheme, which requires that the sources terms have the same regularity as the solutions.
It is translated by imposing $\mathcal{F}[V^\ast]$, $\mathcal{G}[V^\ast]$ $\in \mathscr{C}^0([0, T), H^2_\kappa)$ and $\mathcal{F}_m[V^\ast] \in \mathscr{C}^0([0, T))$ for $V^* \in \mathscr{C}^0([0, T), \mathbb{H}^1 \times \mathbb{R}^4)$. \newline
\indent Moreover by applying Lemma \ref{lem:HidEq} to system \eqref{eq:SLFormGeneral} one has 
\begin{equation}\label{eq:HiddenEq}
\kappa^2 \ddot{\zeta} + \zeta = \mathfrak R_1 \zeta - \ddot{\delta} G(r) + \varepsilon \mathcal{F}[V^\ast] - \varepsilon \mathfrak R_1 \mathcal{F}[V^\ast] + \varepsilon \kappa^2 {\rm{d}}_{r} \mathfrak R_0 \mathcal{G}[V^\ast].
\end{equation}
By taking this at the boundary one obtains 
\begin{equation}\label{eq:HidEqTrace}
\kappa^2 \underline{\ddot{\zeta}} + \underline{\zeta} = \underline{\mathfrak R_1 \zeta} -\kappa \ddot{\delta} G(R) + \varepsilon \underline{\mathcal{F}[V^\ast]} - \varepsilon \underline{\mathfrak R_1 \mathcal{F}[V^\ast] } + \varepsilon \kappa^2 \underline{{\rm{d}}_{r} \mathfrak R_0 \mathcal{G}[V^\ast]}.
\end{equation}
It gives a SHODE like system 
\begin{equation}\label{eq:ODEglob}
\begin{cases}
\tau_\kappa^2(0)\ddot{\delta} + \delta =  \underline \zeta + \kappa^2 \underline{\ddot{\zeta}} + \varepsilon \mathcal{F}_m [V^\ast] + F_{\rm ext},\\
\kappa^2 \underline{\ddot{\zeta}} + \underline{\zeta} = \underline{\mathfrak R_1 \zeta} - \kappa\ddot{\delta} G(R) + \varepsilon \underline{\mathcal{F}[V^\ast]} - \varepsilon \underline{\mathfrak R_1 \mathcal{F}[V^\ast] } + \varepsilon \kappa^2 \underline{ {\rm{d}}_{r} \mathfrak R_0 \mathcal{G}[V^\ast]}.
\end{cases}
\end{equation}
\indent Similarly to the ODE's perspective (see Subsection \ref{sec:estimODE}), one can treat the system \eqref{eq:SLFormGeneral} combined with ODE's system \eqref{eq:ODEglob} as a seven-unknown ODE's. The solution exists and is unique in $\mathscr{C}^0_t([0, T), \mathbb{H}^1\times \mathbb{R}^4)$, for $V^* \in \mathscr{C}^0_t([0, T), \mathbb{H}^1\times \mathbb{R}^4)$ for $T < T_{\rm ODE}$. The goal of this section is to find estimates on the energy of the linear system ($\eps = 0$). Those estimates will be more precised than in the previous section. Then, applying the wanted $\mathcal{F}$, $\mathcal{G}$ and $\mathcal{F}_m$, one will obtain an estimate of the time existence of the solutions of the Augmented formula. One has the estimate given by the following property.
 
\begin{proposition}\label{prop:EstimGlobale}
Let $\alpha = \max(2R^+, 9)$. Let $T > 0$ and $V^\ast  := (u^\ast, Z^\ast) \in \mathscr C^0_t ([0,T), \mathbb{H}^1 \times \mathbb{R}^4)$ and initial condition $(u(0),Z(0)) \in \mathbb{H}^1 \times \mathbb{R}^4$.
There exists a unique solution solution of the system \eqref{eq:SLFormGeneral} combined with ODE's system \eqref{eq:ODEglob} in $\mathscr{C}^0_t([0, T), \mathbb{H}^1\times \mathbb{R}^4)$. Moreover the solution satisfies the energy estimate for $t < T$
\begin{align*}\label{eq:EstimGlobale}
\sqrt{E(t) + E'(t)} \lesssim & \sqrt{\beta_0} + \varepsilon T(T + \sqrt{\alpha})\left(\|\partial_r \mathcal{F}[V^\ast] , \mathcal{G}[V^\ast]\|_{L^\infty([0, T)) L^2_r}^2 + \|\mathcal{F}_m[V^\ast]\|_{L^\infty([0, T))}\right) \\
&+ \varepsilon \kappa^{-1} \int_0^T \left( \max(1, \frac{1}{R^2})\|\mathcal{F}[V^\ast]\|_{H^1_r} + \| \mathcal{G}[V^\ast]\|_{H^1_r} + |\mathcal{F}_m[V^\ast]| \right)  \\
&+ \kappa^{-1} \|F_{\rm ext}\|_{L^1([0, T))},
\end{align*}
with $\beta_0 = E'(0) - \frac{R}{2\tau_\kappa^2(0)}\delta(0)\underline{\zeta}(0) + \frac{\alpha}{2\tau_\kappa^2(0)}E(0)\geq 0$,
and energies given by 
$$
E :=  \frac{||\zeta, q||_{\mathbb{H}}^2 + \frac{R}{2} |\delta, \tau_\kappa^2(0)\dot\delta|^2}{2} 
\quad \text{and}  \quad
E' :=  \frac{||\partial_r \zeta, {\rm{d}}_{r} q||_{\mathbb{H}}^2 + \left( \frac{R}{2\tau_\kappa^2(0)} \right) |\underline{\zeta}, \underline{\dot{\zeta}}|^2}{2}.
$$ 
\end{proposition}
In this proposition $\beta_0$ is positive as it will be shown in the proof of Corollary \ref{cor:4.1}. 
For the sake of clarity we note $G$ rather than $G(R)$, and we will use notations $\mathcal{F}$, $\mathcal{G}$ and $\mathcal{F}_m$ rather than $\mathcal{F}[V^\ast]$, $\mathcal{G}[V^\ast]$ and $\mathcal{F}_m[V^\ast]$.
 
\subsubsection{Proof of Proposition \ref{prop:EstimGlobale}}
As said previously, if $V^* \in \mathscr{C}^0_t([0, T), \mathbb{H}^1\times \mathbb{R}^4)$, solutions of the system \eqref{eq:SLFormGeneral}, \eqref{eq:ODEglob} must exist in $\mathscr{C}^0_t([0, T), \mathbb{H}^1\times \mathbb{R}^4)$ in order to apply a fixed point procedure. The previous estimate is given by Gronwall lemma applied on the two following lemmata.
\begin{lemma} \label{lemma:4.1}
Let $T > 0$ and $V^\ast  := (u^\ast, Z^\ast) \in \mathscr C^0_t ([0,T), \mathbb{H} \times \mathbb{R}^4)$. 
Let $(\zeta, q, \delta, \dot{\delta}, \underline{\zeta}, \underline{\dot{\zeta}})$ a solution of \eqref{eq:SLFormGeneral}, \eqref{eq:ODEglob} and initial conditions $(\zeta, q, \delta, \dot{\delta}, \underline{\zeta}, \underline{\dot{\zeta}})_{|_{t = 0}} \in \mathbb{H}^1 \times \mathbb{R}^4$.
Then $(\zeta, q, \delta, \dot{\delta}, \underline{\zeta}, \underline{\dot{\zeta}})$ satisfies 
$$\forall t \in [0, T), \sqrt{E(t)} \leq \sqrt{E(0)} + \varepsilon ||\mathfrak{F}||_{L^1([0, T))} + ||F_{\rm ext}||_{L^1([0, T))},$$
with $E :=  \frac{||\zeta, q||_{\mathbb{H}}^2 + \frac{R}{2} |\delta, \tau_\kappa^2(0)\dot\delta|^2}{2}$ and $\mathfrak{F}:= |\partial_r \mathcal{F} , \mathcal{G}|_{L^2_r}^2 + \mathcal{F}_m $  time functions.
\end{lemma}
\begin{proof}
Let $T > 0$ and $V^\ast  := (u^\ast, Z^\ast) \in \mathscr C^0_t ([0,T), \mathbb{H} \times \mathbb{R}^4)$.
Let $( \zeta, q )$ be a solution of \eqref{eq:SLFormGeneral}, \eqref{eq:ODEglob} and initial conditions $(\zeta, q, \delta, \dot{\delta}, \underline{\zeta}, \underline{\dot{\zeta}})_{|_{t = 0}} \in \mathbb{H}^1 \times \mathbb{R}^4$. 
Taking \eqref{eq:SLFormGeneral} multiplied by $( \zeta, q )$ leads to 
\begin{equation}
\begin{cases} 
\partial_t \frac{|\zeta|^2}{2} + {\rm{d}}_{r}(\zeta q) = q \partial_r \zeta, \\
\partial_t \frac{|q|^2}{2} + \partial_t \frac{|\kappa {\rm{d}}_{r} q|^2}{2} - \kappa ^2 {\rm{d}}_{r}(\partial_t [{\rm{d}}_{r} q ] q) + q \partial_r \zeta = \varepsilon \partial_r \mathcal{F}q + \varepsilon \mathcal{G} q.
\end{cases}
\end{equation}
The sum of those two equations gives
\begin{equation}
\partial_t \frac{|q|_{L^2_r}^2 + |\zeta|_{L^2_r}^2 + |\kappa\dr q|_{L^2_r}^2}{2} + \dr\left ((\kappa^2\ddot{\zeta} + \zeta)q\right ) = \varepsilon \partial_r \mathcal{F}q + \varepsilon \mathcal{G} q.
\end{equation}
Integrating this on $(R, +\infty)$ with the measure $rdr$ gives
\begin{equation}
 \partial_t \frac{||\zeta, q||_{\mathbb{H}}^2}{2} - \left ( \kappa^2 \underline{\ddot{\zeta}} + \underline{\zeta}\right )\underline q \leq \varepsilon |\partial_r \mathcal{F}, \mathcal{G}|_{L^2_r} \sqrt{E}. 
\end{equation}
Replacing $\underline{q}$ by its expression $\underline{q} = -\frac{R}{2}\dot{\delta}$, it leads to the inequality
\begin{equation}\label{eq:FIN1}
 \partial_t \frac{||\zeta, q||_{\mathbb{H}}^2}{2} + \left( \kappa^2 \underline{\ddot{\zeta}} + \underline{\zeta}\right )\frac{R}{2}\dot{\delta} \leq \varepsilon |\partial_r \mathcal{F}, \mathcal{G}|_{L^2_r} \sqrt{E}. 
\end{equation}
The first ODE linking $\ddot{\delta}$ and $\underline{\ddot{\zeta}}$ \eqref{eq:ADDMASS} multiplied by $\dot\delta$ gives
\begin{equation}
\partial_t \left (\frac{(\tau_\kappa^2(0)\dot{\delta})^2}{2} + \frac{\delta^2}{2}\right ) = ( \underline \zeta + \kappa^2 \underline{\ddot{\zeta_e}} ) \dot{\delta} + \left(\eps \mathcal{F}_m + F_{\rm ext} \right) \dot\delta,
\end{equation}
and we deduce from the previous equation \eqref{eq:FIN1}
\begin{equation}
\partial_t E \leq  \left( \varepsilon (|\partial_r \mathcal{F}, \mathcal{G}|_{L^2_r} + \mathcal{F}_m) + F_{\rm ext} \right)\sqrt{E}. 
\end{equation}
By the Gronwall lemma we deduce Lemma \ref{lemma:4.1}.
\end{proof}
If we have a look at the linear system ($\varepsilon = 0$) without external force ($F_{\rm ext} = 0$), one finds out the conservation of energy and it leads to the global wellposedness of the linear system. That is to say that the solutions of the linear \eqref{eq:AUG} stand in $\mathscr{C}^0([0, +\infty), \mathbb{H} \times \mathbb{R}^4)$. However in the general situation ($\eps \neq 0$) we need to deal with $\partial_r \mathcal{F}$ because it contains a space partial derivative of $\zeta$. To do so we use the same kind of estimate about the derivative of the system \eqref{eq:SLFormGeneral}, given by the following lemma.
\begin{lemma} \label{lemma:4.2}
Let $T > 0$ and $V^\ast  := (u^\ast, Z^\ast) \in \mathscr C^0_t ([0,T), \mathbb{H}^1 \times \mathbb{R}^4)$. 
Let initial conditions $(\zeta, q, \delta, \dot{\delta}, \underline{\zeta}, \underline{\dot{\zeta}})_{|_{t = 0}} \in \mathbb{H}^1 \times \mathbb{R}^4$. 
Then for any solution $(\zeta, q, \delta, \dot{\delta}, \underline{\zeta}, \underline{\dot{\zeta}})$ of \eqref{eq:SLFormGeneral}, \eqref{eq:ODEglob} one has 
$$\partial_t \left(E' - \frac{R}{2\tau_\kappa^2(0)} \delta \underline{\zeta} \right)\lesssim  \kappa^{-1} \left( \eps \left (\max(1, \frac{1}{R^2})\|\mathcal F\|_{H^1_r} + |\mathcal{F}_m| +  \| \mathcal{G}\|_{H^1_r} \right ) + |F_{\rm ext}| \right)\sqrt{E'} + \dot \delta \underline{\zeta}.$$
 with $E' :=  \frac{\|\zeta', q', \kappa \partial_r q'\|_{L^2_r}^2 + \frac{R}{2\tau_\kappa^2(0)}|\underline\zeta, \kappa \underline{\dot\zeta}|^2}{2}$, and $\zeta ' = \partial_r \zeta $ and $q' = {\rm{d}}_{r} q$.
\end{lemma}
If we manage to obtain this lemma, and if we deal with the terms $\delta\underline{\zeta}$ and $\dot{\delta} \underline{\zeta}$, then with a Gronwall lemma it is easy to get Proposition \ref{prop:EstimGlobale}. 
We will treat the terms $\delta\underline{\zeta}$ and $\dot{\delta} \underline{\zeta}$ with Young's inequalities after showing Lemma \ref{lemma:4.2}. 
\begin{proof}
Let $T > 0$, and $V^\ast  := (u^\ast, Z^\ast) \in \mathscr C^0_t ([0,T), \mathbb{H}^1 \times \mathbb{R}^4)$. 
Let $(\zeta, q, \delta, \dot{\delta}, \underline{\zeta}, \underline{\dot{\zeta}})$ a solution of \eqref{eq:SLFormGeneral}, \eqref{eq:ODEglob} with initial conditions $(\zeta, q, \delta, \dot{\delta}, \underline{\zeta}, \underline{\dot{\zeta}})_{|_{t = 0}} \in \mathbb{H}^1 \times \mathbb{R}^4$.
If we apply the operator $(\partial_r, {\rm{d}}_{r})$ to the system \eqref{eq:SLFormGeneral}, with the notations $\zeta ' = \partial_r \zeta $ and $q' = {\rm{d}}_{r} q$ we obtain
\begin{equation}
\begin{cases}
\partial_t \zeta'  + \partial_r q' = 0, \\
(1 - \kappa^2 {\rm{d}}_{r}  \partial_r ) \partial_t q' + {\rm{d}}_{r} \zeta' = \varepsilon {\rm{d}}_{r} [\partial_r \mathcal F] + \varepsilon {\rm{d}}_{r} \mathcal G.
\end{cases}
\end{equation}
The system \eqref{eq:SLFormGeneral} gives the boundary condition of the derivated system 
\begin{equation}\label{eq:BCq'}
\underline{q'} =- \underline{\dot \zeta}.
\end{equation}
Then, as in the proof of Lemma \ref{lemma:4.1} one has the following inequality.
\begin{equation}
\partial_t \frac{\|\zeta', q', \kappa\partial_r q'\|_{L^2_r}^2} {2} - (\kappa^2 \underline{\ddot{\zeta'}} + \underline{\zeta'}) \underline{q'} = \varepsilon \int_R^{+\infty}  {{\rm d}}_{r} [\partial_r \mathcal F ]q'rdr + \varepsilon \int_R^{+\infty} {\rm{d}}_{r} \mathcal G q'rdr.
\end{equation}
However we wanted to eliminate higher space derivatives on $\mathcal{F}$.
Integrating by parts on the right hand term gives
\begin{align*}
\partial_t \frac{\|\zeta', q', \kappa\partial_r q'\|_{L^2_r}^2} {2} - (\kappa^2 \underline{\ddot{\zeta}'} + \underline{\zeta'} - \varepsilon \underline{\partial_r \mathcal F}) \underline{q'} \leq & \varepsilon \kappa^{-1} \|\partial_r \mathcal F\|_{L^2_r} \|\kappa \partial_r q'\|_{L^2_r} \\
&+ \varepsilon \left( \|{\rm{d}}_{r} \mathcal{G}\|_{L^2_r} + \|\frac{1}{r} \partial_r \mathcal{F} \|_{L^2_r}\right) \|q'\|_{L^2_r}.
\end{align*}
Injecting the boundary condition \eqref{eq:BCq'} leads to
\begin{align*}
\partial_t \frac{|\zeta', q', \kappa\partial_r q'|_{L^2_r}^2} {2} + (\kappa^2 \underline{\ddot{\zeta}'} + \underline{\zeta'} - \varepsilon \underline{\partial_r \mathcal F}) \underline{\dot \zeta} \leq & \varepsilon \kappa^{-1} \|\partial_r \mathcal F\|_{L^2_r} \|\kappa \partial_r q'\|_{L^2_r} \\
&+ \varepsilon \left( \|{\rm{d}}_{r} \mathcal{G}\|_{L^2_r} + \|\frac{1}{r} \partial_r \mathcal{F} \|_{L^2_r}\right) \|q'\|_{L^2_r}.
\end{align*}
Moreover when we take the trace of the space derivative $\partial_r$ of the hidden equation \eqref{eq:HiddenEq} (in the meaning of $\partial_r$) we get 
\begin{equation}\label{eq:derivTrace}
\kappa^2 \underline{\ddot{\zeta'}} + \underline{\zeta'} - \varepsilon \underline{\partial_r \mathcal F} = - \varepsilon \underline{\mathcal G} - \ddot{\delta} \kappa \left (\partial_r G\right )(R).
\end{equation}
By the definition of $G$ given in Lemma \ref{lem:EqTrace} one deduces 
\begin{align*}
\kappa \partial_r G(r) &= -\frac{R}{2} \frac{K_1(r/\kappa)}{K_1(R/\kappa)},
\end{align*}
and therefore
$$
\kappa \left (\partial_r G\right )(R) = -\frac{R}{2}.
$$
Hence if we inject the expression of $\ddot{\delta}$ given by 
\begin{equation}
\begin{cases}
\tau_\kappa^2(0)\ddot{\delta} + \delta =  \underline \zeta + \kappa^2 \underline{\ddot{\zeta}} + \varepsilon \mathcal{F}_m [V^\ast] + F_{\rm ext},\\
\kappa^2 \underline{\ddot{\zeta}} + \underline{\zeta} = \underline{\mathfrak R_1 \zeta} - \kappa\ddot{\delta} G(R) + \varepsilon \underline{\mathcal{F}[V^\ast]} - \varepsilon \underline{\mathfrak R_1 \mathcal{F}[V^\ast] } + \varepsilon \kappa^2 \underline{ {\rm{d}}_{r} \mathfrak R_0 \mathcal{G}[V^\ast]},
\end{cases}
\end{equation}
into the differentiated hidden equation \eqref{eq:derivTrace} we obtain
\begin{equation}
\kappa^2 \underline{\ddot{\zeta'}} + \underline{\zeta'} - \varepsilon \underline{\partial_r \mathcal F} = \frac{R}{2\tau_\kappa^2(0)} \left(\eps \mathcal{F}_m + F_{\rm ext} - \delta + \kappa^2 \underline{\ddot{\zeta}} + \underline{\zeta} \right) - \varepsilon \underline{\mathcal{G}}.
\end{equation}
We deduce that
\begin{align*}
\partial_t E' - \frac{R}{2\tau_\kappa^2(0)} \delta \underline{\dot \zeta} +  \frac{R}{2\tau_\kappa^2(0)} \left(\eps\mathcal{F}_m + F_{\rm ext}\right) \underline{\dot \zeta} - \varepsilon \underline{\mathcal{G}} \underline{\dot \zeta} \leq & \varepsilon \kappa^{-1} \|\partial_r \mathcal F\|_{L^2_r} \|\kappa \partial_r q'\|_{L^2_r} \\
&+ \varepsilon \left( \|{\rm{d}}_{r} \mathcal{G}\|_{L^2_r} + \|\frac{1}{r} \partial_r \mathcal{F} \|_{L^2_r}\right) \|q'\|_{L^2_r},
\end{align*}
that is to say
\begin{align*}
\partial_t \left(E' - \frac{R}{2\tau_\kappa^2(0)} \delta \underline{\zeta} \right)\lesssim & \big[ \kappa^{-1} \left( \eps\left (\|\partial_r \mathcal F\|_{L^2_r} + |\mathcal{F}_m| + | \underline{\mathcal{G}} | \right ) +|F_{\rm ext}| \right) \\
&+ \left( \|{\rm{d}}_{r} \mathcal{G}\|_{L^2_r} + \|\frac{1}{r} \partial_r \mathcal{F} \|_{L^2_r} \right) \big] \sqrt{E'} + \frac{R}{2\tau_\kappa^2(0)}|\dot \delta \underline{\zeta}|.
\end{align*}
Because of the 2D cylindric situation one has 
$$ \|\partial_r \mathcal{F} \|_{L^2_r} \lesssim \|{\rm{d}}_{r} \mathcal{F} \|_{L^2_r}+ \frac{1}{R}\|\mathcal{F} \|_{L^2_r}, $$
and 
$$\|\frac{1}{r} \partial_r \mathcal{F} \|_{L^2_r} \lesssim \frac{1}{R}\|{\rm{d}}_{r} \mathcal{F} \|_{L^2_r} + \frac{1}{R^2}\|\mathcal{F} \|_{L^2_r}.$$
Hence we write the inequality as
\begin{equation}\label{eq:InegE'_lem}
\partial_t \left(E' - \frac{R}{2\tau_\kappa^2(0)} \delta \underline{\zeta} \right) \lesssim  \kappa^{-1} \left( \eps \left (\max(1, \frac{1}{R^2})\|\mathcal F\|_{H^1_r} + |\mathcal{F}_m| + \| \mathcal{G}\|_{H^1_r}\right )+ |F_{\rm ext}| \right)\sqrt{E'} + \frac{R}{2\tau_\kappa^2(0)}|\dot \delta |\sqrt{E'}.
\end{equation}
\end{proof}
\begin{remark}
It is important to note that we need to use different derivative operators between $\zeta$ and $q$.
Otherwise we would obtain either a commutation term from $\partial_r \dr q$ or a commutation term from $\dr \partial_r \zeta$. Indeed, one has $\partial_r \dr \cdot = \dr \partial_r \cdot - \frac{\cdot}{r^2}$. 
\end{remark}
To conclude we need to deal with the terms $\delta \underline\zeta$ in the left hand-side and $\dot{\delta}\underline{\zeta}$ in the right hand-side of equation \eqref{eq:InegE'_lem}. 
Indeed we want to avoid having a nonlinear term without $\eps$ in front of it. 
This is why we aim at separating $\delta \underline\zeta$ and $\dot{\delta}\underline{\zeta}$. We treat this through the following corollary. If we integrate Lemma \ref{lemma:4.2} with respect to the time we deduce
\begin{cor}\label{cor:4.1}
Let $\alpha = \max(2R^+, 9)$. Let $T> 0$ and $V^\ast  := (u^\ast, Z^\ast) \in \mathscr C^0_t ([0,T), \mathbb{H}^1 \times \mathbb{R}^4)$. 
Let initial conditions $(\zeta, q, \delta, \dot{\delta}, \underline{\zeta}, \underline{\dot{\zeta}})_{|_{t = 0}} \in \mathbb{H}^1 \times \mathbb{R}^4$. 
Then any solution $(\zeta, q, \delta, \dot{\delta}, \underline{\zeta}, \underline{\dot{\zeta}})$ of \eqref{eq:SLFormGeneral} combined with \eqref{eq:ODEglob} satisfies for $0 \leq t < T$
\begin{align*}
\sqrt{E'(t)} \lesssim &\sqrt{\beta_0} + \eps T(T+\sqrt{\alpha})||\mathfrak F||_{L^\infty([0,T))} + \kappa^{-1}||F_{\rm ext}||_{L^1([0,T))} \\& + \varepsilon \kappa^{-1}\int_0^T \left( \max(1, \frac{1}{R^2}) \|\mathcal F\|_{H^1_r} + |\mathcal{F}_m| + \| \mathcal{G}\|_{H^1_r} \right ),
\end{align*}
with $\beta_0 = E'(0) - \frac{R}{2\tau_\kappa^2(0)}\delta(0)\underline{\zeta}(0) + \frac{\alpha}{2\tau_\kappa^2(0)}E(0)$, $E$, $E'$ as in Lemmata \ref{lemma:4.1} and \ref{lemma:4.2}, and $\mathfrak{F}$ is given in Lemma \ref{lemma:4.1}.
\end{cor}
\begin{proof}
Let $T > 0$, $V^\ast  := (u^\ast, Z^\ast) \in \mathscr C^0_t ([0,T), \mathbb{H}^1 \times \mathbb{R}^4)$ and $(\zeta, q, \delta, \dot{\delta}, \underline{\zeta}, \underline{\dot{\zeta}})$ a solution of \eqref{eq:SLFormGeneral}, \eqref{eq:ODEglob}.
Let $ S := \max(1, \frac{1}{R^2})\|\mathcal F\|_{H^1_r} + \| \mathcal{G}\|_{H^1_r} + |\mathcal{F}_m| + \frac{1}{\eps}|F_{\rm ext}|$. Integrating Lemma \ref{lemma:4.2} with respect to the time over $[0, t)$ gives 
\begin{equation}\label{eq:CorProof1}
E'(t) - \frac{R}{2\tau_\kappa^2(0)} \delta(t) \underline{\zeta}(t) \lesssim E'(0) - \frac{R}{2\tau_\kappa^2(0)}\delta(0) \underline{\zeta}(0) + \varepsilon \kappa^{-1} \int_0^T S \sqrt{E'} + \frac{R}{2\tau_\kappa^2(0)}\int_0^T |\dot{\delta}| \sqrt{E'}.
\end{equation}
In this equation we bounded from above the integrals on $[0, t)$ by the integrals on $[0, T)$ because the quantities are positives. 
On one hand with Lemma \ref{lemma:4.1} one has
$$ |\delta, \dot{\delta}| \leq \sqrt{E(0)} + T\eps ||\mathfrak F||_{L^\infty([0,T))} + ||F_{\rm ext}||_{L^1([0,T))}. $$
On the other hand the term $\delta \underline{\zeta}$ is controlled via a Young's inequality
\begin{equation}
\delta(t) \underline{\zeta}(t) \leq \frac{\alpha}{R} \delta(t)^2 + \frac{R}{\alpha} \underline{\zeta}(t)^2,
\end{equation}
therefore
\begin{equation}
E'(t) - \frac{R}{2\tau_\kappa^2(0)} \delta(t) \underline{\zeta}(t) \geq E'(t) - \frac{\alpha}{2\tau_\kappa^2(0)}\delta(t)^2 - \frac{R^2}{2\alpha\tau_\kappa^2(0)}\underline{\zeta}(t)^2.
\end{equation}
By definition of $E'$ (see Lemma \ref{lemma:4.2}), since $\alpha > R$, 
\begin{equation}
E'(t) \lesssim E'(t) - \frac{R^2}{2\alpha\tau_\kappa^2(0)}\underline{\zeta}(t)^2.
\end{equation}
Injecting the two last equations into \eqref{eq:CorProof1} leads to
\begin{align*}
E'(t) \lesssim & \frac{\alpha}{2\tau_\kappa^2(0)}\left( E(0) + (T\eps ||\mathfrak F||_{L^\infty([0,T))} + ||F_{\rm ext}||_{L^1([0,T))} )^2\right) + E'(0) \\ 
&- \frac{R}{2\tau_\kappa^2(0)}\delta(0) \underline{\zeta}(0) +  \int_0^T \left (\varepsilon \kappa^{-1} S + \frac{R}{2\tau_\kappa^2(0)} |\dot{\delta}|\right ) \sqrt{E'}.
\end{align*}
We set $\beta_0 := \frac{\alpha}{2\tau_\kappa^2(0)} E(0) + E'(0) - \frac{R}{2\tau_\kappa^2(0)}\delta(0) \underline{\zeta}(0)$. Then for $\alpha \geq 9$, $\beta_0$ is non-negative. Indeed one has the following Young's inequality
$$
\delta(0) \underline{\zeta}(0) \leq \frac{\alpha - 1}{4} \delta(0)^2 + \frac{4}{\alpha - 1} \underline{\zeta}(0)^2,
$$
and
$$ 
\frac{\alpha R}{8\tau_\kappa^2(0)}\delta(0)^2 + \frac{R}{4\tau_\kappa^2(0)} \underline{\zeta}(0)^2 - \frac{R}{2\tau_\kappa^2(0)}\delta(0) \underline{\zeta}(0) \geq \frac{\alpha R}{2\tau_\kappa^2(0)}\left ( (\frac{\alpha}{4} - \frac{\alpha - 1}{4})\delta(0)^2 + (\frac{1}{2} - \frac{4}{\alpha - 1})\underline{\zeta}(0)^2\right).
$$
It proves that $\beta_0 \geq 0$ since $\alpha \geq 9$. 
By the Gronwall lemma we obtain
\begin{equation}
\sqrt{E'(t)} \lesssim \sqrt{\beta_0} + \sqrt{\alpha}\left( T \eps \|\mathfrak{F}\|_{L^\infty([0, T))} + \|F_{\rm ext}\|_{L^1([0, T))}\right) +  \int_0^T \left (\varepsilon \kappa^{-1} S + \frac{R}{2\tau_\kappa^2(0)} |\dot{\delta}|\right).
\end{equation}
Finally using Lemma \ref{lemma:4.1}, one obtains 
\begin{equation}\label{eq:int_dotdelta}
\int_0^T |\dot{\delta}| \lesssim T\sqrt{E(0)} + \eps T^2 \|\mathfrak{F}\|_{L^\infty([0, T))} + \|F_{\rm ext}\|_{L^1([0, T))},
\end{equation}
which leads to the lemma. 
\end{proof}
Finally summing Lemma \ref{lemma:4.1} and Corollary \ref{cor:4.1} leads to the inequiality of Proposition \ref{prop:EstimGlobale}. 
\begin{remark}
One could get better estimates by considering the more precise estimate
$$ \delta(t)^2 \leq E(0) + \eps^2 ||\mathfrak{F}||_{L^1([0,T))}^2 + ||F_{\rm ext}||_{L^1([0,T))}^2, $$ instead of $$ \delta(t)^2 \leq E(0) + T^2 \eps^2 ||\mathfrak{F}||_{L^\infty([0,T))}^2 + ||F_{\rm ext}||_{L^1([0,T))}^2,$$ but we did not manage to make it useful.
\end{remark}
\indent Thanks to Proposition \ref{prop:EstimGlobale}, any solution of the system \eqref{eq:SLFormGeneral}, \eqref{eq:ODEglob} where the initial conditions are null satisfies
\begin{equation}\label{eq:InegFP}
|| V ||_{\mathscr{C}_t \left( \mathbb{H}^1 \times \mathbb{R}^4 \right)} \lesssim \varepsilon T \left(T + \max(1, \frac{1}{R^2}) \kappa^{-1} + \sqrt{\alpha} \right)\left ( ||\mathcal{F}, \mathcal{G}||_{L^\infty_t([0, T), H^1_r)}+ \| \mathcal{F}_m \|_{L^\infty_t([0, T))}\right ). 
\end{equation}
 
\subsubsection{Application of the fixed point}
We now go back to the first problem, meaning that we apply the estimate of Proposition \ref{prop:EstimGlobale} with the quantities described at the beginning of Subsection \ref{sec:4.2}. That is to say for $V = (\zeta, q, \delta, \dot{\delta}, \underline{\zeta}, \underline{\dot{\zeta}})$, one uses
\begin{itemize}
\item $\mathcal{F}[V] := \frac{\zeta^2}{2} + \frac{q^2}{h}$,
\item $\mathcal{G}[V] := \frac{q^2}{rh}$, for $r > R$,
\item $\mathcal{F}_m[V] := \gamma = \frac{R^2}{8h_i \left( \tau^2_\kappa (\varepsilon \delta) + G\right) }\delta \left( \eps a(\delta, \underline{\zeta})\dot{\delta}^2 - \eps \underline{f} - \delta + \underline{f_{\rm hyd}} + F_{\rm ext}\right) + a(\delta, \underline{\zeta}) \dot{\delta}^2$.
\end{itemize}
\indent Let $V_1 = (u_1, Z_1)$, $V_2 = (u_2, Z_2)$ $\in \mathscr C^0_t([0, T), \mathbb{H}^1 \times \mathbb{R}^4)$ with the same initial data.
We note $h_k = 1 + \eps \zeta_k$ for $k = 1, 2$.
With Lemma \ref{lem:OpDisp} we obtain from Section \ref{sec:estimODE} the inequality
$$ |\underline{f_{{\rm hyd}, 1}} - \underline{f_{{\rm hyd}, 2}}|\leq \eps c\left(\|\frac{1}{h_1}, \frac{1}{h_2}, q_1, q_2, \zeta_1, \zeta_2\|_{L^\infty}\right)\|V_1 - V_2\|_{\mathbb{H}^1 \times \mathbb{R}^4},$$
where $f_{{\rm hyd}, k}$ stands for $f_{{\rm hyd}}$ for $V_k$, $k \in \{1, 2\}$, and $c$ is a constant independent of $\eps$ and $\kappa$.
With the expression of $a$ is given in Proposition \ref{prop:AddedMass} and thanks to Moser type inequalities \eqref{eq:MTineq} one deduces that
$$|\mathcal{F}_m[V_1] - \mathcal{F}_m[V_2]| \leq \eps C\left( \|\frac{1}{h_1}, q_1, \zeta_1, \frac{1}{h_2}, q_2, \zeta_2\|_{L^\infty_r}, |\dot{\delta}_1, \dot{\delta}_2|, |F_{\rm ext}|\right) || V_1 - V_2 ||_{\mathbb{H}^1 \times \mathbb{R}^4},$$
where $c$ is a constant independent of $\eps$ and $\kappa$.
From Moser-type inequalities \eqref{eq:MTineq} we also obtain that
\begin{equation}
||\mathcal{F}[V_1] - \mathcal{F}[V_2] ||_{H^1_r} \leq C_1 \left(\|\frac{1}{h_1}, q_1, \zeta_1, \frac{1}{h_2}, q_2, \zeta_2\|_{L^\infty_r}, |\dot{\delta}_1, \dot{\delta}_2|, \underset{\mathbb{R}^*_+\setminus\{R\}}{\vinf} \left[\frac{1}{h_1}, \frac{1}{h_2}\right] \right) ||u_1 - u_2||_{\mathbb{H}^1},
\end{equation}
and
\begin{equation}
||\mathcal{G}[V_1] - \mathcal{G}[V_2] ||_{H^1_r}  \leq \frac{1}{R}C_2\left(||u_1, u_2||_{L^\infty_r}, \underset{\mathbb{R}^*_+\setminus\{R\}}{\vinf} \left[\frac{1}{h_1}, \frac{1}{h_2}\right] \right) (||u_1 - u_2||_{\mathbb{H}^1}),
\end{equation}
where $C_1$ and $C_2$ are constant independent of $\eps$ and $\kappa$.
Therefore with inequality \eqref{eq:InegFP}, for $V_1$, $V_2$ solutions of \eqref{eq:AUG} with the same initial condition we obtain
\begin{equation}
||\phi(V_1) - \phi(V_2) ||_{\mathscr{C}^0_t \left( \mathbb{H}^1 \times \mathbb{R}^4 \right)} \lesssim  \varepsilon R^{-1} T \left(T + \max(1, \frac{1}{R^2}) \kappa^{-1} + \sqrt{\alpha} \right) C || V_1 - V_2 ||_{\mathscr{C}^0_t \left( \mathbb{H}^1 \times \mathbb{R}^4 \right)},
\end{equation}
where $C = C\left(||V_1, V_2||_{\mathbb{H}^1 \times \mathbb{R}^4}, \underset{\mathbb{R}^*_+\setminus\{R\}}{\vinf} \left[\frac{1}{h_1}, \frac{1}{h_2}\right], |F_{\rm ext}| \right)$ is a constant independent of $\eps$ and $\kappa$ and $\phi$ stands for 
$$ 
\phi : V^\ast = (u^\ast, Z^\ast) \in \mathscr{C}_t \left( \mathbb{H}^1 \times \mathbb{R}^4 \right) \mapsto V  \in \mathscr{C}_t \left( \mathbb{H}^1 \times \mathbb{R}^4 \right),
$$ 
where $V$ is the solution of the system \eqref{eq:SLFormGeneral} associated with the ODE's system \eqref{eq:ODEglob} and with the non-linearities adapted to \eqref{eq:AUG}
\begin{equation*}
\begin{cases}
\mathcal{D}_\kappa \partial_t u + \mathcal L [u] = \varepsilon \partial_r \begin{pmatrix} 0 \\
\mathcal F[V^\ast] \end{pmatrix} + \varepsilon \begin{pmatrix} 0 \\ \mathcal G[V^\ast] \end{pmatrix}, \\
\underline q = - \frac{R}{2}\dot\delta, \\
\tau_\kappa^2(0)\ddot{\delta} + \delta =  \underline \zeta + \kappa^2 \underline{\ddot{\zeta_e}} + \varepsilon \mathcal{F}_m [V^\ast] + F_{\rm ext},\\
\kappa^2 \underline{\ddot{\zeta}} + \underline{\zeta} = \underline{\mathfrak R_1 \zeta} - \ddot{\delta} G + \varepsilon \underline{\mathcal{F}[V^\ast]} - \varepsilon \underline{\mathfrak R_1 \mathcal{F}[V^\ast] } + \varepsilon \kappa^2 \underline{ {\rm{d}}_{r} \mathfrak R_0 \mathcal{G}[V^\ast]}.
\end{cases}
\end{equation*}
Hence if $T$ satisfies $T = cT_{\eps, \kappa, R}$, with 
$$
T_{\eps, \kappa, R} = \frac{-(\sqrt{\alpha} +\max(1, \frac{1}{R^2})\kappa^{-1}) + \sqrt{(\sqrt{\alpha} +\max(1, \frac{1}{R^2})\kappa^{-1})^2 + 4\eps^{-1}R}}{2},
$$
and $c > 0$ of order 1 depending only on the initial conditions, then the upper bound verifies $\varepsilon R^{-1} T \left(T + \max(1, \frac{1}{R^2}) \kappa^{-1} + \sqrt{\alpha} \right) C  < 1$ and $\phi$ is contractant. Therefore with the Picard fixed point theorem one obtains existence and unicity of a fixed point of $\phi$. It proves the wellposedness of the Augmented formula in $\mathscr{C}^0\left([0, T), \mathbb{H}^1 \times \mathbb{R}^4\right)$.
 
One can remark different regimes given by the time existence.
If we are in the regime $\eps \ll \kappa^2$, then $T_{\eps, \kappa, R} \sim \eps^{-1/2}R^{1/2} - \left(\sqrt{\alpha} + \max(1,\frac{1}{R^2})\kappa^{-1}\right)$. This regime is only due to the boundaries and it is obtained with inequality \eqref{eq:int_dotdelta}. 
On the contrary if $\eps \gg \kappa^2$, then $T_{\eps, \kappa, R} \sim \frac{\eps^{-1/2}R^{1/2}}{\sqrt{\alpha} + \max(1,\frac{1}{R^2})\kappa^{-1}} $, and it confirms that this method prevents $\kappa$ from tending to 0. However $T_{\eps, \kappa, R}$ tends to $0$ much slower than $T_{\rm ODE}$ when $\kappa \to 0$.
This description of the time existence also gives $T \to 0$, for high Helmholtz numbers ($R\to \infty$) and weak Helmholtz numbers ($R \to 0$). Therefore this method does not give either of the limits for the Helmholtz number, to the contrary of $T_{\rm ODE}$ which gives the limit for weak and high Helmholtz numbers limit.
However the solutions given by the ODE's perspective need less regularity than the solutions given by the semi-linear scheme.

Therefore, except for the high Helmholtz numbers limit, $T_{\eps, \kappa, R}$ is better than $T_{\rm ODE}$.
\begin{remark}\label{rk:WPSL_visc}
To consider the viscosity $\nu > 0$ one must consider the component $\nu h\partial_r \frac{\dr q}{h} $. There are two ways to consider it.
\begin{enumerate}
\item We write $h\partial_r \frac{\dr q}{h} = \dr q + \partial_r(\ln(h)) \dr q$ and we consider the term $\partial_r(\ln(h)) \dr q$ in $\mathcal{G}$ in \eqref{eq:SLFormGeneral}. Therefore to obtain an estimate over $\|\mathcal{G}\|_{H^1_r}$ with Moser type inequalities one must have an interest in higher space derivative order as in Lemma \ref{lemma:4.2}.
\item The second way to consider it is to use a quasi-linear scheme and treat the viscous term as 
$$ h^k \partial_r \frac{\dr q ^{k+1}}{h^k}. $$
Therefore it is possible to treat it with the third method explained in the introduction.
\end{enumerate}
\end{remark}
 
\begin{remark}\label{rk:WPSL_ZetaOvH}
If we had considered the same energy balance as in \cite{BeckLannes} (see Remark \ref{rk:EnBudg}) we would have in the modified Newton's equation a factor $\frac{\underline{\ddot{\zeta_e}}}{\underline{h_e}}$
\begin{equation}
\tilde{\tau}_\kappa^2( \varepsilon \delta) \ddot{\delta} + \delta -  \varepsilon \tilde{a}( \varepsilon \delta, \underline{\zeta_e}) \, \dot{\delta}^2
= \kappa^2 \frac{\underline{\ddot{\zeta_e}}}{\underline{h_e}} + \underline{\zeta_e} + F_{\rm ext}, 
\end{equation}
and it would have been much harder to consider a semi-linear scheme. Indeed in this case we cannot separate the non-linear terms over higher order derivatives $\underline{\ddot{\zeta}}$ and $\ddot{\delta}$, and a quasi-linear scheme is needed.
\end{remark}
 
For $k \in \mathbb{N}$ if we apply $\partial_t^k$ to the Augmented formula \eqref{eq:AUG} we think that we can use the same kind of semilinear scheme to show the following conjecture.
\begin{conj}
For any $k \in \mathbb{N}$, the Augmented formula is wellposed in  $\mathscr{C}^k\left([0, cT_{\eps, \kappa, R}), \mathbb{H}^1 \times \mathbb{R}^4\right)$.
\end{conj}
Also for $n\in \mathbb{N}$ if we apply 
$\begin{pmatrix}
\partial_r {\rm{d}}_{r} \cdots {\rm{d}}_{r} \partial_r \\
{\rm{d}}_{r} \partial_r \cdots \partial_r {\rm{d}}_{r}
\end{pmatrix} $ if $n$ is odd and $\begin{pmatrix}
{\rm{d}}_{r} \partial_r \cdots {\rm{d}}_{r} \partial_r \\
\partial_r {\rm{d}}_{r} \cdots \partial_r {\rm{d}}_{r} 
\end{pmatrix} $ if $n$ is even to {\rm \hyperlink{BAr}{(BAr)}}, we think that the following conjecture is true.
 
\begin{conj}
Let $n\geq 2 \in \mathbb{N}$. The Augmented formula \eqref{eq:AUG} wellposed in $\mathscr{C}^k\left([0, T), \mathbb{H}^n \times \mathbb{R}^4\right)$, 
where $ \mathbb{H}^{n}_\kappa $ is a 2D vectorial field equipped with the weighted norm 
\begin{align*}
||u||_{\mathbb{H}^{n}}^2 := &\| u_1 \|_{L^2_r}^2 + \| \partial_r u_1 \|_{L^2_r}^2 + \| \dr \partial_r u_1 \|_{L^2_r}^2 +\cdots +\| (\partial_r \dr)^{\frac{n}{2}} u_1\|_{L^2_r}^2 \\&
+ ||u_2||_{L^2}^2 + \|\dr u_2\|_{L^2_r}^2 + \|\partial_r \dr u_2\|_{L^2_r}^2 + \cdots + \| (\dr \partial_r)^{\frac{n}{2}} u_2\|_{L^2_r}^2 + \kappa^2 \| (\dr \partial_r)^{\frac{n}{2}} \dr u_2\|_{L^2_r}^2,
\end{align*}
if $n$ is even or 
\begin{align*}
||u||_{\mathbb{H}^{n}}^2 := &\| u_1 \|_{L^2_r}^2 + \| \partial_r u_1 \|_{L^2_r}^2 + \| \dr \partial_r u_1 \|_{L^2_r}^2 + \cdots + \| (\dr \partial_r)^{\frac{n-1}{2}} \dr u_1 \|_{L^2_r}^2\\&
+ ||u_2||_{L^2}^2 + \|\dr u_2\|_{L^2_r}^2 + \|\partial_r \dr u_2\|_{L^2_r}^2 + \cdots +\| \dr (\partial_r\dr)^{\frac{n-1}{2}}  u_2\|_{L^2_r}^2 + \kappa^2\| (\partial_r\dr)^{\frac{n+1}{2}}  u_2\|_{L^2_r}^2,
\end{align*}
if $n$ is odd.
\end{conj}
 
\section{The case of the return to the equilibrium in the linear situation}\label{sec:RTTE}
The return to equilibrium problem consists in dropping the floating solid in a water at rest without giving it any speed or any force. Thus one considers the following initial conditions
$$
(\zeta, q, \delta, \dot{\delta} )_{\vert_{t=0}} = (0,0, \delta_0 \neq 0, 0).
$$
In the linear case, that is to say $\eps = 0$, the previous section showed that there exists a unique global solution of \eqref{eq:AUG}, and we want to describe the behavior of the motion of the solid at $t \rightarrow +\infty$.
One shows (see Theorem \ref{thm:pt5.1.1}) that the decay cannot be faster than $\mathcal{O}(t^{-1/2})$ in 2D without viscosity. With viscosity, Theorem \ref{thm:pt5.1.2} shows that the decay cannot be faster than $\mathcal{O}(t^{-3/2})$. 
 However, the method shows that in 2D it is impossible to obtain exponential decreasing of $\delta$.
 Another issue consists in obtaining an ODE on the vertical motion of the floating cylinder without computing the Boussinesq's equation: the Cumming's equation.
In the 1D linear Boussinesq regime without viscosity, such an equation was derived in \cite{Lannes_float} and can be written
\begin{equation}\label{eq:Cummins-1d}
m \ddot{\delta} + \left( \frac{J_1(t/\kappa)}{t} \right) \ast_t \dot{\delta} + \delta = 0,
\end{equation}
where $J_1$ is a Bessel function and $\ast_t$ is the convolution with respect to the time. 
In the 2D case with viscosity, one shows that one has
\begin{equation}\label{eq:Cummins-2d}
[\tau_\kappa^2(0) + k_B \ast] \ddot{\delta} + \nu k_{0} \ast_t  \partial_t^{\frac{1}{2}} \dot{\delta} + k_{1} \ast_t  \dot{\delta} + \delta = 0,
\end{equation}
where the kernels $k_B,k_0,k_1$ have explicit expressions in the Laplace domain and $\partial_t^{\frac{1}{2}}$ stands for the half-derivative. Note that when both the viscosity and the dispersion vanish, one recovers the equation given in \cite{Bocchi}. It will be more convient to work in the Laplace domain where the kernels of the Cummin's equation \eqref{eq:Cummins-2d} have an explicit expression. Thus one needs to introduce some notations. The Laplace transform with respect to the time variable of a locally integrable function $u$ is defined by
$$
 \hat{u}(s) := \mathcal{L} [u] (s) := \displaystyle \int_0^{\infty} u(t) e^{-st} dt,
$$
for all complex numbers $s = \eta + i \omega$ such that this integral absolutely converges. We will seek to explain the behavior of the solid by explicitly knowing the Laplace transform of $\delta$. To do so, we will use the following theorem.
 
\begin{thm}\label{thm:Bergman} [Paley-Wiener Theorem for Bergman spaces.]
Let $\alpha \geq -1$,
the Laplace transform 
$$
\mathcal{L}: L^2\left (t^{-(\alpha+1)}dt, \mathbb{R}^+ \right ) \to \mathcal{B}_{\alpha}
$$ 
is a linear continous isomorphism between Hilbert on the holomorphic functions space $Hol(\mathbb{C})$ where the Bergman spaces are defined by
\[
\mathcal{B}_{\alpha}:=\left \{F \in Hol(\mathbb{C}_0) \, | \, \|F\|_{\mathcal{B}_{\alpha}}^2:=\int_{0}^\infty \int_{\mathbb R} |F(\eta+i\omega)|^2 \, \eta^{\alpha} \,d\omega\,d\eta<\infty\right \},
\]
for any $\alpha > -1$ and
\[
\mathcal{B}_{-1}:=\Big\{F \in Hol(\mathbb{C}_0) \, | \, \|F\|_{\mathcal{B}_{-1}}^2:= \sup_{\eta>0}\int_{\mathbb{R}} |F(\eta+i\omega)|^2 \,d\omega <\infty\Big\},
\]
and where
$$
\mathbb{C}_{0} := \{ s \in \mathbb{C}, \, \Re(s) > 0 \}.
$$
\end{thm}
Proof of this theorem and precisions on Bergman spaces are given in appendix \ref{app-Laplace}. Laplace transform is also useful to define non-local time operators. For example the half-derivative is defined as the operator associated to the symbol $\sqrt{s}$:
$$
\forall u \in L^2, \, \partial_t^{\frac{1}{2}} u := \mathcal{L}^{-1} [ \sqrt{s} \widehat{u} ].
$$ 
Note that $ \sqrt{s}$ is defined as the square root of $s$ with a non-negative real part :
\begin{equation}\label{def:sqrt_sep}
\forall s \notin \mathbb{R}^-, \, \sqrt{s} =  \sqrt{ \frac{ |s| + \Re(s) }{2} } + i \, \mathrm{sgn}[\Im(s)] \, \sqrt{ \frac{ |s| - \Re(s) }{2} }.
\end{equation}
It means that the branching cut for the square root is $\mathbb{R}^-$.
Such a choice imposes that the branching cut for the logarithm also is $\mathbb{R}^-$. Thus the modified Bessel functions are defined as in \cite{A&S} as
\begin{equation}\label{eq:defK}
s \in \mathbb{C}\setminus\mathbb{R}^-, \, K_\alpha(s) = \int_0^\infty e^{-s \, \text{cosh}(t)} \text{cosh} (\alpha t) dt,
\end{equation}
for any $\alpha \in \mathbb{R}$. One can easily verify that this definition is in accordance with Definition \ref{def:K}, with a time dependance rather than a space dependence.
 
We will show in Section \ref{sec:DeltaCummins} that $\delta$ satisfies an equation of type \eqref{eq:Cummins-2d} by solving the Boussinesq equation in the Laplace domain. 
In particular equation \eqref{eq:Cummins-2d} reads in the Laplace domain 
\begin{equation}
\widehat{\delta}(s) = \widehat{H}(s)\delta_0,
\end{equation}
where the transfer function $\widehat{H}$ is defined by
\begin{equation}\label{eq:transferH}
\widehat{H}(s) := \frac{2\tau_\kappa^2(0)s + 2\nu + R\sqrt{1+\nu s + \kappa^2 s^2} B(s) }{s\left(2\tau_\kappa^2(0)s + 2\nu + R\sqrt{1+\nu s + \kappa^2 s^2} B(s)\right) + 2}, \quad
B(s) :=  \frac{K_0\left(\frac{Rs}{\sqrt{1+\nu s + \kappa^2 s^2}}\right)}{K_1\left(\frac{Rs}{\sqrt{1+\nu s + \kappa^2 s^2}}\right)},
\end{equation}
and where $K_0$ and $K_1$ denotes the modified Bessel function of second kind of order 0 and 1 defined above. Due to the presence of branching cuts, the expression \eqref{eq:transferH} of the transfert function is valid only for complex numbers lying in
\begin{equation}\label{def:Dkapnu}
\mathbb{D}_{\kappa, \nu} = \{  s \in \mathbb{C} \, | \,  1+\nu s + \kappa^2 s^2\notin \mathbb{R}^- \,\text{ and }\, \frac{s}{\sqrt{1+\nu s + \kappa^2 s^2}} \notin \mathbb{R}^- \}.
\end{equation}
Then, in Section \ref{sec:5.2} we will deduce from the Cummin's equation and the Paley-Wiener Theorem the expected decay. 
Unfortunately, one cannot succeed to show that the denominator of the transfer function $\widehat{H}$ has no zero in some half complex plane even if numerical evidences go with our believes. 
This is why one must introduce an hypothesis:
\begin{equation}\label{conjecture-denom}
\exists \mathcal{P}_{\min} > 0, \, \forall s \in \overline{\mathbb{C}_0}, |\mathcal{P}(s) | \geq \mathcal{P}_{\min}, \quad
\mathcal{P}(s) := 2\tau_\kappa^2(0)s^2 + 2\nu s + Rs\sqrt{1+\nu s + \kappa^2 s^2} B(s) + 2.
\end{equation}

We remind that in all this section we work in the linear regime $\eps = 0$.
Let $(\zeta, q, \delta, \dot{\delta}, \underline{\zeta}, \underline{\dot{\zeta}})$ the solution of the Augmented formula \eqref{eq:AUG}  with initial conditions $(0,0,\delta_0 \neq 0, 0, 0, 0)$ given by Theorem \ref{thmf:WP}, where the initial conditions satisfy the compatibility conditions at $r = R$.
Then the following results on $\delta$ stand. 
\begin{thm}\label{thm:pt5.1.0}
Let $\kappa,\nu \geq 0$, then the vertical motion satisfies $\delta \in H^{2}(\mathbb{R}^+)$.
In particular
$$ 
\|\delta\|_{L^2(\mathbb{R}^+)} \lesssim \kappa_\ast^{-\frac{1}{2}} (\kappa_\ast^{-1} + 1+R(1+ \nu \kappa_\ast^{-1})) \mathcal{P}_{\min}^{-1} + \kappa_\ast^{-\frac{1}{2}},
$$
$$
\|{\dot{\delta}} \|_{L^2(\mathbb{R}^+)} \lesssim \kappa_\ast^{-\frac{1}{2}} \mathcal{P}_{\min}^{-1} + \frac{\kappa_\ast^{-\frac{3}{2}}}{\tau_{\rm buoy}^2}
\quad \text{and} \quad
\| {\ddot{\delta}} \|_{L^2(\mathbb{R}^+)} \lesssim \kappa_\ast^{-\frac{1}{2}} \mathcal{P}_{\min}^{-1} + \frac{\kappa_\ast^{-\frac{1}{2}}}{\tau_{\rm buoy}^2},
$$
with 
$$
\kappa_\ast= \begin{cases} \kappa & \text{if $\kappa >0$}, \\ 1 & \text{if $\kappa =0$}. \end{cases}
$$
\end{thm}
Furthermore, $\delta$ satisfies the following behaviors for large time scales, depending on the regime.
\begin{thm}[The dispersive non-viscous case]\label{thm:pt5.1.1}
Let $\nu = 0$ and $\kappa > 0$. Then for $k \in \left\{ 0, 1, 2 \right\}$, for all $\beta \in (0,2]$,
$$ 
\delta^{(k)} \notin L^2(t^{\beta}dt, \mathbb{R}^+).
$$
Furthermore for all $\rho \in ( 0, \frac{1}{2})$, for all $T_0 > 0$ and $c > 0$, there exists $t > T_0$ such that for $k \in \left\{ 0, 1, 2 \right\}$
$$ |\delta^{(k)}(t)| \geq c t^{-\frac{1}{2} - \rho}.$$
\end{thm}
\begin{thm}[The viscous or non-dispersive case]\label{thm:pt5.1.2}
Let $\nu > 0$ or $\kappa = 0$. Then for $k \in \left\{ 0, 1, 2 \right\}$, for all $ \beta \in (0, 2) $,
$$ 
t^{k}\delta^{(k)} \in L^2(t^{\beta}dt, \mathbb{R}^+),
\quad \text{ but } \quad
t^{k+1}\delta^{(k)} \notin L^2(\mathbb{R}^+).
$$
Furthermore for all $\rho \in ( 0, \frac{1}{2})$, for all $T_0 > 0$ and $C > 0$, there exists $t > T_0$ such that for $k \in \left\{ 0, 1 \right\}$
$$ C t^{-\frac{3}{2} - k - \rho} \leq |\delta^{(k)}(t)| \leq C t^{-k-1}, \quad \text{and} \quad C t^{-\frac{7}{2} - \rho} \leq |\ddot{\delta}(t)|.$$
\end{thm}
Those theorems are quite similar to the one found by the first author in \cite{BeckLannes} in 1D when $\nu = 0$, but we obtain a stronger decay in $\mathcal O(t^{-1/2})$ rather than $\mathcal O(t^{-3/2})$.  
However the method that we use in this paper also works for the 1D case and one can show that the decay is also in $\mathcal{O}(t^{-1/2})$ (see Section \ref{sec:1D}). In particular in 1D the Assumption \eqref{conjecture-denom} is true (see \cite{BeckLannes}).

In all the following subsections we assume that $(\zeta, q, \delta, \dot{\delta}, \underline{\zeta}, \underline{\dot{\zeta}})$ is the solution of the Augmented formula \eqref{eq:AUG}  with initial conditions $(0,0,\delta_0 \neq 0, 0, 0, 0)$ given by Theorem \ref{thmf:WP}, where the initial conditions satisfy the compatibility conditions at $r = R$.
 
\subsection{Cummin's equation}\label{sec:DeltaCummins}
First of all, one shows that $\delta$ follows a Cummin's equation as \eqref{eq:Cummins-2d}. 
To do so we will begin by giving an expression for $\widehat{\zeta}$ and $\widehat{q}$.
Afterwards we are going to discuss about the space $\mathbb{D}_{\kappa,\nu}$ and its branching points and branching cuts.
Finally we will show the Theorem \ref{thm:pt5.1.0}.
\begin{thm}\label{thm:DeltaCummins}
Let $\kappa, \nu \geq 0$.
Then $\delta$ satisfies the Cummin's equation
\begin{equation}\label{eq:CumminsDelta}
[\tau_\kappa^2(0) + \kappa k_B \ast] \ddot{\delta} + \sqrt{\nu} (k_{0} \ast k_B) \ast  \partial_t^{\frac{1}{2}} \dot{\delta} + \nu \dot{\delta}+ (k_{1} \ast k_B) \ast  \dot{\delta} + \delta = 0,
\end{equation}
where the kernels $k_B,  k_{0}, k_{1}$ have the following explicit expressions in the Laplace domain 
$$ 
\widehat{k_B} = \frac{R}{2}B(s),
\,
\widehat{k_0} = \frac{1}{\sqrt{1 + \frac{\kappa^2}{\nu} s} + \frac{\kappa}{\sqrt{\nu}} \sqrt{s}},
\, \text{ and } \,
\widehat{k_1} = \frac{1}{\sqrt{1+\nu s + \kappa^2 s^2} + \sqrt{\nu s + \kappa^2 s^2}}.
$$
and the following expression in time domain
\begin{center}
\begin{tabular}{|c|c|c|c|}
  \hline
  {\rm Parameters} & $\nu =0$, $\kappa > 0$ & $\nu >0$, $\kappa =0$ & $\nu >0$, $\kappa > 0$ \\ \hline
  $k_0$ & $k_0 = 0$ & Dirac mass & $k_0(t) = \frac{\kappa}{\sqrt[4]{\nu}} \frac{1 - e^{-t\frac{\nu}{\kappa^2}}}{2\pi t^3} \in L^1$ \\ \hline
  $k_1$ & $k_1(t)=\frac{1}{t}J_1(\frac{t}{\kappa}) \in L^1$ (see \cite{BeckLannes}) & $ k_1(t)=\sqrt{\nu} \frac{1-e^{-\frac{t}{\nu}}}{\sqrt{2\pi t^3}}\in L^1 $ & $k_1 \in L^2(e^{t\frac{\nu}{2\kappa^2}}dt)\subset L^1$ \\ \hline
\end{tabular}
\captionof{table}{Expression of $k_0$ and $k_1$ depending on $\nu$ and $\kappa$.}
\label{tab:kk}
\end{center}
\end{thm}

\begin{remark}
In 1D one has $B = 1$ and $k_B$ is the Dirac mass, thus the Cummin's equation \eqref{eq:CumminsDelta} is the same equation as in \cite{BeckLannes} in the 1D non-viscous case. In 2D non-viscous non-dispersive case, the Cummin's equation \eqref{eq:CumminsDelta} is the same equation as in \cite{Bocchi}.
\end{remark}
 
To show such a theorem one needs the expression of $\widehat{\zeta}$ and $\widehat{q}$. It is given by the following proposition.
 
\begin{proposition}\label{prop:zeta_q_Lap}
Let $\kappa, \nu \geq 0$. 
In Laplace domain, for all $s \in \mathbb{D}_{\kappa, \nu}$ (where $\mathbb{D}_{\kappa, \nu}$ has been introduced earlier, see \eqref{def:Dkapnu}), the radial discharge and the surface elevation satisfy the following explicit expressions
\begin{equation}\label{eq:zetaqLap}
\begin{cases}
\widehat{q}(r, s) = -\frac{R \widehat{\dot{\delta}}(s)}{2} K_1\left(\frac{rs}{\sqrt{1+\nu s + \kappa^2 s^2}}\right)/K_1\left(\frac{Rs}{\sqrt{1+\nu s + \kappa^2 s^2}}\right),\\
\widehat{\zeta}(r, s) =- \frac{R}{2} \frac{\widehat{\dot{\delta}}(s)}{\sqrt{1+\nu s + \kappa^2 s^2}} K_0\left(\frac{rs}{\sqrt{1+\nu s + \kappa^2 s^2}}\right)/K_1\left(\frac{Rs}{\sqrt{1+\nu s + \kappa^2 s^2}}\right).
\end{cases}
\end{equation}
\end{proposition}
 
\begin{proof}
Applying Laplace transform to the Boussinesq-Abbott system {\rm \hyperlink{BAr}{(BAr)}} reads 
\begin{equation}\label{eq:zetaqLapEQ}
\forall s \in \mathbb{C}_0, \,
\begin{cases}
s \widehat\zeta + \dr \widehat q = 0, \\
(1 - \kappa^2 \partial_r \dr) s \widehat q + \partial_r \widehat\zeta - \nu \partial_r \dr \widehat q = 0.
\end{cases}
\end{equation}
The boundary condition still is
\begin{equation}\label{eq:BCq}
\underline{\widehat q} = - \frac{R}{2} \widehat{\dot{\delta}},
\end{equation}
where $\delta$ satisfies the ODE
\begin{equation}
\tau_\kappa^2(0)\widehat{\ddot{\delta}} +\nu\widehat{\dot{\delta}} + \widehat{\delta} =\kappa^2 \underline{\widehat{\ddot{\zeta}}} + \nu \underline{\widehat{\dot{\zeta}}} + \underline{\widehat{\zeta}}.
\end{equation}
From \eqref{eq:zetaqLapEQ} we deduce that
\begin{equation}
\begin{cases}
\widehat\zeta =- \frac{1}{s} \dr \widehat q, \\
(s^2 - (\kappa^2 s^2 + \nu s + 1) \partial_r \dr) \widehat q = 0.
\end{cases}
\end{equation}
Hence $\widehat{q}$ satisfies
$$
\partial_r ^2 \widehat q + \frac{1}{r} \partial_r \widehat q - \left (\frac{1}{r^2} + \frac{s^2}{1+\nu s +\kappa^2 s^2}\right ) \widehat{q} = 0.
$$
The solutions of this equation are exactly the modified Bessel functions.
Therefore one can verify that the explicite solution of \eqref{eq:zetaqLapEQ} is
\begin{equation}
\forall s \in \mathbb{D}_{\kappa, \nu}, \,
\begin{cases}
\widehat{q}(r, s) = -\frac{R \widehat{\dot{\delta}}(s)}{2} K_1\left(\frac{rs}{\sqrt{1+\nu s + \kappa^2 s^2}}\right)/K_1\left(\frac{Rs}{\sqrt{1+\nu s + \kappa^2 s^2}}\right),\\
\widehat{\zeta}(r, s) =- \frac{R}{2} \frac{\widehat{\dot{\delta}}(s)}{\sqrt{1+\nu s + \kappa^2 s^2}} K_0\left(\frac{rs}{\sqrt{1+\nu s + \kappa^2 s^2}}\right)/K_1\left(\frac{Rs}{\sqrt{1+\nu s + \kappa^2 s^2}}\right).
\end{cases}
\end{equation}
\end{proof}

Now we know the expression of $\widehat{\zeta}$ and by using the modified Newton's equation \eqref{AddedMass} we can prove Theorem \ref{thm:DeltaCummins}.
 
\begin{proof}[Proof of Theorem \ref{thm:DeltaCummins}]
The vertical motion satisfies the ODE given by the modified Newton's equation \eqref{AddedMass}
\begin{equation}\label{eq:delta_RTTE}
\tau_\kappa^2(0)\ddot{\delta} + \nu \dot{\delta} + \delta = \kappa^2 \underline{\ddot{\zeta}} + \nu \underline{\dot{\zeta}} + \underline{\zeta}.
\end{equation}
Thus with the explicit expression of the surface elevation in the Laplace domain given by Proposition \ref{prop:zeta_q_Lap}, one can express the right hand-side of \eqref{eq:delta_RTTE} as a non-local function of $\delta$ and its derivative. More precisely, for all complex numbers $s \in \mathbb{D}_{\kappa,\nu}$, the ODE \eqref{eq:delta_RTTE} reads in the Laplace domain
\begin{equation}\label{prop:5.3}
\mathcal{L} (\tau_\kappa^2(0)\ddot{\delta} + \nu \dot{\delta} + \delta )(s) = -\frac{R}{2}\sqrt{1+\nu s + \kappa^2 s^2} B(s)\widehat{\dot{\delta}}(s)
\quad \text{where} \quad
B(s) :=  \frac{K_0\left(\frac{Rs}{\sqrt{1+\nu s + \kappa^2 s^2}}\right)}{K_1\left(\frac{Rs}{\sqrt{1+\nu s + \kappa^2 s^2}}\right)}.
\end{equation}
Moreover the right hand term of \eqref{prop:5.3} is the symbol of an operator of order at least one. Indeed for all $s \in \mathbb{D}_{\kappa,\nu}$, one has
\begin{equation}\label{eq:sqrt_sep}
\sqrt{1+\nu s + \kappa^2 s^2} = \kappa s + \sqrt{\nu}\frac{\sqrt{s}}{\sqrt{1 + \frac{\kappa^2}{\nu} s} + \frac{\kappa}{\sqrt{\nu}} \sqrt{s}} + \frac{1}{\sqrt{1+\nu s + \kappa^2 s^2} + \sqrt{\nu s + \kappa^2 s^2}}.
\end{equation}
By injecting the identity \eqref{eq:sqrt_sep} into equation \eqref{prop:5.3} we deduce that
$$
\mathcal{L} (\tau_\kappa^2(0)\ddot{\delta} + \nu \dot{\delta} + \delta )(s) = 
- \widehat{k_B} \left( 
\kappa s + \sqrt{\nu s} \widehat{k_0} + \widehat{k_1} \right)\widehat{\dot{\delta}},
$$
where $\mathcal{L}$ is the the Laplace transform. This equation describes Cummin's type equation \eqref{eq:CumminsDelta} in the Laplace domain.
\end{proof}
The end of this section is devoted to show properties on $\delta$, $\dot{\delta}$ and $\ddot{\delta}$ in the different regimes. In particular we will show Theorems \ref{thm:pt5.1.0}, \ref{thm:pt5.1.1} and \ref{thm:pt5.1.2}.
 
\subsection{Decay of the vertical motion}\label{sec:5.2}
Firstly we will use the explicit expression \eqref{eq:transferH} of $\delta$ in the Laplace domain and the Paley-Wiener Theorem \ref{thm:Bergman} for Bergman spaces to show some decay properties. First of all, one needs to describe more precisely the domain $\mathbb{D}_{\kappa, \nu}$ where the transfer function \eqref{eq:transferH} is well-defined.
 
\begin{proposition}\label{prop:C0DKV}
Let $\kappa, \nu \geq 0$. The set $\mathbb{D}_{\kappa, \nu}$ is the complex plane without the branch cuts $\mathbb{R}^{\ast}_{-}$ and $\Gamma_{\kappa,\nu}$ where $\Gamma_{\kappa,\nu}$ is given by the table
\begin{center}
\begin{tabular}{|c|c|c|c|c|}
  \hline
  {\rm Parameters} & $\kappa$ = 0 & $\nu = 0$ & $\nu \geq 2 \kappa$ & $\nu < 2\kappa$ \\ \hline
  $s_{\rm b}$ & $-\frac{1}{\nu} $ & $\pm i\kappa^{-1}$ & $ r_\pm := \frac{-\nu \pm \sqrt{\nu^2 - 4\kappa^2}}{2\kappa^2} $ & $ \frac{-\nu \pm i \sqrt{4\kappa^2 - \nu^2}}{2\kappa^2} $ \\ \hline 
  $\Gamma_{\kappa,\nu}$ & $\left \{ \eta, \eta \leq -\frac{1}{\nu}\right \}$ & $\left \{ i\omega, |\omega| > \kappa^{-1} \right \}$ & $ \left \{ \eta \in [r_-; r_+] \right\}$ & $\big\{ -\frac{\nu}{2\kappa^2} + i\omega,$ \\ & & & $ \cup \left\{-\frac{\nu}{2\kappa^2} + i\omega, \omega \in \mathbb{R} \right\} $ & $ |\omega| > \frac{1}{2\kappa^2} \sqrt{4\kappa^2 - \nu^2}\big\}$ \\ \hline 
\end{tabular}
\captionof{table}{Branching points $s_{\rm b}$ and $\Gamma_{\nu, \kappa}$ depending on $\nu$ and $\kappa$.}
\label{tab:BPBC}
\end{center}
and the branching points $s_{\rm b}$ solve $1+ \nu s_{\rm b} + \kappa^2 s_{\rm b}^2 =0$.
In particular, one has $\mathbb{C}_0 \subset \mathbb{D}_{\kappa, \nu}$ but $\mathbb{C}_{-\nu_0} \nsubseteq \mathbb{D}_{\kappa, \nu}$ for any $\nu_0 > 0$.
\end{proposition}
 
\begin{figure}[h!]	
\includegraphics[scale=0.5]{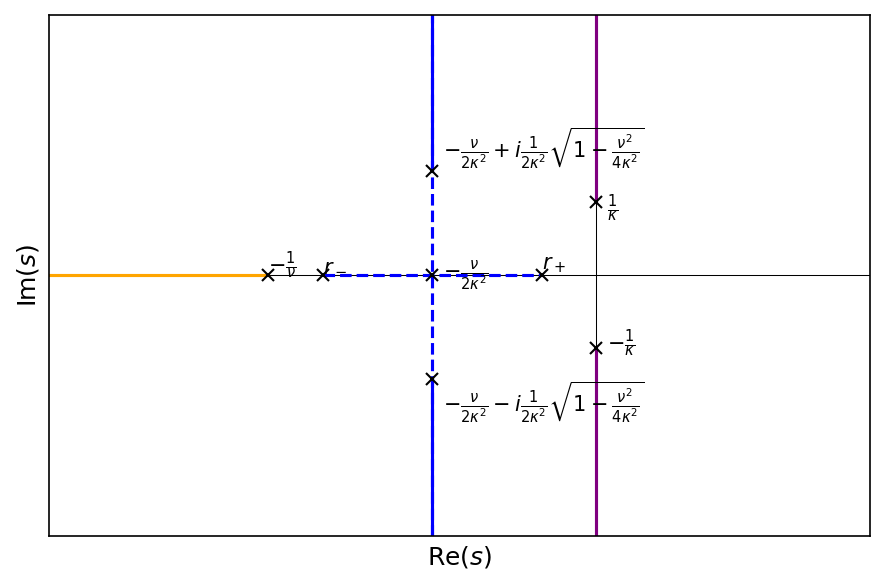}
\caption{Branching points of $\sqrt{1+\nu s + \kappa^2 s^2}$ depending on $\kappa$ and $\nu$.}
\label{fig:branch_cuts}
\end{figure}
Figure \ref{fig:branch_cuts} represents the branching points of $\sqrt{1+\nu s + \kappa^2 s^2}$ depending on $\kappa$ and $\nu$. The orange line is the case $\kappa = 0$, the purple ones are for $\nu = 0$ and the blue ones are for $\kappa \neq 0$, $\nu \neq 0$. The dotted blue lines are the case $\nu \geq 2\kappa$.
The proof of the proposition is given in Appendix \ref{app:Proof5.5}. In particular it implies that the Laplace transform of the vertical motion $\widehat{\delta}$ cannot be holomorphic on $\mathbb{C}_{-\nu_0}$ for any $\nu_0 > 0$, and thus cannot lie in any $\mathcal{B}_\alpha(\mathbb{C}_{-\nu_0})$ for $\alpha \geq -1$. The Paley-Wiener Theorem on Bergman spaces \ref{thm:Bergman} implies that $\delta \notin L^2(t^{-(\alpha+1)} e^{\nu_0 t}dt, \mathbb{R}^+)$. Such a situation is a consequence of the presence of the modified Bessel function in the expression of the transfer function. However in 1D, $B = 1$ and thus $\mathbb{D}_{\kappa, \nu}^{1D} := \{  s \in \mathbb{C} \, | 1+\nu s + \kappa^2 s^2 \notin \mathbb{R}^- \} = \mathbb{C} / \Gamma_{\kappa,\nu}$. Unlike the 2D case, in 1D with $\nu > 0$ there exists a $\nu_0 > 0$ such that
$$
  \mathbb{C}_{-\nu_0} := \left\{ s \in \mathbb{C} \, | \,  \Re(s) > -\nu_0 \right\} \subset \mathbb{D}_{\kappa, \nu}^{1D}.
$$
Therefore, in 1D, one can show with the same proof that one has $\hat{\delta}\in \mathcal{B}_{-1}(\mathbb{C}_{-\nu_0})$. Then, by Theorem \ref{thm:Bergman}, $\delta \in L^2(e^{\nu_0t} dt)$ and thus $\delta$ is exponentially decreasing. 
 
Secondly, one has to show that the transfer function lies in some Bergman spaces. To do so, one needs to show that the transfer function is holomorphic on $\mathbb{C}_0$. The first step is to obtain some properties on the modified Bessel functions and on the coefficient $B$.
 
\subsubsection{Preliminaries on $K$ and $B$}
 
In the following sections we need the behavior of the modified Bessel functions $K_0$ and $K_1$. They are given by the following lemma, according to \cite{A&S}.
\begin{lemma}\label{lem:BehavK}
The modified Bessel functions $K_0$ and $K_1$ satisfy the following behaviors.
\begin{itemize}
\item Near $0$, one has
$$
K_0(z) \underset{|z|\to 0}{=} - \ln(\frac{z}{2})-\gamma + o(z)
\quad \text{and} \quad
K_1(z) \underset{|z|\to 0}{=} \frac{1}{z} + o(z), \quad z \notin \mathbb{R}^-,
$$
where $\gamma$ is the Euler constant.
\item When $|z| \to \infty$ one has for $\alpha = 0, 1$
\begin{equation*}\label{eq:equivK_alpha}
\forall z \notin \mathbb{R}^- ,~ K_\alpha(z) \underset{|z|\to +\infty}{=} \sqrt{\frac{\pi}{2z}} e^{-z}\left (1 + \frac{4\alpha^2 - 1}{8z} + \frac{(4\alpha^2 - 1)(4\alpha^2 - 9)}{2(8z)^2} + o(\frac{1}{|z|^2})\right ).
\end{equation*}
\end{itemize}
\end{lemma} 
In accordance with the definition of the complex logarithm with a non-negative real part we use the principal argument
$$ 
\forall z \in \mathbb{C}\setminus \mathbb{R}^-, ~\ln(z) = \ln(|z|) + i {\rm Arg}(z),
$$
where ${\rm Arg}(z) \in (-\pi, \pi )$.
 
From this lemma one obtains the asymptotic analysis of $B$. In particular it allows to show the following proposition.
\begin{proposition}\label{prop:BehavB}
The function $B$ is holomorphic in $\mathbb{C}_0$ and satisfies $|B(s)|  \leq 1$. Moreover for $\nu =0$ we extend continuously $B$ on the imaginary axis $i\mathbb{R}$ with
$$
\left\{\begin{matrix}
 B(i\omega) = \frac{K_0(\frac{R|\omega|}{\sqrt{\kappa^2\omega^2-1}})}{K_1(\frac{R|\omega|}{\sqrt{\kappa^2\omega^2-1}})} \quad \text{ if } |\omega| > \frac{1}{\kappa}, \\
 B(0) = 0.
\end{matrix}\right.
$$
\end{proposition}
\begin{proof}
By definition of the modified Bessel functions $K_0$ and $K_1$ given by \eqref{eq:defK}, $B$ is holomorphic on $\mathbb{C}_0$. Furthermore for any $t \geq 0$, $\cosh(t) \geq 1$. Therefore one has
$$ 
| K_0(z)| = |\int_0^\infty e^{-\cosh(t)z}dt| < |\int_0^\infty e^{-\cosh(t)z}\cosh(t)dt| \leq | K_1(z)|.
$$
It results that $|B(s)| \leq 1$ for all $s \in \mathbb{C}_0$.

For any complex $s \in \mathbb{C}_0$, one denotes $p(s) = \frac{s}{\sqrt{1+\nu s + \kappa^2 s^2}}$ such that
$$
B(s) = \frac{K_0(p(s))}{K_1(p(s))}.
$$
Firstly for any $s \in \mathbb{C}_0$ one has
$$
p(s) \underset{|s|\to 0}{=} s + o(s).
$$
Therefore thanks to Lemma \ref{lem:BehavK} we deduce that
\begin{equation}\label{eq:equivB0}
B(s) \underset{|s|\to 0}{=} -s\ln(\frac{|s|}{2})-\gamma s - i s {\rm Arg}(s) + o(s^2),
\end{equation}
and
$$
\underset{|s|\to 0}{\lim} B(s) = 0.
$$
Thus $B$ can be extended by continuity in $0$ by $B(0) = 0$.

Secondly if $\nu = 0$ then we have a look at the situation $\eta = 0$. 
It is enough to do this case for the limit $|s|\to \pm i\kappa^{-1}$ because $B$ is holomorphic on the complex plan $\mathbb{C}_{0}$ and $\pm i\kappa^{-1}$ are isolated singularities.\\
\indent In a first place we will only use the asymptotic expansion \eqref{eq:equivK_alpha} from Lemma \ref{lem:BehavK} at order 0
\begin{equation}
\forall z \notin \mathbb{R}^- ,~ K_\alpha(z) \underset{|z|\to +\infty}{=} \sqrt{\frac{\pi}{2z}} e^{-z}\left (1 + o(1)\right ).
\end{equation}
\indent However when $\nu = 0$ and $\Re(s) = 0$, if $\omega < \frac{1}{\kappa}$ then $p(s) = i\frac{\omega}{\sqrt{1-\kappa^2\omega^2}}$, with $\omega = \Im(s)$. 
In particular $\frac{\omega}{\sqrt{1-\kappa^2\omega^2}} \underset{\omega \to \frac{1}{\kappa}^-}{\longrightarrow} \infty$ and $p(s) \in i\mathbb{R}$. 
Therefore for $\alpha = 0, 1$ 
$$
K_\alpha(p(i\omega)) \underset{\omega\to \frac{1}{\kappa}^-}{=} \sqrt{\frac{\pi}{2p(i\omega)}} e^{-p(i\omega)}\left (1 + o(1)\right ),
$$
and
$$
\underset{\omega \to \frac{1}{\kappa}^-}{\lim} B(i\omega) = 1.
$$
For $\omega > \frac{1}{\kappa}$, $p(i\omega) = \frac{\omega}{\sqrt{\kappa^2\omega^2-1}}$, the asymptotic expansion leads to
$$
\underset{\omega \to \frac{1}{\kappa}^+}{\lim} B(i\omega) = 1.
$$
For the same reasons we have
$$
\underset{\omega \to -\frac{1}{\kappa}^+}{\lim} B(i\omega) = \underset{\omega \to -\frac{1}{\kappa}^-}{\lim} B(i\omega) = 1.
$$
Therefore $B$ can be extended continuously in $\pm \frac{1}{\kappa}$ by $B(\pm \frac{1}{\kappa}) = 1$, and on $i\mathbb{R}$.
\end{proof}
Proposition \ref{prop:BehavB} combinated with Assumption \ref{conjecture-denom} leads to the holomorphy of $\widehat{H}$ on the half-plane $\mathbb{C}_0$. This property is necessary for the use of Bergman spaces. It is given by the following corollary.
\begin{cor}[Holomorphy of $\widehat{H}$]\label{hol-H}
If Assumption \eqref{conjecture-denom} is satisfied, then $\widehat{H}$ is holomorphic and bounded on the complex half-plan $\mathbb{C}_0$ and it modulus decreases faster than $|s|^{-1}$ for large $|s|$, in the sense that the decrease can be better depending on the regime between $\kappa$ and $\nu$.
\end{cor}
Assumption \eqref{conjecture-denom} is motivated by several points. Firstly one can compute the heat maps of the real part and the imaginary part of the denominator $\mathcal{P}$, and see the points where they both vanish (see Figure \ref{fig:nu_crit}).
\begin{center}
\includegraphics[width=1\textwidth]{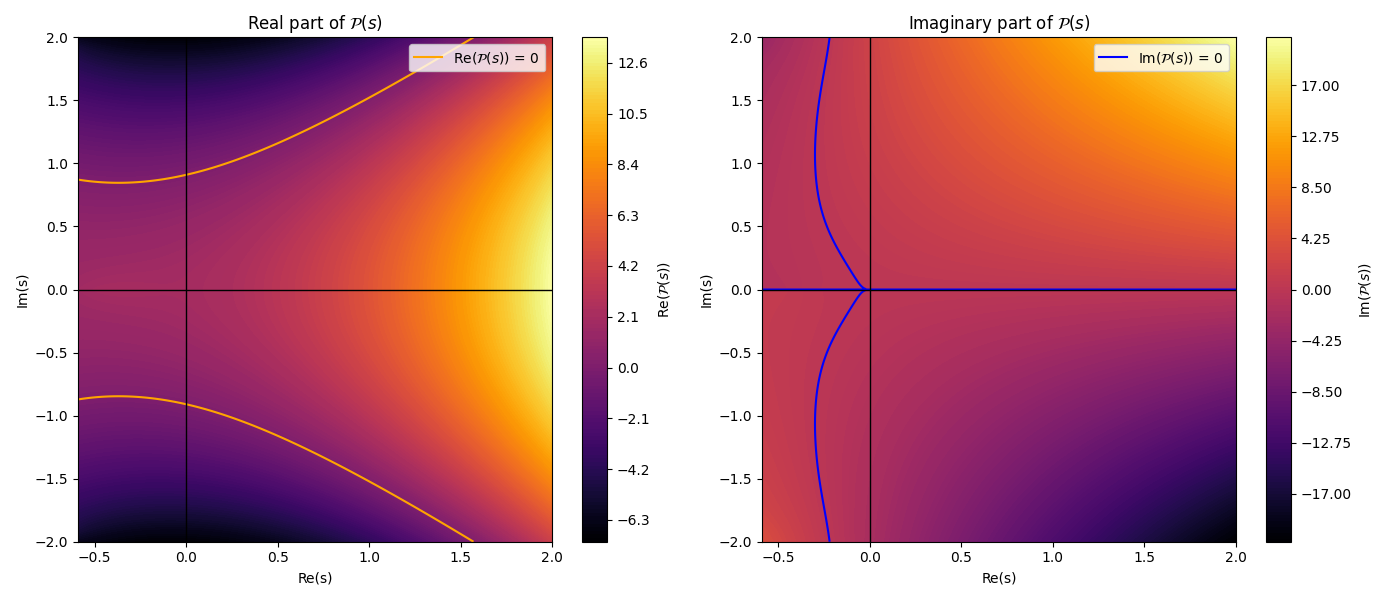}
\captionof{figure}{Real part and imaginary part of $\mathcal{P}$ and points where they both cancel, with $\nu = 0.3$, $\kappa = 0.5$.}
\label{fig:nu_crit}
\end{center}
We intuit from Figure \ref{fig:nu_crit} that $\Im(\mathcal{P}(s)) = 0$ implies that $s \in \mathbb{R}$.
However one has the following proposition.
\begin{proposition}\label{prop:0ofP}
Let $s = \eta \in \mathbb{R}$ such that $\eta > 0$.
Then $\mathcal{P}(s) \neq 0.$
\end{proposition}
This proposition is due to the positivity of the modified Bessel functions of second kind on $\mathbb{R}^\ast_+$.
Moreover assuming that Assumption \eqref{conjecture-denom} is true we show in following parts with the Paley-Wiener Theorem that $\delta$ is in $H^{2}(\mathbb{R}^+)$. 
By using Theorem \ref{thm:Bergman}, in Subsection \ref{sec:4.3} we show that the strongest decay allowed on $\delta$ is $\mathcal{O}(t^{-(1/2^+)})$ in the case $\nu = 0$. 
This result is better than the one found in \cite{BeckLannes} with $\nu = 0$, but the method is also adapted for the 1D case (see Section \ref{sec:1D}).  
However in \cite{BeckLannes}, $B = 1$ and $\mathcal{P}$ is an algebraic holomorphic function. 
In this case Assumption \eqref{conjecture-denom} is no longer a conjecture and can be shown. 
Furthermore Theorem \ref{thm:Bergman} allows us to show that in 1D with viscosity,  $\delta$ is exponentially decreasing (see Section \ref{sec:1D}).
Those comparisons to the 1D case are other points making us believe that the conjecture is true. 
 
We assume for all following subsections that Assumption \eqref{conjecture-denom} is satisfied.
 
\subsubsection{Vertical motion in $H^2$}\label{sec:5.2.1}
 
Let $\kappa,\nu \geq 0$.
When $\nu=0$ the denominator $\mathcal{P}$ of the transfer function is singular on the branching cuts $\{ i \omega \, | \, | \omega | > \kappa^{-1} \}$ and holomorphic on $(-\kappa^{-1},\kappa^{-1})$. When $\nu >0$, there is  no singularity on the imaginary axis. Therefore in all cases, $\mathcal{P}$ can be extended continuously on $i\mathbb{R}$ by
\begin{equation}\label{extension of P}
\mathcal{P}^\ast(i\omega) = \left\{\begin{matrix}
i\omega\left(2i\tau_\kappa^2(0)\omega + i {\rm sgn}(\omega) R\sqrt{\kappa^2 \omega^2 - 1} B(i\omega)\right) + 2, & \text{for} \, \nu=0, \, \kappa \neq 0, \, |\omega| > \frac{1}{\kappa}, \\
i\omega\left(2i\tau_\kappa^2(0)\omega + R\sqrt{1 + i \nu \omega - \kappa^2 \omega^2} B(i\omega)\right) + 2, & \text{in other cases},
\end{matrix}\right.
\end{equation} 
where in the case $\nu = 0$, $\kappa >0$ and $|\omega| > \kappa^{-1}$, $B$ was extended in Proposition \ref{prop:BehavB}.
Furthermore $\mathcal{P}^*$ is bounded from bellow by $\mathcal{P}_{\min}$ on $\mathbb{R}$ as a consequence of Assumption \eqref{conjecture-denom}.
Thus the vertical motion belongs to $L^2$ with the following proposition.
 
\begin{proposition}\label{prop:H} The transfert function $\widehat{H}$ lies in the Hardy space and satisfies
\begin{equation}
 \| \widehat{H} \|_{\mathcal{B}_{-1}} \lesssim \kappa_\ast^{-1/2}  (1+ R^2\sqrt{1+ \nu \kappa_\ast^{-1}} + \kappa_\ast^{-1} ) \mathcal{P}_{\min}^{-1} + \sqrt{\kappa_\ast^{-1}},
\end{equation}
where
$$
\kappa_\ast= \begin{cases} \kappa & \text{if $\kappa >0$}, \\ 1 & \text{if $\kappa =0$}. \end{cases}
$$ 
\end{proposition}
\begin{proof}
Thanks to Proposition \ref{prop:BehavB}, extension \eqref{extension of P} and Assumption \eqref{conjecture-denom}, $\widehat{H}$ is bounded and can be extended by continuity over the imaginary axis by
\begin{equation}\label{eq:H*}
\widehat{H}^\ast(i \omega) = \left\{\begin{matrix}
\frac{2i\tau_\kappa^2(0)\omega + i {\rm sgn}(\omega) R\sqrt{\kappa^2 \omega^2 - 1} B(i\omega)}{\mathcal{P}^\ast(i\omega)}, & \text{for $\nu=0, \kappa \neq 0$ and } |\omega| > \frac{1}{\kappa}, \\
\frac{2i\tau_\kappa^2(0)\omega + R\sqrt{1 + i \nu \omega - \kappa^2 \omega^2} B(i\omega)}{\mathcal{P}^\ast(i\omega)}, & \text{in other cases}.
\end{matrix}\right.
\end{equation}
By definition of the Hardy space norm, one has
$$ 
\|\widehat{H}\|^2_{ \mathcal{B}_{-1}} = 
\int_{- \infty}^{-\kappa_\ast^{-1}} |\widehat{H}^\ast(i\omega)|^2d\omega 
+ \int_{-\kappa_\ast^{-1}}^{\kappa_\ast^{-1}} |\widehat{H}^\ast(i\omega)|^2d\omega
+ \int_{\kappa_\ast^{-1}}^{+\infty } |\widehat{H}^\ast(i\omega)|^2d\omega.
$$
Thanks to the expression of $\widehat{H}^\ast$ given by \eqref{eq:H*}, one has
$$
\int_{-\kappa_\ast^{-1}}^{\kappa_\ast^{-1}} |\widehat{H}^\ast(i\omega)|^2d\omega \lesssim \kappa_\ast^{-1}  (1+R \sqrt{1+ \nu \kappa_\ast^{-1}})^2 \mathcal{P}_{\min}^{-2},
$$
Furthermore for all $|\omega| > \kappa_\ast^{-1}$, the transfer function is bounded by
$$
|\widehat{H}^\ast(i\omega)| \leq \frac{1}{|\omega|} \left (\kappa_*^{-1} + 2 \mathcal{P}_{\min}^{-1} \right ).
$$
Thus the integral satisfies
$$
\int_{|\omega| > \kappa_\ast^{-1}} |\widehat{H}^\ast(i\omega)|^2d\omega \leq 2  \kappa_\ast^{-1}  \left (1 + 2 \mathcal{P}_{\min}^{-1} \right )^2.
$$
It shows that
$$ 
 \| \widehat{H} \|_{\mathcal{B}_{-1}}^2 \lesssim \kappa_\ast^{-1}  (1+R \sqrt{1+ \nu \kappa_\ast^{-1}})^2 \mathcal{P}_{\min}^{-2} + \kappa_\ast^{-1}  \left (1 +  \mathcal{P}_{\min}^{-1} \right )^2.
$$ 
\end{proof}
 
We introduce $ \widehat I$ and $\widehat J$ defined by
\begin{equation}\label{def-widehatIJ}
\forall s \in \mathbb{C}_0, \, \widehat{\dot{\delta}}(s) = \widehat I(s) \delta_0 :=\frac{\delta_0}{\mathcal{P}}
, \quad \text{and} \quad \widehat{\ddot{\delta}}(s) = \widehat J(s) \delta_0 := \frac{s  \delta_0}{\mathcal{P}}.
\end{equation}
 
Thus we do an analogue work on $ \widehat{I}$ and $ \widehat{J}$ to show that $ \widehat{\dot{\delta}}, \widehat{\ddot{\delta}} \in \mathcal{B}_{-1}$. Furthermore, $\mathcal{P}^\ast$ satisfies for $\omega > \frac{1}{\kappa}$
$$
\frac{1}{|\mathcal{P}^\ast(i\omega)|} \leq \frac{1}{\omega^2 \tau_\kappa^2(0) - 1} \leq \frac{1}{\omega^2 \tau_{\rm buoy}^2}.
$$ 
Indeed, thanks to the expression of $\tau_\kappa^2(0)$ given in Proposition \ref{prop:AddedMass}, for $|\omega| > \kappa_*^{-1}$, $\omega^2 \tau_\kappa^2(0) > \omega^2 \tau_{\rm buoy}^2 + 1$.
Therefore one has
$$
\int_{\kappa_\ast^{-1}}^\infty \frac{1}{|\mathcal{P}(i\omega)|^2} \leq \frac{2\kappa_\ast^{3}}{3(\tau_{\rm buoy}^2)^2}
\quad
\text{and}
\quad
\int_{\kappa_\ast^{-1}}^\infty \frac{\omega^2}{|\mathcal{P}(i\omega)|^2} \leq \frac{2\kappa_\ast}{(\tau_{\rm buoy}^2)^2}.
$$ 
We deduce that
\begin{equation}\label{eq:EstimHardy}
\| \widehat{I} \|_{\mathcal{B}_{-1}}^2 \leq 2\kappa_\ast^{-1} \mathcal{P}_{\min}^{-2} + \frac{2\kappa_\ast^3}{3(\tau_{\rm buoy}^2)^2}
\quad
\text{and}
\quad
\| \widehat{J} \|_{\mathcal{B}_{-1}}^2 \leq 2\kappa_\ast^{-1} \mathcal{P}_{\min}^{-2} + \frac{2\kappa_\ast}{(\tau_{\rm buoy}^2)^2}.
\end{equation}
 
Consequently, thanks to the Paley-Wiener Theorem on Bergman spaces \ref{thm:Bergman}, Proposition \ref{prop:H} and estimates \ref{eq:EstimHardy} conclude the proof of Theorem \ref{thm:pt5.1.0}.
 
\subsection{Proof of Theorems \ref{thm:pt5.1.1} and \ref{thm:pt5.1.2}}\label{sec:4.3}
 
To show Theorems \ref{thm:pt5.1.1} and \ref{thm:pt5.1.2}, we will use the derivatives of the transfer functions $\widehat{H}$, $\widehat{I}$ and $\widehat{J}$, and the Paley-Wiener Theorem for the Bergman spaces \ref{thm:Bergman}. In order to do so, some preliminary results on the derivatives of $B$ and $\mathcal{P}$ are needed, in particular on their asymptotic behaviors.
 
\subsubsection{Some preliminary results}
Firstly one has the following lemma on the derivates of $B$ (see \eqref{eq:transferH}).
 
\begin{lemma}\label{lem:DeriveesB}
For $\nu, \kappa \geq 0$, one has
\begin{equation}
\forall s \in \mathbb{C}_0, \, B'(s) = \frac{1+\frac{\nu}{2}s}{(1+\nu s + \kappa^2 s^2)^{3/2}}\left ( B(s)^2 - 1 + \frac{\sqrt{1+\nu s + \kappa^2 s^2}}{s}B(s)\right ).
\end{equation} 
$B'$ is holomorphic on $\mathbb{C}_0$, bounded on $\mathbb{C}_0 \setminus B(0, \rho\min(1, \kappa^{-1}))$ for $0 < \rho \ll 1$, and we have the following behaviors:
\begin{itemize}
\item Near $0$ one has
\begin{equation}
B'(s) \underset{|s| \to 0}{\sim} \ln(|s|), 
\quad 
B''(s) \underset{|s| \to 0}{\sim} \frac{\ln(|s|)}{s},
\quad \text{and} \quad
B^{(3)}(s) \underset{|s| \to 0}{\sim} \frac{\ln(|s|)}{s^2}
\quad \text{for } s \in \mathbb{C}_0.
\end{equation}
\item When $|s| \to \infty$, if $\kappa = \nu = 0$, $B'$ decreases in $|s|^{-2}$ at infinity. If $\kappa =0$ and $ \nu > 0$, then $B'$ decreases in $|s|^{-3/2}$ at infinity. If $\kappa > 0$, $B'$ decreases in $|s|^{-2}$ at infinity if $\nu >0$ and in $|s|^{-3}$ is $\nu = 0$. Moreover $B''(s) \underset{|s| \to \infty}{=} \mathcal{O}(1)$ and $B^{(3)}(s) \underset{|s| \to \infty}{=} \mathcal{O}(1)$.
\end{itemize}
\end{lemma}
 
\begin{proof}
We remind that near $0$, according to equation \eqref{eq:equivB0}, 
$$
\forall s \in \mathbb{C}_0, \, B(s) \underset{|s|\to 0}{\sim} s\ln(|s|).
$$
One can compute the derivative of $B$ to obtain 
$$
\forall s \in \mathbb{C}_0, \, B'(s) = \frac{1+\frac{\nu}{2}s}{(1+\nu s + \kappa^2 s^2)^{3/2}}\left ( B(s)^2 - 1 + \frac{\sqrt{1+\nu s + \kappa^2 s^2}}{s}B(s)\right ).
$$
Therefore we have the first point of the lemma. Let us remind that $p(s) = \frac{s}{\sqrt{1+\nu s +\kappa^2 s^2}}$.
Let $\kappa = 0$ and $\nu \geq 0$.
When $|s| \to \infty $, $p(s) \to \infty$, and thanks to the asymptotic expansion \eqref{eq:equivK_alpha} one has 
\begin{align*}
B(s) & \underset{|s| \to \infty}{=} \frac{1- \frac{1}{8p(s)} + \frac{9}{2(8p(s))^2}}{1+ \frac{3}{8p(s)} - \frac{15}{2(8p(s))^2}} + o\left ( \frac{1}{p(s)^2}\right )
\\& \underset{|s| \to \infty}{=} \left (1- \frac{1}{8p(s)} + \frac{9}{2(8p(s))^2}\right )\left (1- \frac{3}{8p(s)} + \frac{15}{2(8p(s))^2} + o\left ( \frac{1}{p(s)^2}\right )\right )
\\& \underset{|s| \to \infty}{=} 1 - \frac{1}{2p(s)} + \frac{15}{64p(s)^2} + o\left (\frac{1}{p(s)^2}\right ).
\end{align*}
We deduce that 
\begin{equation}
B(s)^2 \underset{|s| \to \infty}{=} 1 - \frac{1}{p(s)} + \frac{15}{32p(s)^2} + o\left (\frac{1}{p(s)^2}\right ),
\end{equation}
and it leads to
\begin{equation}
B(s)^2 - 1 + \frac{1}{p(s)}B(s) \underset{|s| \to \infty}{=} - \frac{1+\frac{\nu}{2}s}{32 s^2} + o\left ( \frac{1}{p(s)^2}\right ).
\end{equation}
If $\nu = 0$, then 
\begin{equation}
B(s)^2 - 1 + \frac{1}{p(s)}B(s) \underset{|s| \to \infty}{=} - \frac{1}{32 s^2} + o\left ( \frac{1}{s^2}\right ),
\end{equation}
and $B'$ decreases in $|s|^{-2}$ at infinity.
Else if $\nu > 0$, one has
\begin{equation}
B(s)^2 - 1 + \frac{1}{p(s)}B(s) \underset{|s| \to \infty}{=} - \frac{\nu}{64 s} + o\left ( \frac{1}{s}\right ).
\end{equation} 
Moreover $\frac{1+\frac{\nu}{2}s}{(1+\nu s)^{3/2}} \underset{|s| \to \infty}{=} \mathcal{O}(|s|^{-1/2})$, and we deduce that $B'$ decreases in $|s|^{-3/2}$ at infinity. \\
\indent Now if $\kappa > 0$, then $p(s) \underset{|s| \to \infty}{=} \mathcal{O}(1)$ and $B(s)^2 - 1 + \frac{1}{p(s)}B(s)$ is bounded when $|s|\to +\infty$. However $\frac{1+\frac{\nu}{2}s}{(1+\nu s + \kappa^2 s^2)^{3/2}}$ decreases in $|s|^{-2}$ if $\nu > 0$, and in $|s|^{-3}$ if $\nu = 0$. We deduce that $B'$ decreases in $|s|^{-2}$ or $|s|^{-3}$ at infinity. 
Furthermore thanks to the asymptotic expension \eqref{eq:equivB0}, one deduces that
\begin{equation}
B'(s) \underset{|s| \to 0}{\sim} \ln(|s|).
\end{equation}
The second derivative of $B$ reads
\begin{align*}
B''(s) = & \left (-\frac{\nu/2 + (\nu^2/4 + \kappa^2) s + \nu \kappa^2 s^2 / 2}{(1+\nu s + \kappa^2 s^2)^\frac{5}{2}} - \frac{\nu/2 + 2\kappa^2 s + \frac{\nu}{2} \kappa^2 s^2}{(1+\nu s + \kappa^2 s^2)^3} \right )\left (B(s)^2 - 1 + \frac{1}{p(s)}B(s)\right ) \\ & + \frac{1+\nu/2 s}{(1+\nu s + \kappa^2 s^2)^2}\left (2B'(s) B(s) - \frac{p'(s)}{p(s)}B(s) + \frac{1}{p(s)}B'(s)\right ).
\end{align*}
Thanks to the expression of $B'$ it can be put under the form 
\begin{align*}
B''(s) =  \frac{\tilde{B}(s)}{p(s)^2}B(s) + \breve{B}(s),
\end{align*}
where $\tilde{B}(s)$ and $\breve{B}(s)$ admit finite limits in $0$ and in $|s|\to \infty$. Furthermore $\underset{|s|\to 0}{\lim}\tilde{B}(s) $ is non-null.
Therefore, near $0$, the fastest coefficient to diverge in $B''$ is $\frac{1}{p(s)^2}B(s)$ and $B''(s) \underset{|s| \to 0}{\sim} \frac{\ln(|s|)}{s}$. 
Moreover $B''(s) \underset{|s| \to \infty}{=} \mathcal{O}(1)$. 
Similarly if we differentiate this expression one more time, $B^{(3)}$ can be put under a similar form 
\begin{equation}
B^{(3)}(s) =  \frac{\tilde{\mathscr{B}}(s)}{p(s)^3}B(s) + \breve{\mathscr{B}}(s),
\end{equation}
where $\underset{|s|\to 0}{\lim}\tilde{\mathscr{B}}(s) $ is finite and non-null. Thus the strongest coefficient in $B^{(3)}$ in $0$ is $\frac{1}{p(s)^3}B(s)$ and $B^{(3)}(s) \underset{|s| \to 0}{\sim} \frac{\ln(|s|)}{s^2}$. Also we have that $B^{(3)}(s) \underset{|s| \to \infty}{=} \mathcal{O}(1)$. 
\end{proof}
The decrease of $B''$ and $B^{(3)}$ at infinity can be better depending on the regime (i-e in function of the relation between $\kappa$ and $\nu$), as for $B'$, but it is sufficient to have this behavior.
This lemma gives the behavior of $\mathcal{P}$ and its derivatives by the following corollary.
 
\begin{cor}\label{cor:derivP}
$\mathcal{P}$ and its derivatives satisfy the following behaviors
\begin{itemize}
\item Near $0$, one has
\begin{equation}
\forall s \in \mathbb{C}_0, \, \mathcal{P}'(s) \underset{|s| \to 0}{\sim} s\ln(|s|), 
\quad
\mathcal{P}''(s) \underset{|s| \to 0}{\sim} \ln(|s|),
\quad \text{and} \quad
\mathcal{P}^{(3)}(s) \underset{|s| \to 0}{\sim} \ln(|s|)/s.
\end{equation}
\item When $|s| \to \infty$, one has
\begin{equation}
\forall s \in \mathbb{C}_0, \, \mathcal{P}'(s)\underset{|s| \to \infty}{\sim} s,
\quad
\mathcal{P}''(s)\underset{|s| \to \infty}{=} \mathcal{O}(|s|^2),
\quad \text{and} \quad
\mathcal{P}^{(3)}(s)\underset{|s| \to \infty}{=} \mathcal{O}(|s|^2).
\end{equation}
\end{itemize}
Furthermore if $\nu = 0$ and $\kappa > 0$, $\mathcal{P}'$ can be written
$$
\forall s \in \mathbb{C}_0, \ \mathcal{P}' = \frac{1}{\sqrt{1+\kappa^2 s^2}} \tilde{\mathcal{P}} + \breve{\mathcal{P}},
$$
where $\tilde{\mathcal{P}}$, $\tilde{\mathcal{P}}$ are holomorphic on $\mathbb{C}_0$ and $\tilde{\mathcal{P}}(\pm i\kappa^{-1}) \neq 0$.
\end{cor}
\begin{proof}
We deduce this corollary from the expression of the derivatives of $\mathcal{P}$. Indeed one has
\begin{align*}
\forall s \in \mathbb{C}_0, \, \mathcal{P}'(s) = & 4 s \tau_\kappa^2(0) + 2\nu + R \sqrt{1+\nu s + \kappa^2 s^2} B(s) + R \frac{\frac{\nu}{2}s + \kappa^2 s^2}{\sqrt{1+\nu s + \kappa^2 s^2}}B(s) \\& + Rs \sqrt{1+\nu s + \kappa^2 s^2}B'(s),
\end{align*}
and
\begin{align*}
\forall s \in \mathbb{C}_0, \, \mathcal{P}''(s) = & 4 \tau_\kappa^2(0) + 2R \frac{\frac{\nu}{2} + \kappa^2 s}{\sqrt{1+\nu s + \kappa^2 s^2}}\left (B(s) + sB'(s)\right ) + 2R \sqrt{1+\nu s + \kappa^2 s^2}B'(s) \\&
+ Rs \sqrt{1+\nu s + \kappa^2 s^2}B''(s) + Rs\frac{\kappa^2 - \frac{\nu^2}{4}}{(1+\nu s +\kappa^2 s^2)^{3/2}}B(s).
\end{align*}
Therefore if $\nu = 0$, $\mathcal{P}'$ can be written under the form
$$
\mathcal{P}' = \frac{1}{\sqrt{1+\kappa^2 s^2}} \tilde{\mathcal{P}} + \breve{\mathcal{P}},
$$
with $\tilde{\mathcal{P}} = R (\nu s + \kappa^2 s^2)B(s)$, and  $\tilde{\mathcal{P}}(\pm i\kappa^{-1}) \neq 0$ thanks to Lemma \ref{prop:BehavB}.
Also by Lemma \ref{lem:DeriveesB} one has
\begin{equation}
\mathcal{P}'(s) \underset{|s| \to 0}{\sim} s\ln(|s|)
\quad \text{and} \quad 
\mathcal{P}''(s) \underset{|s| \to 0}{\sim} \ln(|s|).
\end{equation}
Similarly if we differentiate $\mathcal{P}$ one more time we obtain
\begin{align*}
\forall s \in \mathbb{C}_0, \, \mathcal{P}^{(3)}(s) = & R\sqrt{1+\nu s + \kappa^2 s^2}\left (3B''(s) + sB^{(3)}\right ) + 3 R \frac{\nu/2 + \kappa^2 s}{\sqrt{1+\nu s + \kappa^2 s^2}}\left (2B'(s) + sB''(s)\right ) \\
& + 3R \frac{\kappa^2 - \frac{\nu^2}{4}}{(1+\nu s +\kappa^2 s^2)^{3/2}}\left (B(s) + sB'(s) \right ) - 3 R\frac{(\kappa^2 - \frac{\nu^2}{4})(\nu/2 + \kappa^2 s)}{(1+\nu s + \kappa^2s^2)^{5/2}}B(s).
\end{align*}
Thus the fastest diverging coefficient near $0$ in $\mathcal{P}^{(3)}$ is $Rs \sqrt{1+\nu s + \kappa^2 s^2}B^{(3)}(s)$ and $\mathcal{P}^{(3)}(s) \underset{|s| \to 0}{\sim} \ln(|s|)/s$.
Lemma \ref{lem:DeriveesB} also gives behaviors when $|s|\to \infty$. 
\end{proof}
Now we can give characterizations of the derivatives of the transfer functions $\widehat{H}, \, \widehat{I}$ and $\widehat{J}$.
 
\subsubsection{Characterizations of the derivatives of $\widehat{H}$, $\widehat{I}$ and $\widehat{J}$}
 
\begin{lemma}[Characterization of the derivative of $\widehat{H}$]\label{prop:derivH}
For all $s \in \mathbb{C}_0$ the derivative of the transfer function can be written
\begin{equation}
\frac{d \widehat{H}}{ds} = \frac{1}{\sqrt{1+\nu s + \kappa^2s^2}} \tilde{H}(s) + \breve H(s),
\end{equation}
where $\tilde{H}$ is given by
\begin{equation}
\forall s \in \mathbb{C}_0, \, \tilde{H}(s) = \frac{R(\nu + 2\kappa^2s)B(s)}{\mathcal{P}(s)^2}.
\end{equation} 
Moreover $\tilde{H}$ is holomorphic, uniformly bounded on $\mathbb{C}_0$ and $\tilde{H}( \pm i \kappa) \neq 0$. $\breve{H}$ is bounded on $\mathbb{C}_0 \setminus B(0, \rho \min(1, \kappa^{-1}))$, for $0 < \rho \ll 1$ and satisfies on the singularity at 0 
$$
\forall s \in \mathbb{C}_0, \, \breve{H}(s) \underset{|s| \to 0}{\sim} \ln(|s|).
$$
$\tilde{H}$ and $\breve{H}$ satisfy the asymptotic behavior
$$
\forall s \in \mathbb{C}_0, \, \tilde{H}(s) \underset{|s| \to \infty}{\sim} \frac{1}{|s|^3}
\quad
\text{and}
\quad
\breve H(s) \underset{|s| \to \infty}{\sim} \frac{1}{|s|^2}.
$$
\end{lemma} 
\begin{proof}
Let $s \in \mathbb{C}_0$.
After calculus, one gets
\begin{align*}
\frac{d\widehat{H}}{ds}(s) = \frac{1}{\mathcal{P}(s)^2} &\left[ 4\tau_\kappa^2(0)\left ( 1 + (\nu +\tau_\kappa^2(0))s^2 - 3 \nu s\right ) - 4\nu^2 \right .\\
&\left . + 2R\sqrt{1+\nu s + \kappa^2 s^2}B'(s) - 4(\tau_\kappa^2(0) s + \nu)R\sqrt{1+\nu s + \kappa^2 s^2}B(s) \right.\\
&\left . - R^2(1+\nu s + \kappa^2 s^2)B(s)^2 + R \frac{\nu + 2\kappa^2 s}{\sqrt{1+\nu s +\kappa^2 s^2}} B(s) \right].
\end{align*}
Therefore, we set
$$
\tilde{H}(s) = R\frac{\nu + 2 \kappa^2 s }{\mathcal{P}(s)^2}B(s),
$$
and
\begin{align*}
\breve{H}(s) = \frac{1}{\mathcal{P}(s)^2} &\left[ 4\tau_\kappa^2(0)\left ( 1 + (\nu +\tau_\kappa^2(0))s^2 - 3 \nu s\right ) - 4\nu^2 + 2R\sqrt{1+\nu s + \kappa^2 s^2}B'(s)\right .\\
&\left . - 4(\tau_\kappa^2(0) s + \nu)R\sqrt{1+\nu s + \kappa^2 s^2}B(s) - R^2(1+\nu s + \kappa^2 s^2)B(s)^2 \right].
\end{align*}
Thanks to Lemma \ref{lem:DeriveesB} and Assumption \eqref{conjecture-denom}, $\tilde{H}$ is holomorphic on $\mathbb{C}_0$ and uniformaly bounded. Because of the presence of $B'$ in the expression of $\breve{H}$, $\breve{H}$ is bounded on $\mathbb{C}_0 \setminus B(0, \rho \min(1, \kappa^{-1}))$, for $0 < \rho \ll 1$, and satisfy the expected behavior. Also at $|s|\to \infty$ thanks to Lemma \ref{lem:DeriveesB} and Corollary \ref{cor:derivP} one has
$$
\tilde{H}(s) \underset{|s| \to \infty}{\sim} \frac{1}{|s|^3}
\quad
\text{and}
\quad
\breve H(s) \underset{|s| \to \infty}{\sim} \frac{1}{|s|^2} .
$$
\end{proof}
The same kind of characterizations stand for $\widehat{I}$ and $\widehat{J}$ and their derivatives.
\begin{lemma}[Characterization of the derivatives of $\widehat{I}$]\label{prop:derivI}
Let $\widehat{I}$ defined in \eqref{def-widehatIJ}.
For all $s \in \mathbb{C}_0$ the derivatives of the transfer function $\widehat{I}$ can be written
\begin{equation}
\frac{d \widehat{I}}{ds} = \frac{1}{\sqrt{1+\nu s + \kappa^2s^2}} \tilde{I}(s) + \breve I(s),
\end{equation}
and
\begin{equation}
\frac{d^2 \widehat{I}}{ds^2} = \frac{1}{(1+\nu s + \kappa^2s^2)^\frac{3}{2}} \tilde{\mathfrak I}(s) + \breve{ \mathfrak I}(s),
\end{equation}
where the different functions satisfy
\begin{enumerate}
\item $\breve{I}$ and $\breve{\mathfrak{I}} $ are holomorhic on $\mathbb{C}_0 \setminus \mathcal{B}(0, \rho\min(1, \kappa^{-1}))$ for $0 < \rho \ll 1$, and satisfy near $0$:
$$
\forall s \in \mathbb{C}_0, \, \breve{I}(s) \underset{|s| \to 0}{\sim} s\ln(|s|)
\quad \text{and} \quad
\breve{\mathfrak I}(s) \underset{|s| \to 0}{\sim} \ln(|s|).
$$
\item $\tilde{I}$, $\tilde{\mathfrak I}$, $\breve{I}$ are holomorphic and uniformly bounded on $\mathbb{C}_0$ and satisfy the asymptotic expansions
$$
\forall s \in \mathbb{C}_0, \, \tilde{I}(s) \underset{|s| \to 0}{\sim} s 
\quad \text{and} \quad
\tilde{\mathfrak I}(s) \underset{|s| \to 0}{=} \mathcal{O}(1).
$$
\item If $\kappa > 0$, $\tilde{I}(\pm i\kappa^{-1}) \neq 0$.
\end{enumerate}
Furthermore $\frac{d \widehat{I}}{ds}$ and $\frac{d \widehat{I}}{ds}$ satisfy the asymptotic expansions 
$$
\forall s \in \mathbb{C}_0, \, \frac{d \widehat{I}}{ds} \underset{|s| \to \infty}{\sim} |s|^{-3}
\quad \text{and} \quad
\frac{d \widehat{I}}{ds} \underset{|s| \to \infty}{\sim} |s|^{-2},
$$
\end{lemma} 
\begin{proof}
Let $s \in \mathbb{C}_0$.
Since $\widehat{I}$ is defined as $\widehat{I} = -\frac{1}{\mathcal{P}}$, its derivatives reads
$$
\forall s \in \mathbb{C}_0, \, \frac{d\widehat{I}}{ds} = \frac{\mathcal{P}'(s)}{\mathcal{P}(s)^2}
\quad \text{and} \quad 
\frac{d^2\widehat{I}}{ds^2} = \frac{\mathcal{P}''(s)}{\mathcal{P}(s)^2} - 2\frac{\mathcal{P}'(s)^2}{\mathcal{P}(s)^3}.
$$
Thanks to Corollary \ref{cor:derivP}, one obtains the form that we want as much as the behaviors in $0$ and when $|s| \to \infty$.
\end{proof}
\begin{lemma}[Characterization of the derivatives of $\widehat{J}$]\label{prop:derivJ}
Let $\widehat{J}$ defined in \eqref{def-widehatIJ}.
For all $s \in \mathbb{C}_0$ one can write the first derivative of $\widehat{J}$ as
\begin{equation}
\frac{d \widehat{J}}{ds} = \frac{1}{\sqrt{1+\nu s + \kappa^2s^2}} \tilde{J}(s) + \breve J(s),	
\end{equation}
where 
\begin{enumerate}
\item $\tilde{J}$ and $\breve{J}$ are bounded and holomorphic on $\mathbb{C}_0$,
\item for all $s \in \mathbb{C}_0$, $\tilde{J}(s) \underset{|s| \to \infty}{\sim} |s|^{-2}$, $\breve{J}(s) \underset{|s| \to \infty}{\sim} |s|^{-2}$,
\item if $\nu = 0 = \kappa$, $\tilde{J} = 0$, if $\kappa > 0$, $\tilde{J}(\pm i\kappa^{-1}) \neq 0$.
\end{enumerate}
Furthermore in the case $\nu > 0$ or $\kappa = 0$, $\frac{d^2\widehat{J}}{ds^2}$ is uniformly bounded and holomorphic on $\mathbb{C}_0$, and the third derivative of $\widehat{J}$ satisfies the asymptotic behaviors
\begin{equation}
\forall s \in \mathbb{C}_0, \, \frac{d^3\widehat{J}}{ds^3}(s) \underset{|s| \to 0}{\sim} \ln(|s|), \quad \frac{d^3\widehat{J}}{ds^3}(s) \underset{|s| \to \infty}{=} \mathcal{O}(|s|^2).
\end{equation}
\end{lemma} 
 
\begin{proof}
Thanks to the definition of $\widehat{J}$ one has in $\mathbb{C}_0$
$$
\frac{d\widehat{J}}{ds}(s)= - \frac{1}{\mathcal{P}(s)} + \frac{s\mathcal{P}'(s)}{\mathcal{P}(s)^2}.
$$
One more time thanks to Corollary \ref{cor:derivP} we deduce the point $(1)$. One has
$$
\frac{d^2\widehat{J}}{ds^2}(s) = 2\frac{\mathcal{P}'(s)}{\mathcal{P}(s)^2}+ \frac{s\mathcal{P}''(s)}{\mathcal{P}(s)^2}-2\frac{s\mathcal{P}'(s)^2}{\mathcal{P}(s)^3},
$$
and $\frac{d^2\widehat{J}}{ds^2}$ is holomorphic and bounded thanks to Corollary \ref{cor:derivP}. Finally 
$$	
\frac{d^3\widehat{J}}{ds^3}(s) = 3\frac{\mathcal{P}''(s)}{\mathcal{P}(s)^2}- 6 \frac{\mathcal{P}'(s)^2}{\mathcal{P}(s)^3}+ \frac{s\mathcal{P}'''(s)}{\mathcal{P}(s)^2}-6\frac{s\mathcal{P}''(s)\mathcal{P}'(s)}{\mathcal{P}(s)^3}+6\frac{s\mathcal{P}'(s)^3}{\mathcal{P}(s)^4},
$$
and with Corollary \ref{cor:derivP} we obtain the behavior of $\frac{d^3\widehat{J}}{ds^3}$ when $|s|\to 0$ and $|s|\to \infty$.
\end{proof}
One more time, the different decreases when $|s|\to\infty$ can be more precise depending on $\nu$ and $\kappa$. 
We can now separate two cases. In the nonviscous dispersive case, i-e $ \nu = 0$ and $\kappa >0$, the square root admits two singularities in $\pm i \kappa^{-1}$ and the derivatives cannot lie in any Bergman space. 
However when we add viscosity, or if there is no dispersion, then the square root is welldefined on $\mathbb{C}_0$. 
In this case the derivatives can lie in some Bergman spaces, which translates in the temporal domain by Theorems \ref{thm:pt5.1.1} and \ref{thm:pt5.1.2}.
 
\begin{proposition}\label{prop:dBergmann_nv_d}
For all $\nu = 0$ and $\kappa > 0$, one has that $\frac{d\widehat{H}}{ds}, \, \frac{d \widehat{I}}{ds}, \, \frac{d \widehat{J}}{ds} \, \notin \, \mathcal{B}_{\alpha}$ for $\alpha \in [-1, 1)$.
\end{proposition}
\begin{proof}
Let $s = \eta + i \omega \in \mathbb{C}_0$, $\nu = 0$ and $\kappa > 0$. Lemmata \ref{prop:derivH}, \ref{prop:derivI} and \ref{prop:derivJ} give that 
\begin{equation}
\eta^\alpha|\frac{d \widehat{\mathcal{G}}}{ds}|^2 = \frac{\eta^\alpha}{1+\kappa^2s^2} |\tilde{\mathcal{G}}(s)|^2 + \eta^\alpha |\breve{\mathcal{G}}(s)|^2 + 2 \frac{\eta^{\alpha}}{\sqrt{1+\kappa^2s^2}} \tilde{\mathcal{G}}(s)\breve{\mathcal{G}}(s),
\end{equation}
for $\mathcal{G} \in \{ H, I, J \}$ and with the convention $\eta^{\alpha} = 1$ if $\alpha= -1$.
This expression is divergent in $s = \pm i \kappa^{-1}$ for any $\alpha < 1$ since $\tilde{\mathcal{G}}( \pm i \kappa^{-1}) \neq 0$.
\end{proof}
We could have expected this result from the 1D situation. Indeed in \cite{BeckLannes}, the first author shows that $\frac{d \widehat{\delta}}{ds}$ does not belong to $\mathcal{B}_{-1}$. 
\begin{proposition}\label{prop:dBergmann_v/nd}
If $\nu > 0$ or $\kappa = 0$, one has that $\frac{d\widehat{H}}{ds}, \, \frac{d^2 \widehat{I}}{ds^2}, \, \frac{d^3 \widehat{J}}{ds^3} \in \mathcal{B}_{\alpha}$ for all $\alpha \in (-1, 1)$, but $\frac{d\widehat{H}}{ds}, \, \frac{d^2 \widehat{I}}{ds^2}, \, \frac{d^3 \widehat{J}}{ds^3} \notin \mathcal{B}_{-1}$. 
\end{proposition}
\begin{proof}
Let $s = \eta + i \omega \in \mathbb{C}_0$, $\nu > 0$ or $\kappa = 0$. Thanks to Lemma \ref{prop:derivH} one has that for any $\alpha \geq -1$, 
\begin{equation}
\eta^\alpha|\frac{d \widehat{H}}{ds}|^2 = \frac{\eta^\alpha}{|1+\nu s +\kappa^2s^2|} |\tilde{H}(s)|^2 + \eta^\alpha |\breve{H}(s)|^2 + 2 \frac{\eta^\alpha}{|\sqrt{1+\nu s +\kappa^2s^2}|} |\tilde{H}(s)\breve{H}(s)|,
\end{equation}
with the convention $\eta^{\alpha} = 1$ if $\alpha= -1$. As $1+\nu s +\kappa^2s^2$ has no zero for $s \in \mathbb{C}_0$, thanks to Proposition \ref{prop:derivH}, for all $\alpha \in (-1, 1)$,
$$
\int_{\mathbb{C}_0} \frac{\Re(s)^\alpha}{|1+\nu s +\kappa^2s^2|} |\tilde{H}(s)|^2 ds, 
\int_{\mathbb{C}_0} \Re(s)^\alpha |\breve{H}(s)|^2 ds,
\int_{\mathbb{C}_0} \frac{\Re(s)^\alpha}{|\sqrt{1+\nu s +\kappa^2s^2}|} |\tilde{H}(s)\breve{H}(s)| ds < +\infty,
$$
but since $\ln(|s|)$ is not bounded near $0$, one has 
$$
\underset{\eta > 0}{\sup} \int_\mathbb{R} |\breve{H}(\eta + i\omega)|^2d\omega = +\infty.
$$
Therefore $\| \frac{d \widehat{H}}{ds} \|_{\mathcal{B}_{\alpha}}$ is finite when $\alpha \in (-1, 1)$, but $\frac{d \widehat{H}}{ds} \notin \mathcal{B}_{-1}$. 
Similarly thanks to Lemata \ref{prop:derivI} and \ref{prop:derivJ} one can show that  $\frac{d^2 \widehat{I}}{ds^2}, \, \frac{d^3 \widehat{J}}{ds^3} \in \mathcal{B}_{\alpha}$ for $\alpha \in (-1, 1)$ but $\frac{d^2 \widehat{I}}{ds^2}, \, \frac{d^3 \widehat{J}}{ds^3}  \notin \mathcal{B}_{-1}$.
\end{proof}
 
\subsubsection{Consequences in temporal domain}
Thanks to the Paley-Wiener Theorem on Bergman spaces \ref{thm:Bergman}, we deduce from Proposition \ref{prop:dBergmann_nv_d} that in the nonviscous dispersive case ($\nu = 0$, $\kappa > 0$) and thanks to the Paley-Wiener Theorem on Bergman spaces \ref{thm:Bergman}, one has
$$
\forall \, k \in \{ 0, 1, 2 \}, \, \forall \, \alpha \in [-1, 1), \, t\delta^{(k)}(t) \notin L^2(t^{-(\alpha+1)}dt, \mathbb{R}^+).
$$
If $\nu > 0$ or $\kappa = 0$, one more time thanks to Proposition \ref{prop:dBergmann_v/nd} and thanks to Theorem \ref{thm:Bergman} one has
$$
\forall \, k \in \{ 0, 1, 2 \}, \, \forall \alpha \in (-1, 1), \, t^{k+1}\delta^{(k)} \in L^2(t^{-(\alpha +1)}dt, \mathbb{R}^+)
\quad \text{ but } \quad
t^{k+1}\delta^{(k)} \notin L^2(\mathbb{R}^+).
$$
those results constitute the first point of Theorems \ref{thm:pt5.1.1} and \ref{thm:pt5.1.2}. To obtain the second point of each theorem, one needs the following lemma. 
\begin{lemma}\label{prop:DecrL2}
Let $\gamma \in \mathbb{R}$ and $u \in L^2(dt, \mathbb{R}^+)$.
If $u \notin L^2( t^\gamma dt, [1, +\infty))$, then for all $\rho \in (0, 1/2)$, for all $T_0> 1$, and $C > 0$, there exists $t > T_0$ such that
$$
 |u(t)| \geq C t^{-\gamma/2 - 1/2 - \rho}.
$$
Furthermore if $u \in H^1( t^\gamma dt, \mathbb{R}^+ )$ then there exists $C > 0$ such that for all $t > 0$, 
$$
 |u(t)| \leq C t^{-\gamma/2}.
$$
\end{lemma}
\begin{proof}
Let $u \notin L^2( t^\gamma dt, [1, +\infty))$. By contradiction we suppose that there exists $\rho \in (0, 1/2)$, $T_0> 1$ and $C > 0$ such that for all $t > T_0$, 
$$
 |u(t)| \leq C t^{-\gamma/2 - 1/2 - \rho}.
$$
Therefore we deduce that
$$
 t^\gamma|u(t)|^2 \leq C t^{- 1 - \rho},
$$
which is in contradiction with the fact that $ u \notin L^2( t^\gamma dt, [1, +\infty) ) $. 
The second point is given by the Sobolev's embeddings of $H^1(\mathbb{R}^+)$ into $L^\infty(\mathbb{R}^+)$.
\end{proof}
 
We obtain the end of Theorem \ref{thm:pt5.1.1} by noting that if $\nu = 0$ and $\kappa > 0$, since $\delta \in H^2(\mathbb{R}^+)$, then $t\delta^{(k)} \in L^2((0, 1))$ and in particular 
$$
\forall \, k \in \{ 0, 1, 2 \}, \, \forall \, \alpha \in [-1, 1), \, t\delta^{(k)}(t) \notin L^2(t^{-(\alpha+1)}dt, [1, +\infty)).
$$
Therefore with Proposition \ref{prop:DecrL2} we deduce the second point of Theorem \ref{thm:pt5.1.1}.
 
Similarly, if $\nu > 0$ or $\kappa = 0$, then $t^{k+1}\delta^{(k)} \notin L^2([1, +\infty))$. With Proposition \ref{prop:DecrL2} we deduce the second point of Theorem \ref{thm:pt5.1.2}.
 
\subsection{The 1D case}\label{sec:1D}
In 1D, one has $B = 1$ and the following result stands.
\begin{proposition}
Let $(\kappa, \nu) \neq (0,0)$. Let $\eta_0$ defined by 
$$
\eta_0 = \left\{\begin{matrix}
 \frac{\nu}{2\kappa^2} \quad \text{if } \kappa > 0, \\
 \frac{1}{\nu} \quad \text{if } \kappa = 0.
\end{matrix}\right.
$$
Then one has 
$$
\delta \in H^2(e^{\eta_0t}dt, \mathbb{R}^+),
$$
but for all $k \in \{0, 1, 2\}$, for all $\beta \in (0, 2]$
$$
\delta^{(k)} \notin L^2(e^{\eta_0t}t^\beta dt, \mathbb{R}^+).
$$
Also for all $\rho \in (0, 1/2)$, for all $T_0> 1$, and $C > 0$, there exists $t > T_0$ such that
$$
|\delta^{(k)}(t)| \geq C e^{-\frac{\eta_0}{2} t}t^{-\frac{1}{2}-\rho}.
$$
\end{proposition}
\begin{proof}
If $(\kappa, \nu) \neq (0, 0)$, the square root $\sqrt{1+\nu s+\kappa^2 s^2}$ that appears in $\widehat{H}$ is well-defined for any complex number $s$ such that $\Re(s) > - \eta_0$.
Consequently similarly to Proposition \ref{prop:H}, one has that $\delta \in H^2(e^{\eta_0t}dt, \mathbb{R}^+)$.
However one can write $\frac{d\widehat{H}}{ds}$ as in Proposition \ref{prop:derivH} and the square root appears as $\frac{1}{\sqrt{1+\nu s + \kappa^2s^2}}$.
Therefore, as in Proposition \ref{prop:dBergmann_nv_d} one can show that
$$
\forall \alpha \in [-1, 1), \frac{d\widehat{H}}{ds},\frac{d\widehat{I}}{ds}, \, \frac{d\widehat{J}}{ds} \notin \mathcal{B}_\alpha(\mathbb{C}_{-\eta_0}).
$$
Thus according the Paley-Wiener Theorem \ref{thm:Bergman} one has that for $k \in \{0, 1, 2 \}$, for all $\alpha \in [-1, 1)$, $t\delta^{(k)} \notin L^2\left (e^{\eta_0 t}t^{-(\alpha+1)}dt, \mathbb{R}^+ \right )$ and thanks to Proposition \ref{prop:DecrL2} we can conclude the proof of the proposition. 
\end{proof}
As previously said this result is better than the one obtained in \cite{BeckLannes} since we showed that the strongest decay is in $\mathcal{O}(t^{-1/2})$ when $\nu = 0$. Moreover the viscosity allows an exponential decreasing of $\delta$ at infinity. Furthermore, if $\nu = \kappa = 0$, then the decay also is exponential since the vertical motion is given by
$$
\forall t > 0, \, \delta(t) = \delta_0 \cos(\omega t)e^{-\frac{Rt}{2\tau_\kappa^2(0)}},
$$
where $\omega = \frac{\sqrt{4\tau_\kappa^2(0) - R^2}}{2\tau_\kappa^2(0)}$.

\subsection{Consequences on $q$ and $\zeta$}

The point of this subsection is to explore the consequences of Theorems \ref{thm:pt5.1.0}, \ref{thm:pt5.1.1} and \ref{thm:pt5.1.2}. We recall the expression of $q$ and $\zeta$ given by Proposition \ref{prop:zeta_q_Lap}:
\begin{equation}
\begin{cases}
\widehat{q}(r, s) = -\frac{R \widehat{\dot{\delta}}(s)}{2} K_1\left(\frac{rs}{\sqrt{1+\nu s + \kappa^2 s^2}}\right)/K_1\left(\frac{Rs}{\sqrt{1+\nu s + \kappa^2 s^2}}\right),\\
\widehat{\zeta}(r, s) =- \frac{R}{2} \frac{\widehat{\dot{\delta}}(s)}{\sqrt{1+\nu s + \kappa^2 s^2}} K_0\left(\frac{rs}{\sqrt{1+\nu s + \kappa^2 s^2}}\right)/K_1\left(\frac{Rs}{\sqrt{1+\nu s + \kappa^2 s^2}}\right).
\end{cases}
\end{equation}
In particular, for any $s \in \mathbb{C}_0$, for any $r > R$, 
$$
|r \widehat{q}(r, s)| \leq \frac{R}{2} \left |\widehat{\dot{\delta}}(s)\right | \left | \frac{rK_1\left(\frac{rs}{\sqrt{1+\nu s + \kappa^2 s^2}}\right)}{K_1\left(\frac{Rs}{\sqrt{1+\nu s + \kappa^2 s^2}}\right)} \right |.
$$
Since $|K_1(z)|$ decreases exponentionaly at infinity, one deduces that 
$$
\frac{K_1\left(\frac{rs}{\sqrt{1+\nu s + \kappa^2 s^2}}\right)}{K_1\left(\frac{Rs}{\sqrt{1+\nu s + \kappa^2 s^2}}\right)} \in L^\infty_r,
$$
where we remind that $ L^\infty_r = L^\infty(rdr, [R, +\infty))$, and
\begin{itemize}
\item if $\nu = 0$ and $\kappa > 0$,
$$
q \in L^\infty_r H^1_t(\mathbb{R}^+),
$$
\item if $\nu > 0$ or $\kappa = 0$,
$$
q \in L^\infty_r H^1_t(t^{2 + \beta}dt, \mathbb{R}^+), \quad \text{ for } \quad \beta \in (0, 2).
$$
\end{itemize}
Similarly, 
for any $s \in \mathbb{C}_0$, for any $r > R$, 
$$
|r \widehat{\zeta}(r, s)| \leq \frac{R}{2} \left |\widehat{\dot{\delta}}(s)\right | \left | \frac{rK_0\left(\frac{rs}{\sqrt{1+\nu s + \kappa^2 s^2}}\right)}{\sqrt{1+\nu s + \kappa^2 s^2}K_1\left(\frac{Rs}{\sqrt{1+\nu s + \kappa^2 s^2}}\right)} \right |,
$$
and
$$
 \frac{K_0\left(\frac{rs}{\sqrt{1+\nu s + \kappa^2 s^2}}\right)}{\sqrt{1+\nu s + \kappa^2 s^2}K_1\left(\frac{Rs}{\sqrt{1+\nu s + \kappa^2 s^2}}\right)} \in L^\infty_r.
$$
It also leads to 
\begin{itemize}
\item if $\nu = 0$ and $\kappa > 0$,
$$
\zeta \in L^\infty_r H^1_t(\mathbb{R}^+),
$$
\item if $\nu > 0$ or $\kappa = 0$,
$$
\zeta \in L^\infty_r H^1_t(t^{2 + \beta}dt, \mathbb{R}^+), \quad \text{ for } \quad \beta \in (0, 2).
$$
\end{itemize}
By the equation \eqref{BAr:eq1}, one has $\partial_t \zeta = - {\rm d}_r q$ and therefore:
\begin{itemize}
\item if $\nu = 0$ and $\kappa > 0$,
$$
q \in L^\infty_r H^1_t(\mathbb{R}^+) 
\quad \text{and} \quad
{\rm d}_r q \in L^\infty_r L^2_t(\mathbb{R}^+),
$$
\item if $\nu > 0$ or $\kappa = 0$,
$$
q \in L^\infty_r H^1_t(t^{2 + \beta}dt, \mathbb{R}^+)
\quad \text{and} \quad
{\rm d}_r q \in L^\infty_r L^2_t(t^{2 + \beta}dt, \mathbb{R}^+), \quad \text{ for } \quad  \beta \in (0, 2).
$$
\end{itemize}

\section*{Perspectives}


In the full space without floating body, smooth solutions of the Boussinesq-Abbott equations converge to smooth solutions of the nonlinear shallow water equations when the dispersion parameter $\kappa$ vanishes, provided that no wave breaking occurs on the limit equations. The same problem in our situation is more difficult. Firstly the lifespan we showed vanishes when $\kappa \to 0$. Moreover the ODE satisfied by the trace of the surface elevation at the interface between the exterior and interior domains becomes stiff. To overcome those difficulties, as it was explained at the beginning of Section \ref{sec:WP}, one can use a quasi-linear scheme. As it was suggested in \cite{BeckLannes}, it probably provides energy estimates independent on $\kappa$ if the data satisfy compatibility conditions. We believe, as in \cite{BLM}, that these compatibility conditions will be expressed in the form of inequalities that depend on $\kappa$ and converge towards the compatibility conditions of the hyperbolic case, which are written in the form of equalities. The dispersionless limit may be valid if and only if the data satisfy the hyperbolic compatibility conditions. \\
\indent
The most natural continuation of our work consists of the numerical study of the floating cylinder situation under consideration. The presence of a boundary can completely change the scheme if stability is to be ensured. Furthermore, classical finite volume schemes do not necessarily provide access to all variables at the boundary. For example, in the waves-structure case, we know that the discharge at the boundary is given by the displacement of the solid, but the surface elevation at the boundary is not imposed. In the hyperbolic case, it is well-known that one can use the Riemann invariant linked to negative eigenvalue to compute the missing boundary condition (see \cite{Beck-Martaud-hyp}). In the dispersive case, one can use the hidden ODE on the trace elevation as it was done in \cite{BeckLannesWeynans, Rigal, LannesWeynans}. In \cite{Beck-Martaud-disp}, the authors build some finite volume schemes based on the Augmented formula with nonlocal numerical flux and use the hidden equations in the linear case and show the discrete stability. One has to adapt this strategy to our 2D situation where some radial terms cannot be written as flux term. The full waves-structure interaction problem can also be written as an ODE (see Subsection \ref{sec:estimODE}). We can therefore propose a numerical scheme based on this formulation. This will consist of solving time ODEs on each cell of the spatial mesh. \\
\indent
The next step of the study is to enrich our model by considering an elastic floating body rather than a flat one. The surface elevation is always constrained to the bottom of the floating object, but since the floating object is elastic, then the surface elevation is not necessary constant with respect to the space. More precisely the surface elevation in the interior domain is not only related to the displacement of the center of mass of the floating object that solves the Newton's equation, but also to the deflection of the elastic body. Thus formula \eqref{QintBefore} cannot be applied to compute the discharge in the interior domain.  In a first step the elastic body can be modeled by a beam and then the deflection solves the Euler-Bernoulli beam equation. In \cite{Alazard-Beam} the case of an elastic material that fits the whole free surface in deep water regime was considered. We want to study the shallow water situation with a localized floating body. Secondly, we would like to derive a wave-beam interaction model via an asymptotic analysis of the free surface Euler equation coupling with a thin floating elastic plate with several small parameters including the shallowness and the thinness. \\
\indent
In this study, we assume that the fluid is irrotational. Numerical simulations of the free-surface Navier-Stokes equations with non-flat bathymetry and very high Reynolds number performed in \cite{Riquier_Dormy_2024_irrot} exhibit some eddies. It would be interesting to study fluid-floating structure interaction with a wave model capable of generating enstrophy, such as the model given in \cite{Kazakova-Richard}. Can enstrophy be generated due to the presence of the floating solid even if the fluid is initially enstrophy free ? Could this phenomenon occur for all object geometries ?

\appendix
\section{Local balance of fluid's energy}	
\label{app:B}

This appendix is devoted the proof of the local energy balance of the fluid given in \eqref{eq:ConsOfTheEnergy}.

\begin{thm}
For all regular enough solution of Boussinesq's equation {\rm \hyperlink{BAr}{(BAr)}}, one has the following local energy budget:
The energy balance of the fluid is given by:
\begin{equation}
\partial_t \mathfrak e + {\rm{d}}_{r} \mathfrak f = \frac{P}{\varepsilon} {\rm{d}}_{r} q - \nu \frac{|\dr q|^2}{h}
+ \varepsilon (\kappa^2 \mathfrak{r} - \nu \dr  \mathfrak{r}_\nu),
\end{equation}
where
\begin{equation}
\begin{cases}
\mathfrak e:= \frac{1}{2} \left( \zeta^2 + \frac{q^2}{h} + \kappa^2 \frac{({\rm{d}}_{r} q)^2}{h} \right),\\
\mathfrak f:= q\left( \kappa^2 \ddot{\zeta} + \nu \dot{\zeta} + \zeta + \frac{\eps}{2}\frac{q^2}{h^2} + \frac{P}{\varepsilon} \right),
\end{cases}
\end{equation}
and
$$
\mathfrak{r}:= \frac{\zeta \ddot{\zeta}}{h} - \frac{({\rm{d}}_{r} q)^3}{2h^2} - q \ddot{\zeta} \frac{\partial_r \zeta}{h}
\quad \text{and} \quad 
\mathfrak{r}_\nu = \left(\frac{q \dot{\zeta}\zeta}{h}\right).
$$
\end{thm}

\begin{proof}
Let $(\zeta, q)$ be a regular enough solution of the radial Boussinesq-Abbott system {\rm \hyperlink{BAr}{(BAr)}}. Then if we multiply the Boussinesq system by $\zeta$ and $q$ one obtains
\begin{equation}
\begin{cases}
\partial_t \left(\frac{\zeta^2}{2}\right) + \dr(\zeta~q) = \partial_r \zeta q, \\
\frac{q}{h}\partial_t q - \frac{q}{h}\kappa^2 \partial_r \dr\left( \partial_t q \right) + \frac{q}{h}\varepsilon \dr\, \left( \frac{q^2}{h} \right)+ q \partial_r  \zeta = - q \partial_r \frac{P}{\varepsilon}  + \nu q \partial_r \left( \frac{1}{h} \,\dr\, q \right), \\
 h= 1 + \varepsilon \zeta.
\end{cases}
\quad \text{in}  \quad (R, \infty),
\end{equation}
Replacing $q \partial_r  \zeta$ in the second equation by its expression in the first one leads to
$$
\frac{q}{h}\partial_t q - \frac{q}{h}\kappa^2 \partial_r \dr\left( \partial_t q \right) + \frac{q}{h}\varepsilon \dr\, \left( \frac{q^2}{h} \right)+ \partial_t \left(\frac{\zeta^2}{2}\right) + \dr(\zeta	q) = - q \partial_r \frac{P}{\varepsilon}  + \nu q \partial_r \left( \frac{1}{h} \,\dr\, q \right).
$$
On one hand, by the expression of $h$ one has
\begin{align*}
\frac{q}{h}\partial_t q &= \partial_t \frac{|q|^2}{2h} + \eps \frac{|q|^2}{2h^2}\dot{\zeta}\\
&= \partial_t \frac{|q|^2}{2h} - \eps \frac{|q|^2}{2h^2}{\rm d}_r q.
\end{align*}
On the other hand one can obtain 
\begin{equation*}
\frac{q}{h}\varepsilon \dr\, \left( \frac{q^2}{h} \right) = \frac{\eps}{2} {\rm d}_r \frac{|q|^3}{h^2} + \frac{\eps}{2}\frac{|q|^2}{h^2}{\rm d}_r q,
\end{equation*}
and it leads to
\begin{equation}
\frac{q}{h}\partial_t q + \frac{q}{h}\varepsilon \dr\, \left( \frac{q^2}{h} \right) = \partial_t \frac{|q|^2}{2h} + \frac{\eps}{2} {\rm d}_r \frac{|q|^3}{h^2}.
\end{equation}
Furthermore one has
\begin{equation*}
-\frac{q}{h}\kappa^2 \partial_r {\rm d}_r \partial_t q = \kappa^2 {\rm d}_r \left( \frac{q\ddot{\zeta}}{h}\right) + \kappa^2 \partial_t \frac{|{\rm d}_r q|^2}{2h} + \eps \kappa^2\left( \frac{|{\rm d}_r q|^3}{2h^2} - \frac{\nu}{h}\dr q^2 + q\ddot{\zeta}\frac{\partial_r \zeta}{h^2} \right).
\end{equation*} 
The right hand term is simply treated by the duality between $\partial_r$ and ${\rm d}_r$ to obtain
$$
q \partial_r \frac{P}{\varepsilon} = {\rm d}_r\left(q \frac{P}{\varepsilon}\right) - (\dr q) \frac{P}{\varepsilon},
$$
and
\begin{align*}
 q \partial_r \left( \frac{1}{h} \,\dr\, q \right) & = {\rm d}_r \left(q \frac{1}{h} \,\dr\, q\right) - \frac{1}{h}(\dr q )^2\\
 & = {\rm d}_r \left(q \dr q\right) - \eps {\rm d}_r \left(q \frac{\zeta}{h} \,\dr\, q\right) - \frac{1}{h}(\dr q )^2.
\end{align*}
Thus we obtain
\begin{align*}
&\partial_t \frac{|q|^2}{2h} + \frac{\eps}{2} {\rm d}_r \frac{|q|^3}{h^2} + \kappa^2 {\rm d}_r \left( \frac{q\ddot{\zeta}}{h}\right) + \kappa^2 \partial_t \frac{|{\rm d}_r q|^2}{2h} + \eps \kappa^2\left( \frac{|{\rm d}_r q|^3}{2h^2} + q\ddot{\zeta}\frac{\partial_r \zeta}{h^2} \right) +\partial_t \left(\frac{\zeta^2}{2}\right) \\
& + \dr(\zeta	q) = -{\rm d}_r\left(q \frac{P}{\varepsilon}\right) + (\dr q) \frac{P}{\varepsilon} +\nu {\rm d}_r \left(q \dr q\right) - \nu \eps {\rm d}_r \left(q \frac{\zeta}{h} \,\dr\, q\right) - \nu\frac{1}{h}(\dr q )^2,
\end{align*}
which reads better as:
\begin{align*}
\partial_t \left( \frac{ q^2 }{2h} + \kappa^2\frac{\dr q^2}{2h} + \frac{\zeta^2}{2}\right)& + \dr \left(\kappa^2 q \frac{\ddot{\zeta}}{h} + \frac{\eps}{2}\frac{q^2}{h^2}  + \nu q\frac{\dot\zeta}{h} +q ~\frac{P}{\varepsilon} + \zeta q \right)\\
& + \varepsilon \kappa^2 \left( \frac{(\dr q)^3}{2h^2}~  + q ~ \ddot{\zeta}~ \frac{\partial_r \zeta}{h^2} \right) = (\dr q) \frac{P}{\varepsilon} - \frac{\nu}{h}(\dr q )^2.
\end{align*}
This balance is the same as in \cite{BeckLannes}, but we decided to modify it by observing that $\frac{1}{h} = 1 -\eps \frac{1}{h}$ to put the nonlinear part of terms $\kappa^2 \frac{\ddot{\zeta}}{h}$ and $\nu \frac{\dot\zeta}{h}$ into the remaining term. It leads to the energy balance we aimed at getting.
\end{proof}
 
\section{Rewrite of the SHODE system}\label{app:LinAddM}
In this section we write the ODEs system 
\begin{equation}\label{eq:ODEtrace_app}
\begin{cases}
\tau_\kappa^2(\varepsilon \delta)\ddot{\delta} + \delta  + \nu \dot{\delta}-  \varepsilon a( \delta, \underline{\zeta}) \, \dot{\delta}^2= \kappa^2 \underline{\ddot{\zeta}} + \nu\underline{\dot{\zeta}} + \underline{\zeta} + F_{\rm ext}, \\
\kappa^2 \underline{\ddot{\zeta}_e} + \nu \underline{\dot\zeta} +\underline{\zeta} + \eps \underline{\mathfrak{f}} = \underline{f_{\rm{hyd}}} - \kappa \ddot{\delta} G(R).
\end{cases}
\end{equation}
in the SHODE variable $Z = (\delta, \dot{\delta}, \underline{\zeta}, \underline{\dot{\zeta}})$ under the form
\begin{equation}\label{eq:goal}
\mathcal{M} \partial_t Z + \mathcal{T} Z = P\left(\underline{f_{\rm hyd}}, F_{\rm ext}\right) + \varepsilon \tilde{P}\left(Z, \underline{f_{\rm hyd}}, F_{\rm ext} \right),
\end{equation}
with $\mathcal{M}$, $\mathcal{T}$ real matrix, $P$ linear and $\tilde{P}$ nonlinear, are all defined in \eqref{eq:ODEOper}. 
Firstly one can remark that with the expression of $\tau_\kappa^2$ given in \ref{AddedMass} we have
\begin{equation}\label{eq:AML1_app}
\tau^2_\kappa(0) - \tau^2_\kappa(\varepsilon \delta) = \varepsilon \frac{R^2}{8h_{\rm i}}\delta.
\end{equation}
Therefore we can rephrase the first equation of \eqref{eq:ODEtrace_app} as
\begin{equation}\label{eq:AddMEq2_app}
\tau^2_\kappa (0)  \ddot{\delta} + \varepsilon \frac{R^2}{8h_{\rm i}}\delta \ddot{\delta} + \delta - \eps a(\delta, \underline{\zeta}) \dot{\delta}^2 = \kappa^2 \underline{\ddot{\zeta}} + \underline{\zeta} + F_{\rm ext}.
\end{equation}
\indent In this equation nonlinear terms are exposed in front of $\eps$, but we still have to eliminate nonlinearities over $\ddot{\delta}$. To do so we mix both ODEs of system \eqref{eq:ODEtrace_app}. Indeed by replacing $\kappa^2 \underline{\ddot{\zeta}} + \nu\underline{\dot{\zeta}} + \underline{\zeta}$ by its expression in the second equation we obtain 

\begin{equation}\label{eq:AML2_app}
\tau_\kappa^2(\varepsilon \delta)\ddot{\delta} +\nu\dot{\delta} + \delta - \eps a(\delta, \underline{\zeta}) \dot{\delta}^2 = \underline{f_{\rm{hyd}}} - \kappa \ddot{\delta} G(R) - \eps \underline{\tilde{\mathfrak{f}}} + F_{\rm ext}.
\end{equation}

Consequently injecting \eqref{eq:AML2_app} into \eqref{eq:AddMEq2_app} brings us to write the modified Newton's equation as
\begin{equation}\label{eq:AddMassEq_app}
\tau^2_\kappa (0)  \ddot{\delta} +\nu\dot{\delta} + \delta = \kappa^2 \underline{\ddot{\zeta}} + \nu\underline{\dot{\zeta}} + \underline{\zeta} + \eps \left(\gamma_{loc}(\delta, \dot{\delta}, \underline{\zeta})+ a(\delta, \underline{\zeta}) \dot{\delta}^2\right) + \eps \gamma_{nloc}(\underline{f_{\rm{hyd}}}, \delta) + \eps \tilde{\gamma}(\delta)F_{\rm ext} + F_{\rm ext},
\end{equation}
with $\gamma_{loc}(\delta, \dot{\delta}, \underline{\zeta}) = \frac{R^2}{8h_{\rm i} \left( \tau^2_\kappa (\varepsilon \delta) + \kappa G(R)\right) }\delta \left( \eps a(\delta, \underline{\zeta})\dot{\delta}^2 - \eps \underline{\mathfrak f}- \delta - \nu\dot\delta\right) $,
$ \gamma_{nloc}(f_{\rm{hyd}}, \delta) = \frac{R^2}{8h_{\rm i} \left( \tau^2_\kappa (\varepsilon \delta) + \kappa G(R)\right) }\delta \underline{f_{\rm{hyd}}} $ and $\tilde{\gamma}(\delta) = \frac{R^2}{8h_{\rm i} \left( \tau^2_\kappa (\varepsilon \delta) + \kappa G(R)\right) }\delta $. $\gamma_{loc}$ only contains local terms. In the opposite $\gamma_{nloc}$ is nonlocal through the dependence on $f_{\rm{hyd}}$. We want to underline the fact that from an additive force we obtained a multiplicative force when the force applied to the solid interacts with the force exerced by the water under the solid.\\
\indent For sake of clarity we note $\gamma(\delta, \dot{\delta}, \underline{\zeta}, \underline{\dot{\zeta}}, f_{\rm{hyd}}) = \gamma_{loc}(\delta, \dot{\delta}, \underline{\zeta}) + \gamma_{nloc}(f_{\rm{hyd}}, \delta) + \tilde{\gamma}(\delta)F_{\rm ext}$.
The dependence on $(\zeta, q)$ with the intermediate variable $f_{\rm{hyd}}$ is remarked since it is made of nonlocal components. 
We have now the system of ODEs
\begin{equation}
\begin{cases}
\left( \tau^2_\kappa (0) + \kappa G(R) \right) \ddot{\delta} + \nu \dot{\delta}+ \delta = \underline{f_{\rm hyd}} - \eps\underline{{\mathfrak{f}}} + \eps \gamma  + F_{\rm ext},\\
\left( 1 + \frac{\kappa G(R)}{\tau_\kappa^2(0)}\right)\left(\kappa^2 \underline{\ddot{\zeta}} +\nu \underline{\zeta} + \underline{\zeta} + \eps\underline{{\mathfrak{f}}}\right) = \underline{f_{\rm{hyd}}} + \frac{\kappa G(R)}{\tau_\kappa^2(0)}\left( \nu \dot{\delta} + \delta - F_{\rm ext} - \eps \gamma\right).
\end{cases}
\end{equation}
We deduce easily the form \eqref{eq:goal}.
 
\section{Domain of a the transfer function}\label{app:Proof5.5}

Proposition \ref{prop:C0DKV} is a corollary of the two following ones. In a first time we will discuss about branching points and branching cuts linked to $1+\nu s + \kappa^2 s^2 \notin \mathbb{R}^-$, and in a second time about branching points and branching cuts linked to $\frac{s}{\sqrt{1+\nu s + \kappa^2 s^2}} \notin \mathbb{R}^-$.
\begin{proposition}
The fonction $\sqrt{1+\nu s + \kappa^2 s^2}$ admits branching points $s_{\rm b}$ and branching cuts $\Gamma_{\kappa,\nu}$ where  $s_{\rm b}$ and $\Gamma_{\kappa,\nu}$ are given in Table \ref{tab:BPBC}.
\end{proposition}
\begin{proof}
In the case $\kappa = 0$ the branching cut is given by the line of equation $s = \eta$, $\eta < -\nu^{-1}$ and the branching point is $s = -\nu^{-1}$. 
$\nu = 0$ is the same as in the 1D case \cite{BeckLannes} and it provides the branching cuts $s = i\omega$ with $|\omega| > \kappa^{-1}$. The associated branching points are $s = \pm i \kappa^{-1}$.
We do not consider the case $\kappa = \nu = 0$ here because if it is so the expressions of $\widehat{\zeta}$ and $\widehat q$ do not depend on the symbol $\sqrt{1+\nu s + \kappa^2 s^2}$.
If $\kappa \neq 0$ and $\nu \neq 0$ the zeros of the polynomial $1+\nu s + \kappa^2 s^2$ give the branching points. Also the definition of the complex square root gives:
\begin{equation}
\begin{cases}
1 + \nu \eta + \kappa^2(\eta^2 - \omega^2) \leq 0,\\
\omega(\nu + 2\kappa^2 \eta) = 0.
\end{cases}
\end{equation}
For $\omega = 0$, the solutions are $\eta \in [\frac{-\nu - \sqrt{\nu^2 - 4\kappa^2}}{2\kappa^2}; \frac{-\nu + \sqrt{\nu^2 - 4\kappa^2}}{2\kappa^2}]$, only if $\nu \geq 2\kappa^2$. For $\omega \neq 0$, it leads to 
$$ 
\eta = -\frac{\nu}{2\kappa^2}
\quad
\text{and}
\quad
\omega^2 \geq \frac{1}{4\kappa^4}\left(4\kappa^2-\nu^2\right).
$$

In the case where $\nu > 2\kappa$ the branching cut is the whole line of equation $s = -\frac{\nu}{2\kappa^2} + i \omega$ for $\omega \in \mathbb{R}$. If $\nu = 2\kappa$ then one just has to remove the point $\omega =0$. In the other case the branching cuts are the lines $ \left\{ -\frac{\nu}{2\kappa^2} + i \omega, \omega > \frac{1}{2\kappa^2} \sqrt{4\kappa^2 - \nu^2} \right\} $ and $\left\{ -\frac{\nu}{2\kappa^2} - i \omega, \omega > \frac{1}{2\kappa^2} \sqrt{4\kappa^2 - \nu^2} \right\}$ (see Figure \ref{fig:branch_cuts}). 
 \end{proof}
 
 \begin{proposition}
 The branching cut associated to the modified Bessel function $K_{\alpha}$ is 
 $$\{ s \in \mathbb{C} \setminus \Gamma_{\kappa,\nu}  \, | \, \frac{s}{\sqrt{1+\nu s + \kappa^2 s^2}} \notin \mathbb{R}_{-}^{\ast} \} = \mathbb{R}_{-}^{\ast}.$$
 \end{proposition}
  \begin{proof}
For any $s \in \mathbb{C} / \Gamma_{\kappa,\nu}$, let us denote $p(s) := \frac{s}{\sqrt{1+\nu s + \kappa^2 s^2}}$, $R = \Re(\sqrt{1+\nu s + \kappa^2 s^2}) \geq 0$ and $I = \Im(\sqrt{1+\nu s + \kappa^2 s^2})$. One has $p(s) \in \mathbb{R}^-$ if and only if
\begin{equation}
\left\{ \begin{matrix}
\Re\left ( \frac{s}{\sqrt{1+\nu s + \kappa^2 s^2}}\right) \leq 0, \\
\Im\left (\frac{s}{\sqrt{1+\nu s + \kappa^2 s^2}}\right ) = 0,
\end{matrix}\right.
\end{equation}
which also reads:
\begin{empheq}[left=\empheqlbrace]{align}
&\eta R + \omega I \leq 0, \label{cond:p1}\\
&\omega R - \eta I = 0, \label{cond:p2}
\end{empheq}
where $\eta:=\Re(s)$ and $\omega:= \Im(s)$.

Firstly if $\eta > 0$ then \eqref{cond:p2} shows that $I$ and $\omega$ must have the same symbol whereas \eqref{cond:p1} shows that $I$ and $\omega$ must have the opposite symbol. Thus $\omega R= \eta I = 0$.
Since $\eta > 0$ we obtain that $R = 0$ and then $\omega = 0$. Then $p(s) > 0$. Hence for $\Re(s) > 0$, $\frac{s}{\sqrt{1+\nu s + \kappa^2 s^2}} \notin \mathbb{R}^-$.

If $\eta = 0$ then \eqref{cond:p2} shows that $\omega = 0$ or $R = 0$. If $\omega = 0$, then $p(s) = 0$. One deduces that $0 \in \left \{ s \in \mathbb{C}, p(s) \in \mathbb{R}^-\right \}$. If $R = 0$, then $(iI)^2 = -I^2$ is real. However by definition 
$$
(iI)^2 = \sqrt{1+i \nu \omega + \kappa^2 \omega^2}^2 = 1+i \nu \omega - \kappa^2 \omega^2.
$$  
Therefore in the case $\nu > 0$ it implies that $\omega = 0$ and we already know this case. 
If $\nu = 0$, then $p(s) = \frac{i\omega}{\sqrt{1-\kappa^2 \omega^2}}$. 
If $|\omega| < \frac{1}{\kappa}$, $p(s)$ is complex. 
Else by the definition of the square root we have 
$$
p(s) = \frac{i\omega}{i \text{sgn}(\omega) \sqrt{\kappa^2 \omega^2 - 1}} = \frac{|\omega|}{\sqrt{\kappa^2 \omega^2 - 1}} > 0.
$$
Hence, in the case where $\eta = 0$, one has $p(s) \in \mathbb{R}^-$ if and only if $\omega = 0$.

If $\eta < 0$ then \eqref{cond:p2} implies that $\omega$ and $I$ have opposite sign. In that case \eqref{cond:p1} is always satisfied.  In particular, if $\omega = 0$ then $I = 0$ and therefore $p(s) \in \mathbb{R}^{\ast}_{-}$. If $\omega \neq 0$, then the complex numbers $s$ such that $p(s) \in \mathbb{R}^{\ast}_{-}$ are included in the set of solutions of the polynomial equation
\begin{equation}\label{eq-polyn}
\omega^2 (|1+\nu s + \kappa^2 s^2| +  \Re(1+\nu s + \kappa^2 s^2))= - \eta^2 (|1+\nu s + \kappa^2 s^2| - \Re(1+\nu s + \kappa^2 s^2)).
\end{equation}
Since the modulus is higher than the real part, the solutions of \eqref{eq-polyn} must satisfies 
\begin{equation}\label{eq-polyn-2}
\omega^2 + \eta^2 \leq | \omega^2 - \eta^2 |.
\end{equation}
On one hand, if $\omega^2 > \eta^2 >0$, then \eqref{eq-polyn-2} implies that $2 \eta^2 \leq 0$ which is incompatible with $\eta < 0$. On the other hand, if $0 < \omega^2 < \eta^2$, then \eqref{eq-polyn-2} implies that $2 \omega^2 \leq 0$ which is incompatible with $\omega = 0$. If $\omega^2 = \eta^2>0$ then the solution of \eqref{eq-polyn} must satisfies $(1 + \nu \eta)^2 + \omega^2 (\nu + 2 \kappa^2 \eta)^2 =0$ which admits no solutions. Thus \eqref{eq-polyn} have no solution for $\eta <0$ and $\omega \neq 0$.
\end{proof}

\section{Laplace transform and Bergman spaces}\label{app-Laplace}

The set of holomorphic function on $\mathbb{C}_{0}$ is denoted by ${\mbox{Hol }}(\mathbb{C}_{0})$,
the Bergman spaces by
\[
\mathcal{B}_{\alpha}=\Big\{F\in{\mbox{Hol }}(\mathbb{C}_{0}):\|F\|_{\mathcal{B}_{\alpha}}^2:=\int_{\eta>0}\int_{\omega\in\mathbb R} |F(\eta+i\omega)|^2 \,\eta^{\alpha}\,d\omega\,d\eta<\infty\Big\},
\]
for any $\alpha > -1$ and the Hardy space by
\[
\mathcal{B}_{-1}=\Big\{F\in{\mbox{Hol }}(\mathbb{C}_{0}):\|F\|_{\mathcal{B}_{-1}}^2:= \sup_{\eta>0}\int_{\omega\in\mathbb R} |F(\eta+i\omega)|^2 \,\eta^{\alpha}\,d\omega <\infty\Big\}.
\]
It is well known that Bergman and Hardy spaces are reproducing kernel Hilbert spaces (see \cite{hedenmalm1999bergman}).
Laplace transform with respect to the time variable of locally integrable function $u$ is defined by
$$
 \hat{u}(s) := \mathcal{L} [u] (s) := \displaystyle \int_0^{\infty} u(t) e^{-st} {\rm d}t,
$$
for all complex number $s = \eta + i \omega$ such that this integral converges absolutely. The goal of the following theorem is to describe asymptotic behavior of a time domain function from its Laplace transform. 
\begin{thm}[Paley-Wiener Theorem for Bergman spaces] 
Let $\alpha \geq -1$.
    The Laplace transform $\mathcal{L}:L^2\big(t^{-(\alpha+1)}\,dt\big) \to \mathcal{B}_{\alpha} $ is a linear isomorphism between Hilbert spaces. Moreover for $\alpha > -1$, one has for all $u \in L^2\big(t^{-(\alpha+1)}\,dt\big)$ the Plancherel type equality
    $$
    \|\hat{u}\|_{\mathcal{B}_{\alpha}}^2 = 2\pi\frac{\Gamma(\alpha+1)}{2^{\alpha+1}} 
    \int_{\mathbb{R}^+} | u(t)|^2 t^{-(\alpha+1)}\,dt,
    $$
    and the Laplace transform $\mathcal{L}$ is an isometry for $\alpha=-1$.
\end{thm}
 
\begin{proof}
For $\alpha=-1$ the theorem is known as Paley-Wiener Theorem. We show its generalisation for any $\alpha > -1$. \\
 
{\it{Step 1 :}} Continuity of $\mathcal{L}:L^2\big(t^{-(\alpha+1)}\,dt\big) \to \mathcal{B}_{\alpha} $ \\
 
 Let $u$ a regular function with compact support. For any $\eta>0$, $\hat{u}(\eta+ i \cdot)$ is the Fourier transform of $e^{-\eta t}u(t)$. Then by the Plancherel theorem one has 
\[
\int_{-\infty}^{\infty}|\widehat{u}(\eta+i\omega)|^2\,d\omega
=2\pi\int_0^\infty e^{-2\eta t}|u(t)|^2\,dt.
\]
Multiply both sides by $\eta^{\alpha}\ge 0$ and integrate in $\eta\in(0,\infty)$. As every integrand is nonnegative one can apply Fubini-Tonelli to exchange the order of integration:
\[
\int_{0}^{\infty} \left( \int_{-\infty}^{\infty} |\widehat{u}(\eta+i\omega)|^2\,\eta^{\alpha}\,d\omega \right)\,d\eta
=2\pi\int_0^\infty |u(t)|^2\Big(\int_0^\infty \eta^{\alpha} e^{-2\eta t}\,d\eta\Big)\,dt.
\]
The inner integral is the Gamma integral:
\[
\int_0^\infty \eta^{\alpha} e^{-2\eta t}\,d\eta
=\frac{\Gamma(\alpha+1)}{(2t)^{\alpha+1}}.
\]
Therefore, one has
\[
\|\widehat{u}\|_{\mathcal{B}_{\alpha}}^2
=2\pi\frac{\Gamma(\alpha+1)}{2^{\alpha+1}}\int_0^\infty |u(t)|^2 t^{-(\alpha+1)}\,dt.
\]
Since the space of regular functions with compact is dense in $L^2\big(t^{-(\alpha+1)}\,dt\big)$, then the previous equality is also true for any $u \in L^2\big(t^{-(\alpha+1)}\,dt\big)$. In particular the linear map $\mathcal{L}:L^2\big(t^{-(\alpha+1)}\,dt\big) \to \mathcal{B}_{\alpha}$ is bounded and injective. \\
 
{\it{Step 2 :}} Surjectivity of $\mathcal{L}:L^2\big(t^{-(\alpha+1)}\,dt\big) \to \mathcal{B}_{\alpha} $ \\
 
Let $ \hat{G} \in \mathcal{B}_{\alpha}$ such that
\[
\int_{\mathbb{C}_0} \widehat{G}(s)\overline{\widehat{u}(s)}\, \Re(s)^{\alpha}\,ds=0,
\]
for any $\hat{u} \in \mathcal{B}_{\alpha}$. If one takes in particular $(\widehat{u}_p)_{p \in \mathbb{C}_0}$ the reproducing kernel of the Bergman space $\mathcal{B}_{\alpha}$, one has by definition of reproducing kernel
\[
\forall p \in \mathbb{C}_0, \, \widehat{G}(p) =\int_{\mathbb{C}_0} \widehat{G}(s)\overline{\widehat{u}_p(s)}\, \Re(s)^{\alpha}\,ds=0.
\]
One has $\widehat{G}=0$ and thus the map $\mathcal{L}:L^2\big(t^{-(\alpha+1)}\,dt\big) \to \mathcal{B}_{\alpha} $ is surjective.
 
\end{proof}

\section*{Acknowledgement}
The first author is supported by the BOURGEONS project, grant ANR-23-CE40-0014-01 of the French National Research Agency (ANR). The second author is funded by the Simons Foundation Award ID: 651475 related to Simons Collaboration on Wave Turbulence.
 
 \section*{Conflict of Interest}

The authors declare that they have no conflict of interest.

\section*{Declaration statements like data availability}

No data were generated or analysed in this study.

\end{document}